\documentclass{article}%
\usepackage{amsmath}
\usepackage{graphicx}
\usepackage{amsfonts}
\usepackage{amssymb}%
\providecommand{\U}[1]{\protect\rule{.1in}{.1in}}
\newtheorem{theorem}{Theorem}

\newtheorem{lemma}{Lemma}[section]

\newtheorem{proposition}{Proposition}
\newtheorem{remark}{Remark}

\newenvironment{proof}[1][Proof]{\textbf{#1.} }{\ \rule{1em}{1em}}
\numberwithin{equation}{section}
\begin{document}

\title{ Instability of large solitary water waves}
\author{Zhiwu Lin\\Mathematics Department\\University of Missouri\\Columbia, MO 65211 USA}
\date{}
\maketitle

\begin{abstract}
We consider the linearized instability of 2D irrotational solitary
water waves. The maxima of energy and the travel speed of solitary
waves are not obtained at the highest wave, which has a 120 degree
angle at the crest. Under the assumption of non-existence of
secondary bifurcation which is confirmed numerically, we prove
linear instability of solitary waves which are higher than the wave
of maximal energy and lower than the wave of maximal travel speed.
It is also shown that there exist unstable solitary waves
approaching the highest wave. The unstable waves are of large
amplitude and therefore this type of instability can not be captured
by the approximate models derived under small amplitude assumptions.
For the proof, we introduce a family of nonlocal dispersion
operators to relate the linear instability problem with the elliptic
nature of solitary waves. A continuity argument with a moving kernel
formula is used to study these dispersion operators to yield the
instability criterion.

\end{abstract}

\section{Introduction}

\textbf{Preliminaries.}{\large \ \ }The water-wave problem in its simplest
form concerns two-dimensional motion of an incompressible inviscid liquid with
a free surface, acted on only by gravity. Suppose, for definiteness, that in
the $(x,y)$-Cartesian coordinates gravity acts in the negative $y$-direction
and that the liquid at time $t$ occupies the region bounded from above by the
free surface $y=\eta(t;x)$ and from below by the flat bottom $y=-h$, where
$h>0$ is the water depth. In the fluid region $\{(x,y):-h<y<\eta(t;x)\}$, the
velocity field $(u(t;x,y),v(t;x,y))$ satisfies the incompressibility
condition
\begin{equation}
\partial_{x}u+\partial_{y}v=0 \label{E:Euler1}%
\end{equation}
and the Euler equation
\begin{equation}%
\begin{cases}
\partial_{t}u+u\partial_{x}u+v\partial_{y}u=-\partial_{x}P\\
\partial_{t}v+u\partial_{x}v+v\partial_{y}v=-\partial_{y}P-g,
\end{cases}
\label{E:Euler2}%
\end{equation}
where $P(t;x,y)$ is the pressure and $g>0$ denotes the gravitational constant
of acceleration. The kinematic and dynamic boundary conditions at the free
surface $\{y=\eta(t;x)\}$
\begin{equation}
v=\partial_{t}\eta+u\partial_{x}\eta\quad\text{and}\quad P=P_{\text{atm}}
\label{E:top}%
\end{equation}
express, respectively, that the boundary moves with the velocity of the fluid
particles at the boundary and that the pressure at the surface equals the
constant atmospheric pressure $P_{\text{atm}}$. The impermeability condition
at the flat bottom states that
\begin{equation}
v=0\qquad\text{at}\quad\{y=-h\}. \label{E:bottom}%
\end{equation}
In this paper we consider the irrotational case with $\operatorname{curl}v=0$,
for which the Euler equation (\ref{E:Euler1})--(\ref{E:bottom}) is reduced to
the Bernoulli equation
\[
\frac{\partial\phi}{\partial t}+\frac{1}{2}\left\vert \nabla\phi\right\vert
^{2}+g\eta=c\left(  t\right)
\]
where $\phi$ is the vector potential such that $\left(  u,v\right)
=\nabla\phi$. The local well-posedness of the full water wave problem was
proved by Wu (\cite{wu}) for deep water and by Lannes for water of finite
depth (\cite{lannes}).

We consider a traveling solitary wave solution of (\ref{E:Euler1}%
)--(\ref{E:bottom}), that is, a solution for which the velocity field, the
wave profile and the pressure have space-time dependence $(x+ct,y)$, where
$c>0$ is the speed of wave propagation. With respect to a frame of reference
moving with the speed $c$, the wave profile appears to be stationary and the
flow is steady. It is traditional in the traveling-wave problem to define the
relative stream function $\psi(x,y)$ and vector potential $\phi\left(
x,y\right)  $ such that:
\begin{equation}
\psi_{x}=-v,\qquad\psi_{y}=u+c \label{D:stream}%
\end{equation}
and
\begin{equation}
\phi_{x}=u+c,\qquad\phi_{y}=v. \label{D:v-potential}%
\end{equation}
The boundary conditions at infinity are
\[
\left(  u,v\right)  \rightarrow\left(  0,0\right)  ,\text{ }\eta\left(
x\right)  \rightarrow0\text{, as }\left\vert x\right\vert \rightarrow+\infty.
\]
The solitary wave problem for (\ref{E:Euler1})--(\ref{E:bottom}) is then
reduced to an elliptic problem with the free boundary $\left\{  y=\eta
(x)\right\}  \ $(\cite{at81}):

Find $\eta(x)$ and $\psi(x,y)$, in $\{(x,y):-\infty<x<+\infty,\ -h<y<\eta
(x)\}$, such that
\begin{subequations}
\label{stream}%
\begin{align}
\Delta\psi &  =0\qquad\text{in}\quad-h<y<\eta(x){,}\\
\psi &  =0\quad\text{on}\quad y=\eta(x){,}\\
|\nabla\psi|^{2}+ &  2gy=c^{2}\qquad\text{on}\quad y=\eta(x){,}\\
\psi &  =-ch\quad\text{on}\quad y=-h,
\end{align}
with
\end{subequations}
\[
\nabla\psi\rightarrow\left(  0,c\right)  \text{, }\eta\left(  x\right)
\rightarrow0\text{, as }\left\vert x\right\vert \rightarrow+\infty.\text{ }%
\]
First we give a summary of the existence theory of solitary water waves.
Lavrentiev (\cite{lavrentiv}) got the first proof of the existence of small
solitary waves by studying the long wave limit. A direct construction of small
solitary waves was given by Fridrichs and Hyers (\cite{fre-heys}), and their
proof was readdressed by Beale (\cite{beale77}) via the Nash-Moser method. The
existence of large amplitude solitary waves was shown by Amick and Toland
(\cite{at81}). The highest wave was also shown to exist by Toland
(\cite{toland78}), and its angle at the crest was shown to be 120\ degree
(Stokes's Conjecture in 1880) by Amick, Toland and Fraenkel (\cite{atf}). The
symmetry of solitary waves was studied by Craig and Sternberg (\cite{craige88}%
). Plotnikov (\cite{plot91}) studied the secondary bifurcation and showed that
solitary waves are not unique for certain travelling speed. The particle
trajectory for solitary waves was studied by Constantin and Escher
(\cite{con-esher}). We list some properties of the solitary waves which will
be used in the study of their stability. Denote the Froude number by
\[
F=\frac{c}{\sqrt{gh}},
\]
and the Nekrasov parameter by
\[
\mu=\frac{6ghc}{\pi q_{c}^{3}},
\]
where $q_{c}$ is the (relative) speed of the flow at the crest. We note that
$\mu$ is the bifurcation parameter used in \cite{at81}. The highest wave
corresponds to $\mu=+\infty$ since $q_{c}=0.$The following properties of
solitary waves are proved:

(P1) (\cite{at81}) There exists a curve of solitary waves that are symmetric,
positive $\left(  \eta>0\right)  $ and monotonically decay on either side of
the crest, with the parameter $\mu\in\left(  \frac{6}{\pi},+\infty\right)  $.
When $\mu\nearrow+\infty$, the solitary waves tend to the highest wave with
the 120 degree angle at the crest. When $\mu\searrow\frac{6}{\pi},$ the
solitary waves tend to the small waves constructed in \cite{fre-heys} and
\cite{beale77}. Moreover, we have $\nabla\psi\rightarrow\left(  0,c\right)  $,
$\eta\left(  x\right)  \rightarrow0$ exponentially as $\left\vert x\right\vert
\rightarrow+\infty.\ $Below, we call this solitary wave curve the primary branch.

(P2) (\cite{starr}, \cite{at81}, \cite{mcleod-F}) Any positive and symmetric
solitary wave which decays monotonically on either side of its crest is
supercritical, that is, $F>1$ or equivalently $c>\sqrt{gh}$. The limit of
small waves corresponds to $F\searrow1$ (\cite{at81}, \cite{fre-heys}).

(P3) (\cite{craige88}) Any supercritical solitary wave $(F>1)\ $is symmetric,
positive and decays monotonically on either side of its crest. Moreover, any
nontrivial solitary wave curve connected to the primary branch must have $F>1$.

(P4) (\cite{plot91}) For small amplitudes waves with $\mu\approx\frac{6}{\pi}%
$, there is no secondary bifurcation on the primary branch. When the highest
wave is approached, that is, when $\mu\rightarrow+\infty$, there are
infinitely many points on the primary branch which are either a secondary
bifurcation point or a turning point where $c^{\prime}\left(  \mu\right)  =0.$

The property (P4) is essentially what was proved in \cite{plot91}, though our
statement above adapts the explanation in \cite[p. 245]{btd-2}. Moreover,
numerical evidences (\cite{saff-chen}, \cite{kataoka-soli}) indicate that the
following assumption holds true:

(H1) There are no secondary bifurcation points on the primary branch.

Under the assumption (H1), above property (P4) implies that there are
infinitely many turning points where $c^{\prime}\left(  \mu\right)  =0$. So
the travel speed $c\ $does not always increase with the wave amplitude, and
this differs greatly from KDV and other approximate models for which the
higher waves travel faster. More precisely, for full solitary water waves the
travel speed obtains its maximum before the highest wave and then it becomes
highly oscillatory near the highest wave. This fact was first observed from
numerical computations (\cite{bay-longuet}, \cite{long-fenton}), then
confirmed by the asymptotic analysis (\cite{long-fox1}, \cite{long-fox-soli}).
Indeed, almost all physical quantities (i.e. energy and momentum) do not
achieve their maxima at the highest wave, and are highly oscillatory around
the highest wave (see above references). This fact turns out to imply the
instability of large solitary waves, which was first discovered from numerical
computations (\cite{tanaka86}) and is rigorously proved in this paper.

\textbf{Main results.} \ Denote by $\mu_{1}$ the first turning point where
$c\left(  \mu\right)  $ obtains its global maximum and, by $\tilde{\mu}_{1}$
the first and also the global maximum point of $E\left(  \mu\right)  $, where
\begin{equation}
E\left(  \mu\right)  =\int_{\mathbf{R}}\int_{-h}^{\eta\left(  x\right)  }%
\frac{1}{2}\left(  u^{2}+v^{2}\right)  dydx+\int_{\mathbf{R}}\frac{1}{2}%
g\eta^{2}dx.\label{energy}%
\end{equation}
is the energy of the solitary wave with the parameter $\mu$. Numerical
computations (\cite{long-fenton}, \cite{tanaka86}, \cite{long-tanaka})
indicate that $\mu_{1}>\tilde{\mu}_{1}$, and $\tilde{\mu}_{1}$ is the only
critical point of $E\left(  \mu\right)  $ in $\left(  \frac{6}{\pi},\mu
_{1}\right)  $. We state it as another hypothesis:

(H2) The energy maximum is achieved on the primary branch before the wave of
the maximal travel speed (the first turning point).

\begin{theorem}
\label{thm-main} Under the assumptions (H1) and (H2), the solitary wave is
linearly unstable when $\mu\in\left(  \tilde{\mu}_{1},\mu_{1}\right)  $, where
$\mu_{1}$ and $\tilde{\mu}_{1}$ are the maxima points of the travel speed and
energy, respectively. The linearly instability is in the sense that there
exists a growing mode solution $e^{\lambda t}\left[  \eta\left(  x\right)
,\psi\left(  x,y\right)  \right]  $ ($\lambda>0$) to the linearized problem
(\ref{eqn-L}), where $\eta\left(  x\right)  ,\psi\left(  x,y\right)  \in
C^{\infty}\cap H^{k}$ for any $k>0.$
\end{theorem}

Our next theorem shows that there exist unstable solitary waves approaching
the highest wave.

\begin{theorem}
\label{thm-minor}Under the assumption (H1), there exists infinitely many
intervals $I_{i}\ \subset(\mu_{1},+\infty),\ \left(  i=1,2,\cdots\right)  $
with $\lim_{n\rightarrow\infty}\max\left\{  \mu|\ \mu\in I_{n}\right\}  $
$=+\infty$, such that solitary waves with the parameter $\mu\in I_{i}$ are
linearly unstable in the sense of Theorem \ref{thm-main}.
\end{theorem}

Theorem \ref{thm-minor} suggests that the highest wave
$(\mu=\infty)$ constructed in \cite{atf} is unstable. This contrasts
with the stability of peaked solitary waves in some shallow water
wave models (\cite{cs-peakon}, \cite{lin-dp}). Numerical evidences
(\cite{tanaka86}, \cite{long-tanaka}) suggest that solitary waves
are spectrally stable when $\mu\in\left(  \frac
{6}{\pi},\tilde{\mu}_{1}\right)  $, and linearly unstable when
$\mu>\tilde {\mu}_{1}$, at least before the first few turning points
where the computations are reliable. We note that the amplitude of
the maximal energy wave with the parameter $\tilde{\mu}_{1}$ is
already close to the maximal height (\cite{tanaka86}). So the
unstable waves proved in Theorems \ref{thm-main} and \ref{thm-minor}
are of large amplitude, and therefore this type of instability can
not appear in approximate models which are derived under the small
amplitude assumptions, such as KDV equation or Boussinesq systems.
Numerical evidences (\cite{tanaka-dold-pere87}) also suggest that
this large amplitude instability can lead to wave breaking. Such
wave breaking induced by large unstable waves had also been used to
explain the breaking of waves approaching beaches (\cite{duncan},
\cite{peregrine}, \cite{peregrine2}). More discussions of these
issues are found in Remarks \ref{remark-appx} and \ref{remark-break}
(Section 5).

The proof of Theorems \ref{thm-main} and \ref{thm-minor} also has some
implications for the spectral stability of solitary waves with $\mu<\tilde
{\mu}_{1}$. We note that the traveling waves of full water waves even with
vorticity are shown (\cite{bona-sachs}, \cite{hur-lin}) to be always highly
indefinite energy saddles under the constraints of constant momentum, mass
etc. Therefore, their stability cannot be studied by the traditional method of
proving (constrained) energy minimizers as in many model equations such as the
KDV type equations (\cite{benjamin72}, \cite{bss87}). So far there are few
effective methods for proving nonlinear stability of energy saddles. So
naturally, the first step is to study their spectral stability, namely, to
show that there does not exist an exponentially growing solution to the
linearized problem. The following theorem might be useful for this purpose.

\begin{theorem}
\label{thm-transition}Assume the hypothesis (H1). Suppose that there is a
sequence of purely growing modes $e^{\lambda_{n}t}\left[  \eta_{n}\left(
x\right)  ,\psi_{n}\left(  x,y\right)  \right]  $ $\left(  \lambda
_{n}>0\right)  \ $to the linearized problem for solitary waves with parameters
$\left\{  \mu_{n}\right\}  $, with $\lambda_{n}\rightarrow0+$ and $\mu
_{n}\rightarrow\mu_{0}$ where $\mu_{0}$ is not a turning point, then we must
have $\frac{\partial E}{\partial\mu}\left(  \mu_{0}\right)  =0$.
\end{theorem}

By the above theorem, if an oscillatory instability can be excluded, that is,
any growing mode is shown to be purely growing, then the transition of
instability can only happen at the energy extrema or turning points. Numerical
results in \cite{tanaka86}, \cite{long-tanaka} justify that the growing modes
found are indeed purely growing for solitary waves before the first few
turning points. If additionally the spectral stability of small solitary waves
can be proved, then it follows that the solitary waves are spectrally stable
up to the wave of maximal energy. \

\textbf{Comments and ideas of the proof. }\ \ First, we comments on related
results in the literature. In \cite{saff85}, Saffman considered the spectral
stability of periodic waves in deep water (Stokes waves), under perturbations
of the same period (so called superharmonic perturbations). The picture of
superharmonic instability of Stokes waves (\cite{tanaka83}) is similar to that
of the instability of solitary waves. The approach of \cite{saff85} is to take
the finite mode truncation of the linearized Hamiltonian formulation of
Zakharov (\cite{zakrahov}) and study the eigenvalue problem for the matrix
obtained. By assuming the existence of a sequence of purely growing modes with
the growth rate $\lambda_{n}\rightarrow0+$ and parameters $\mu_{n}%
\rightarrow\mu_{0}$, the solvability conditions are checked to the
second order to show that $\mu_{0}$ must be an energy extremum. That
is, an analogue of Theorem \ref{thm-transition} was established in
\cite{saff85} for Stokes waves. However, the analysis in
\cite{saff85} is at a rather formal level. First, Zakharov's
Hamiltonian formulation has a highly indefinite quadratic form that
is unbounded from both below and above. This is due to the
indefiniteness of the energy functional of the pure gravity water
wave problem (\cite{bona-sachs}) as mentioned before. So it is
unclear how to pass the finite truncation results in \cite{saff85}
to the original water wave problem. Secondly, an implicit assumption
in \cite{saff85} is that the truncated matrix has the zero
eigenvalue of geometric multiplicity $1$. It is unclear how to check
and relate this assumption to the properties of steady waves. For
solitary waves, the truncation approach of \cite{saff85} seems
difficult to apply because of the unbounded domain. Recently, in
\cite{kataoka-peri}, \cite{kataoka-soli}, Kataoka recovered
Saffman's formal result (or analogues of Theorem
\ref{thm-transition}) for periodic waters in water of finite depth
and for interfacial solitary waves in a different way. The analysis
of \cite{kataoka-peri}, \cite{kataoka-soli} is again formal and of
similar nature as \cite{saff85}. That is, by assuming the existence
of purely growing modes with vanishing growth rates, the first two
solvability conditions were checked to show that the limiting
parameter is an energy extremum. We note that in the above papers of
Kataoka and Saffman, the existence of a sequence of \textit{purely}
growing modes with vanishing growth rates was only \textit{assumed
}but never proved.  Moreover, their analysis are perturbative, only
for travelling waves near energy extrema. In this paper, we
rigorously prove the linear instability of large solitary waves and
our method is non-perturbative, which can apply to solitary waves
far from the energy extrema.

Below, we briefly discuss main ideas in the proof of Theorems \ref{thm-main}
and \ref{thm-minor}. To avoid the issue of indefiniteness of energy
functional, we do not adapt Zakharov's Hamiltonian formulation in terms of the
vector potential $\phi$ on the free surface and the wave profile $\eta$. We
use the linearized system derived in \cite{hur-lin}, in terms of the
infinitesimal perturbations of the wave profile $\eta$ and the stream function
$\psi$ restricted on the steady surface $\mathcal{S}_{e}.$Then we further
reduce this system to get a family of operator equations $\mathcal{A}%
_{e}^{\lambda}\left(  \psi|_{\mathcal{S}_{e}}\right)  =0$, where $\lambda>0$
is the unstable eigenvalue to be found$.$The operator $\mathcal{A}%
_{e}^{\lambda}$ is the sum of the Dirichlet-Neumann operator and a bounded but
nonlocal operator. The idea of above reduction is to relate the eigenvalue
problems to the elliptic type problems for steady waves. The hodograph
transformation is then used to get equivalent operators $\mathcal{A}^{\lambda
}$ defined on the whole line. The existence of a purely growing mode is
equivalent to find some $\lambda>0$ such that the operator $\mathcal{A}%
^{\lambda}$ has a nontrivial kernel. This is achieved by using a continuity
argument to exploit the difference of the spectra of $\mathcal{A}^{\lambda}$
near infinity and zero.

The idea of introducing nonlocal dispersion operators with a continuity
argument to get instability criteria originates from our previous works
(\cite{lin-sima}, \cite{lin-cmp}, \cite{lin-bgk}) on 2D ideal fluid and 1D
electrostatic plasma, which have also been extended to galaxy dynamics
\cite{guo-lin} and 3D electromagnetic plasmas \cite{lin-strauss1},
\cite{lin-strauss2}. The new issue in the current case is the influence of the
symmetry of the problem. More specifically, we need to understand the movement
of the kernel of $\mathcal{A}^{0}$ that is due to the translation symmetry,
under the perturbation of $\mathcal{A}^{0}$ to $\mathcal{A}^{\lambda}$ for
small $\lambda$. This is obtained in a moving kernel formula (Lemma
\ref{lemma-moving}). The convergence of $\mathcal{A}^{\lambda}$ to
$\mathcal{A}^{0}$ is very weak, so the usual perturbation theories do not
apply and the asymptotic perturbation theory by Vock and Hunziker
(\cite{vock-hunz82}) has to be used to study perturbations of the eigenvalues
of $\mathcal{A}^{0}$. An important technical part in our proof is to use the
supercritical property $F>1$ and the decay of solitary waves to obtain a
priori estimates and gain certain compactness. In particular, $F>1$ implies
that the essential spectra of the operators $\mathcal{A}^{\lambda}$ lie in the
right half complex plane. The techniques developed in this paper have been
recently extended to show instability of large Stokes waves (\cite{lin-stokes}%
) and get instability criteria for periodic and solitary waves of rather
general dispersive wave equations (\cite{lin-dispersive}, \cite{lin-periodic}).

In Lemma \ref{lemma-A0-property}, we prove that the zero-limiting
operator $\mathcal{A}^{0}\ $is exactly the \textit{same} operator
used in \cite{plot91} for studying the secondary bifurcation of
solitary waves. This link is interesting and a little unexpected
since our derivation of the operator $\mathcal{A}^{0}$ is totally
unrelated to the formulation used in \cite{plot91}. We note that the
bifurcation results in \cite{plot91} have no implications for
instability of solitary waves. Indeed, for water wave problems,
there seems to be no definite relations between the stability and
bifurcation of travelling waves. For example, it was shown in
\cite[Remark 4.13]{hur-lin} that the bifurcation of nontrivial
traveling water waves are unrelated to the exchange of stability of
trivial flows. From numerical works (\cite{tanaka86},
\cite{saff-chen}, \cite{kataoka-soli}), the exchange of instability
at energy extrema for solitary waves does not imply any secondary
bifurcation there.

This paper is organized as follows. In Section 2, we give the formulation of
the linearized problem and derive the nonlocal dispersion operators
$\mathcal{A}^{\lambda}$. Section 3 is devoted to study properties of the
operators $\mathcal{A}^{\lambda}$, in particular, their essential spectrum. In
Section 4, we apply the asymptotic perturbation theory to study the
eigenvalues of $\mathcal{A}^{\lambda}$ for $\lambda$ near $0$. In Section 5,
we derive a moving kernel formula and prove the main theorems. Some important
formulae are proved in Appendix.

\section{Formulation for linear instability}

In this Section, a solitary wave solution of (\ref{stream}) is held fixed, as
such it serves as the undisturbed state about which the system (\ref{E:Euler1}%
)--(\ref{E:bottom}) is linearized. The derivation is performed in the moving
frame of references, in which the wave profile appears to be stationary and
the flow is steady. Let us denote the undisturbed wave profile and relative
stream function by $\eta_{e}(x)$ and $\psi_{e}(x,y)$, respectively, which
satisfy the system (\ref{stream}). The steady relative velocity field is
\[
(u_{e}(x,y)+c,v_{e}(x,y))=(\psi_{ey}(x,y),-\psi_{ex}(x,y)),
\]
and the steady pressure $P_{e}(x,y)$ is determined through
\begin{equation}
P_{e}(x,y)=\frac{1}{2}c^{2}-\tfrac{1}{2}|\nabla\psi_{e}(x,y)|^{2}-gy.
\label{pressure}%
\end{equation}
Let
\[
\mathcal{D}_{e}=\{(x,y):-\infty<x<+\infty\,,\,-h<y<\eta_{e}(x)\}
\]
and%
\[
\mathcal{S}_{e}=\{(x,\eta_{e}(x)):-\infty<x<+\infty\}
\]
denote, respectively, the undisturbed fluid domain and the steady wave profile.

Let us denote
\[
\left(  \eta(t;x),u(t;x,y),v(t;x,y),P(t;x,y)\right)
\]
to be the infinitesimal perturbations of the wave profile, the velocity field
and the pressure respectively. The stream function perturbation is
$\psi\left(  t;x,y\right)  $, such that $\left(  u,v\right)  =\left(  \psi
_{y},-\psi_{x}\right)  .$

The linearized water-wave problem was derived in \cite{hur-lin}, and it takes
the following form in the irrotational case:
\begin{subequations}
\label{eqn-L}%
\begin{equation}
\Delta\psi=0\qquad in\quad\mathcal{D}_{e}, \label{eqn-L-vor}%
\end{equation}%
\begin{equation}
\partial_{t}\eta+\partial_{\tau}(\psi_{ey}\eta)+\partial_{\tau
}\psi=0\qquad\text{on}\quad\mathcal{S}_{e}; \label{eqn-L-eta}%
\end{equation}%
\begin{equation}
P+P_{ey}\eta=0\qquad\text{on}\quad\mathcal{S}_{e};
\label{eqn-L-P}%
\end{equation}%
\begin{equation}
\partial_{t}\partial_{n}\psi+\partial_{\tau}(\psi_{ey}%
\partial_{n}\psi)+\partial_{\tau}P=0\qquad\text{on}\quad\mathcal{S}_{e};
\label{eqn-L-phi-n}%
\end{equation}%
\begin{equation}
\partial_{x}\psi=0\qquad\text{on}\quad\{y=-h\},
\label{eqn-L-bott}%
\end{equation}
\end{subequations}
where%

\[
\partial_{\tau}f=\partial_{x}f+\eta_{ex}\partial_{y}f\ \ \ \text{and
\ }\partial_{n}f=\partial_{y}f-\eta_{ex}\partial_{x}f
\]
denote the tangential and normal derivatives of a function $f\left(
x,y\right)  $ on the curve $\{y=\eta_{e}(x)\}$. Alternatively, $\partial
_{\tau}f(x)=\frac{d}{dx}f(x,\eta_{e}(x))$. Note that the above linearized
system may be viewed as one for $\psi(t;x,y)$ and $\eta(t;x)$. Indeed,
$P(t;x,\eta_{e}(x))$ is determined through (\ref{eqn-L-P}) in terms of
$\eta(t;x)$ and other physical quantities are similarly determined in terms of
$\psi(t;x,y)$ and $\eta(t;x)$.

A \emph{growing mode} refers to a solution to the linearized water-wave
problem (\ref{eqn-L-vor})-(\ref{eqn-L-bott}) of the form
\[
(\eta(t;x),\psi(t;x,y))=(e^{\lambda t}\eta(x),e^{\lambda t}\psi(x,y))
\]
and $P(t;x,\eta_{e}(x))=e^{\lambda t}P(x,\eta_{e}(x))$ with $\operatorname{Re}%
\,\lambda>0$. For a growing mode, the linearized system (\ref{eqn-L}) becomes%
\begin{equation}
\Delta\psi=0\qquad in\quad\mathcal{D}_{e} \label{eqn-G-phi}%
\end{equation}
and the following boundary conditions on $\mathcal{S}_{e},$
\begin{equation}
\lambda\eta(x)+\frac{d}{dx}\left(  \psi_{ey}(x,\eta_{e}(x))\eta(x)\right)
=-\frac{d}{dx}\psi(x,\eta_{e}(x)), \label{eqn-G-eta}%
\end{equation}%
\begin{equation}
P(x,\eta_{e}(x))+P_{ey}(x,\eta_{e}(x))\eta(x)=0, \label{eqn-G-P}%
\end{equation}%
\begin{equation}
\lambda\psi_{n}(x)+\frac{d}{dx}\left(  \psi_{ey}(x,\eta_{e}(x))\psi
_{n}(x)\right)  =-\frac{d}{dx}P(x,\eta_{e}(x)). \label{eqn-G-phi-n}%
\end{equation}
We impose the following boundary condition on the flat bottom
\begin{equation}
\psi(x,-h)=0, \label{eqn-G-bottom}%
\end{equation}
from which (\ref{eqn-L-bott}) follows. In summary, the growing-mode problem
for a solitary water-wave is to find a nontrivial solution of (\ref{eqn-G-phi}%
)-(\ref{eqn-G-bottom}) with $\operatorname{Re}\,\lambda>0$. Below, we look for
purely growing modes with $\lambda>0$ and reduce the system (\ref{eqn-G-phi}%
)-(\ref{eqn-G-bottom}) to one single equation for $\psi|_{\mathcal{S}_{e}}$.
For simplicity, here and in the sequel we identify $\psi_{ey}(x)$ with
$\psi_{ey}(x,\eta_{e}(x))$ and $\phi(x)$ with $\phi(x,\eta_{e}(x))$, etc.
First, we introduce the following operator%
\begin{equation}
\mathcal{C}^{\lambda}=\left(  \lambda+\frac{d}{dx}\left(  \psi_{ey}%
(x)\cdot\right)  \right)  ^{-1}\frac{d}{dx}. \label{defn-c-lb-direct}%
\end{equation}
Note that $\psi_{ey}>0$ in $\mathcal{D}_{e}\ $by the maximum principle and
Hopf's principle (\cite{craige88}), and the fact that $\psi_{ey}%
=u_{e}+c\rightarrow c$ as $\left\vert x\right\vert \rightarrow\infty$. Thus
\begin{equation}
c_{0}\leq\psi_{ey}(x,\eta_{e}(x))=u_{e}+c\leq c_{1}, \label{bound-phi-e}%
\end{equation}
for some constant $c_{0},c_{1}>0$. Defining the operator
\[
\mathcal{D}\phi=\frac{d}{dx}\left(  \psi_{ey}(x)\phi(x)\right)  ,
\]
we can write $\mathcal{C}^{\lambda}$ as
\begin{equation}
\mathcal{C}^{\lambda}=\frac{\mathcal{D}}{\lambda+\mathcal{D}}\frac{1}%
{\psi_{ey}(x)}=\left(  1-\frac{\mathcal{\lambda}}{\lambda+\mathcal{D}}\right)
\frac{1}{\psi_{ey}(x)}. \label{defn-C-lb}%
\end{equation}
Denote $L_{\psi_{ey}}^{2}(\mathcal{S}_{e})$ to be the $\psi_{ey}-$weighted
$L^{2}$ space on $\mathcal{S}_{e}$. Because of the bound (\ref{bound-phi-e}),
$L_{\psi_{ey}}^{2}(\mathcal{S}_{e})$ and $L^{2}(\mathcal{S}_{e})$ are norm
equivalent. Note that the operator $\mathcal{D}$ is anti-symmetric on
$L_{\psi_{ey}}^{2}(\mathcal{S}_{e})$.

\begin{lemma}
\label{lemma-e-lb}For $\lambda>0,$ define the operator $\mathcal{E}%
^{\lambda,\pm}:L_{\psi_{ey}}^{2}(\mathcal{S}_{e})\rightarrow L_{\psi_{ey}}%
^{2}(\mathcal{S}_{\epsilon})$ by%

\[
\mathcal{E}^{\lambda,\pm}=\frac{\mathcal{\lambda}}{\lambda\pm\mathcal{D}}.
\]
Then,

\textrm{{(a)} The operator }$\mathcal{E}^{\lambda,\pm}$\textrm{ is continuous
in }$\lambda$, \textrm{
\begin{equation}
\left\Vert \mathcal{E}^{\lambda,\pm}\right\Vert _{L_{\psi_{ey}}^{2}%
(\mathcal{S}_{e})\rightarrow L_{\psi_{ey}}^{2}(\mathcal{S}_{e})}%
\leq1,\ \label{estimate-E-lb}%
\end{equation}
and}%
\begin{equation}
\left\Vert 1-\mathcal{E}^{\lambda,\pm}\right\Vert _{L_{\psi_{ey}}%
^{2}(\mathcal{S}_{e})\rightarrow L_{\psi_{ey}}^{2}(\mathcal{S}_{e})}\leq1.
\label{estimate-E-lb-one}%
\end{equation}

\textrm{{(b)} When }$\lambda\rightarrow0+$, $\mathcal{E}^{\lambda,\pm}$
converges to $0$ strongly in $L_{\psi_{ey}}^{2}(\mathcal{S}_{e})$.

(c) When $\lambda\rightarrow+\infty,$ $\mathcal{E}^{\lambda,\pm}$ converges to
$1$ strongly in $L_{\psi_{ey}}^{2}(\mathcal{S}_{e})$.
\end{lemma}

\begin{proof}
Denote $\left\{  M_{\alpha};\alpha\in\mathbf{R}^{1}\right\}  $ to be the
spectral measure of the self-adjoint operator $\mathcal{R}=-i\mathcal{D}$ on
$L_{\psi_{ey}}^{2}(\mathcal{S}_{e})$. Then
\[
\left\Vert \mathcal{E}^{\lambda,\pm}\phi\right\Vert _{L_{\psi_{ey}}^{2}}%
^{2}=\int_{\mathbb{R}}\left\vert \frac{\lambda}{\lambda\pm i\alpha}\right\vert
^{2}d\Vert M_{\alpha}\phi\Vert_{L_{\psi_{ey}}^{2}}^{2}\leq\int_{\mathbb{R}%
}d\Vert M_{\alpha}\phi\Vert_{L_{\psi_{ey}}^{2}}^{2}=\left\Vert \phi\right\Vert
_{L_{\psi_{ey}}^{2}}^{2}%
\]
and (\ref{estimate-E-lb}) follows. Similarly, we get the estimate
(\ref{estimate-E-lb-one}). To prove (b), we take any $\phi\in L_{\psi_{ey}%
}^{2}(\mathcal{S}_{e})$ and denote the function $\xi(\alpha)$ to be such that
$\xi(\alpha)=0$ for $\alpha\neq0$ and $\xi(0)=1$. Then by the dominant
convergence theorem, when $\lambda\rightarrow0+,$
\begin{align*}
\left\Vert \mathcal{E}^{\lambda,\pm}\phi\right\Vert _{L_{\psi_{ey}}^{2}}^{2}
&  =\int_{\mathbb{R}}\left\vert \frac{\lambda}{\lambda\pm i\alpha}\right\vert
^{2}d\Vert M_{\alpha}\phi\Vert_{L_{\psi_{ey}}^{2}}^{2}\\
&  \rightarrow\int_{\mathbb{R}}\xi(\alpha)d\Vert M_{\alpha}\phi\Vert
_{L_{\psi_{ey}}^{2}}^{2}=\left\Vert M_{\{0\}}\phi\right\Vert _{L_{\psi_{ey}%
}^{2}}^{2}.
\end{align*}
Note that $M_{\{0\}}$ is the projector of $L_{\psi_{ey}}^{2}$ to
$\ker\mathcal{D=}\left\{  0\right\}  $. So $M_{\{0\}}=0$ and $\mathcal{E}%
^{\lambda,\pm}\phi\rightarrow0$ in $L_{\psi_{ey}}^{2}$. The proof of (c) is
similar to that of (b) and we skip it.
\end{proof}

By the above lemma and the bound (\ref{bound-phi-e}) on $\psi_{ey},\ $we have

\begin{lemma}
\label{lemma-C-lb}For $\lambda>0$, the operator $\mathcal{C}^{\lambda}%
:L^{2}(\mathcal{S}_{e})\rightarrow L^{2}(\mathcal{S}_{e})$ defined by
(\ref{defn-c-lb-direct}) has the following properties:

(a)
\[
\left\Vert \mathcal{C}^{\lambda}\right\Vert _{L^{2}(\mathcal{S}_{e}%
)\rightarrow L^{2}(\mathcal{S}_{e})}\leq C,
\]
for some constant $C$ independent of $\lambda.$

(b) When $\lambda\rightarrow0+$, $\mathcal{C}^{\lambda}$ converges to
$\frac{1}{\psi_{ey}(x)}$ strongly in $L^{2}(\mathcal{S}_{e})$.

(c) When $\lambda\rightarrow+\infty,$ $\mathcal{C}^{\lambda}$ converges to $0$
strongly in $L^{2}(\mathcal{S}_{e})$.
\end{lemma}

By using the operator $\mathcal{C}^{\lambda}$, the growing mode system
(\ref{eqn-G-phi})-(\ref{eqn-G-bottom}) is reduced to
\begin{equation}
\psi_{n}(x)+\mathcal{C}^{\lambda}P_{ey}\left(  x\right)  \mathcal{C}^{\lambda
}\psi=0,\ \text{on }\mathcal{S}_{e}, \label{eqn-phi-n}%
\end{equation}%
\[
\Delta\psi=0\qquad\text{in}\quad\mathcal{D}_{e},
\]%
\[
\psi(x,-h)=0.
\]
We define the following Dirichlet-Neumann operator $\mathcal{N}_{e}$:
$H^{1}\left(  \mathcal{S}_{e}\right)  \rightarrow L^{2}(\mathcal{S}_{e})$ by
\[
\mathcal{N}_{e}f=\partial_{n}\psi_{f}=\left(  \partial_{y}\psi_{f}-\eta
_{ex}\partial_{x}\psi_{f}\right)  \left(  x,\eta_{e}(x)\right)  ,
\]
where $\psi_{f}$ is the unique solution of the following Dirichlet problem for
$f\in H^{1}\left(  \mathcal{S}_{e}\right)  $
\[
\Delta\psi_{f}=0\qquad\text{in}\quad\mathcal{D}_{e},
\]%
\[
\psi_{f}|_{\mathcal{S}_{e}}=f,
\]%
\[
\psi_{f}(x,-h)=0.
\]
Then the existence of a purely growing mode is reduced to find some
$\lambda>0$ such that the operator $\mathcal{A}_{e}^{\lambda}\ $defined by
\begin{equation}
\mathcal{A}_{e}^{\lambda}=\mathcal{N}_{e}+\mathcal{C}^{\lambda}P_{ey}\left(
x\right)  \mathcal{C}^{\lambda}:H^{1}\left(  \mathcal{S}_{e}\right)
\rightarrow L^{2}(\mathcal{S}_{e}) \label{defn-A-e}%
\end{equation}
has a nontrivial kernel. Note that if we denote by $\phi_{f}$ the holomorphic
conjugate of $\psi_{f}$ in$\ \mathcal{D}_{e}$, then $\partial_{n}\psi
_{f}=\frac{d}{dx}\phi_{f}$. This motivates us to define an analogue of the
Hilbert transformation as in \cite{plot91}, by
\[
\left(  \mathcal{C}_{e}f\right)  \left(  x\right)  =\int_{0}^{x}%
\mathcal{N}_{e}f\ dx.
\]
Then the operator $\mathcal{N}_{e}$ can be written as $\mathcal{N}_{e}%
=\frac{d}{dx}\mathcal{C}_{e}$. From the definition, $\mathcal{C}_{e}f+if$ and
$f-i\mathcal{C}_{e}f$ are the boundary values on $\mathcal{S}_{e}$ of some
analytic functions in $\mathcal{D}_{e}$. Below, we further reduce the operator
$\mathcal{A}_{e}^{\lambda}$ to one defined on the real line. First, we define
the holomorphic mapping $F:\mathcal{D}_{e}\rightarrow\mathbf{R\times}\left(
-ch,0\right)  $ by $F\left(  x,y\right)  =\left(  \phi_{e}\left(  x,y\right)
,\psi_{e}\left(  x,y\right)  \right)  $. We denote
\[
\left(  \xi,\varsigma\right)  =\frac{1}{c}\left(  \phi_{e},\psi_{e}\right)
\in D_{0}=\mathbf{R\times}\left(  -h,0\right)
\]
and define the mapping $G:D_{0}\rightarrow\mathcal{D}_{e}$ by
\[
G\left(  \xi,\varsigma\right)  =\left(  x\left(  \xi,\varsigma\right)
,y\left(  \xi,\varsigma\right)  \right)  =F^{-1}\left(  c\xi,c\varsigma
\right)  .
\]
The flat Dirichlet-Neumann operator $\mathcal{N}$: $H^{1}\left(
\mathbf{R}\right)  \rightarrow L^{2}(\mathbf{R})$ is defined by
\[
\mathcal{N}f=\partial_{\varsigma}\psi_{f}|_{\left\{  \varsigma=0\right\}  },
\]
where $\psi_{f}$ is the solution of the following Dirichlet problem for $f\in
H^{1}\left(  \mathbf{R}\right)  $
\[
\Delta\psi_{f}=0\qquad\text{in}\quad D_{0},
\]%
\[
\psi_{f}|_{\left\{  \varsigma=0\right\}  }=f,
\]%
\[
\psi_{f}|_{\left\{  \varsigma=-h\right\}  }=0.
\]
Similarly, we define the operator $\mathcal{C}$ by $\mathcal{C}f=\int_{0}%
^{x}\mathcal{N}f\ dx$. Then $\mathcal{N}=\frac{d}{d\xi}\mathcal{C}$ and
$\mathcal{C}f+if$ or $f-i\mathcal{C}f$ are the boundary values on $\left\{
\varsigma=0\right\}  $ of analytic functions in $D_{0}$. Moreover,
$\mathcal{N\ }$is a Fourier multiplier operator with the symbol
(\cite{craige88})
\begin{equation}
n\left(  k\right)  =\frac{k}{\tanh\left(  kh\right)  }. \label{symbol-N}%
\end{equation}
To separate the uniform flow $\left(  c,0\right)  ,$ we rewrite
\[
\left(  x\left(  \xi,\varsigma\right)  ,y\left(  \xi,\varsigma\right)
\right)  =\left(  \xi,\varsigma\right)  +\left(  x_{1}\left(  \xi
,\varsigma\right)  ,y_{1}\left(  \xi,\varsigma\right)  \right)  .
\]
Denote
\begin{equation}
w\left(  \xi\right)  =y_{1}\left(  \xi,0\right)  . \label{defn-w}%
\end{equation}
Then we can set $x_{1}\left(  \xi,0\right)  =\mathcal{C}w$ by adding a proper
constant to the vector potential $\phi_{e}$. The mapping $G$ restricted on
$\left\{  \varsigma=0\right\}  $ induces a mapping $B:\mathbf{R\rightarrow
}\mathcal{S}_{e}$ defined by $B\left(  \xi\right)  =\left(  \xi+\mathcal{C}%
w,w\right)  $. Denote $z=x+iy$ and $p=\xi+i\varsigma$, then
\[
u_{e}+c-iv_{e}=\frac{d\left(  \phi_{e}+i\psi_{e}\right)  }{dz}=c\frac{dp}%
{dz}=c\frac{1}{\frac{dz}{dp}}=\frac{c}{1+\partial_{\xi}x_{1}+i\partial_{\xi
}y_{1}}.
\]
So on $\left\{  \varsigma=0\right\}  $, we get
\begin{equation}
u_{e}+c=c\frac{1+\mathcal{N}w}{\left\vert W\right\vert ^{2}},\ v_{e}%
=c\frac{w^{\prime}}{\left\vert W\right\vert ^{2}} \label{velocity-w}%
\end{equation}
where
\begin{equation}
W=\left(  1+\partial_{\xi}x_{1}+i\partial_{\xi}y_{1}\right)  |_{\left\{
\varsigma=0\right\}  }=1+\mathcal{C}w^{\prime}+iw^{\prime}=1+\mathcal{N}%
w+iw^{\prime} \label{defn-W}%
\end{equation}
and $^{\prime}$ denotes the $\xi-$derivative. From (\ref{velocity-w}),
\[
1+\mathcal{N}w=\frac{\left(  u_{e}+c\right)  c}{\left(  u_{e}+c\right)
^{2}+v_{e}^{2}},
\]
and thus by (\ref{bound-phi-e}), there exists $c_{2},c_{3}>0$ such that
\begin{equation}
c_{2}<1+\mathcal{N}w<c_{3}. \label{bound-cw}%
\end{equation}
We define the operator $\mathcal{B}:L^{2}(\mathcal{S}_{e})\rightarrow
L^{2}\left(  \mathbf{R}\right)  $ by
\[
\left(  \mathcal{B}f\right)  \left(  \xi\right)  =f\left(  B\left(
\xi\right)  \right)  =f\left(  \xi+\mathcal{C}w\left(  \xi\right)  ,w\left(
\xi\right)  \right)  \text{, for any }f\in L^{2}(\mathcal{S}_{e}).
\]
Since $\mathcal{BC}_{e}=\mathcal{CB}$ and $\frac{d}{d\xi}\mathcal{B=}\left(
1+\mathcal{N}w\right)  \mathcal{B}\frac{d}{dx},\,$we have
\[
\mathcal{BN}_{e}\mathcal{B}^{-1}=\mathcal{B}\frac{d}{dx}\mathcal{C}%
_{e}\mathcal{B}^{-1}=\frac{1}{1+\mathcal{N}w}\frac{d}{d\xi}\mathcal{C=}%
\frac{1}{\left(  1+\mathcal{N}w\right)  }\mathcal{N},
\]
and%
\begin{align*}
\mathcal{BA}_{e}^{\lambda}\mathcal{B}^{-1}  &  =\mathcal{BN}_{e}%
\mathcal{B}^{-1}+\left(  \mathcal{BC}^{\lambda}\mathcal{B}^{-1}\right)
\mathcal{B}P_{ey}\mathcal{B}^{-1}\left(  \mathcal{BC}^{\lambda}\mathcal{B}%
^{-1}\right) \\
&  =\frac{1}{1+\mathcal{N}w}\mathcal{N}+\mathcal{\tilde{C}}^{\lambda}%
P_{ey}\left(  \xi\right)  \mathcal{\tilde{C}}^{\lambda}.
\end{align*}
Here,
\[
\mathcal{\tilde{C}}^{\lambda}=\mathcal{BC}^{\lambda}\mathcal{B}^{-1}=\left(
\lambda+\frac{1}{1+\mathcal{N}w}\frac{d}{d\xi}\left(  \psi_{ey}(\xi
)\cdot\right)  \right)  ^{-1}\left(  \frac{1}{1+\mathcal{N}w}\frac{d}{d\xi
}\right)
\]
and we use $\psi_{ey}(\xi),P_{ey}\left(  \xi\right)  $ to denote $\psi
_{ey}\left(  B\left(  \xi\right)  \right)  ,$ $P_{ey}\left(  B\left(
\xi\right)  \right)  \ $etc. For $\lambda>0,\ $we define the operator
$\mathcal{A}^{\lambda}:H^{1}\left(  \mathbf{R}\right)  \rightarrow
L^{2}(\mathbf{R})$ by%
\[
\mathcal{A}^{\lambda}=\mathcal{N}+\left(  1+\mathcal{N}w\right)
\mathcal{\tilde{C}}^{\lambda}P_{ey}\left(  \xi\right)  \mathcal{\tilde{C}%
}^{\lambda}.
\]
Then the existence of a purely growing mode is equivalent to find some
$\lambda>0$ such that the operator $\mathcal{A}^{\lambda}$ has a nontrivial kernel.

\section{Properties of the operator $\mathcal{A}^{\lambda}$}

In this section, we study the spectral properties of the operator
$\mathcal{A}^{\lambda}$. First, we have the following estimate for the
Dirichlet-Neumann operator $\mathcal{N}.$

\begin{lemma}
\label{lemma-operator-N}There exists $C_{0}>0$, such that for any $\delta
\in\left(  0,1\right)  \ $and $\ f\in H^{\frac{1}{2}}\left(  \mathbf{R}%
\right)  ,$ we have
\[
\left(  \mathcal{N}f,f\right)  \geq\left(  1-\delta\right)  \frac{1}%
{h}\left\Vert f\right\Vert _{L^{2}}^{2}+C_{0}\delta\left\Vert f\right\Vert
_{H^{\frac{1}{2}}}^{2},\ \
\]
where%
\[
\left\Vert f\right\Vert _{H^{\frac{1}{2}}}^{2}=\int\left(  1+\left\vert
k\right\vert \right)  \left\vert \hat{f}\left(  k\right)  \right\vert ^{2}dk
\]
and $\hat{f}\left(  k\right)  $ is the Fourier transformation of $f$.
\end{lemma}

\begin{proof}
By the definition (\ref{symbol-N}),
\[
\left(  \mathcal{N}f,f\right)  =\int\frac{k}{\tanh\left(  kh\right)
}\left\vert \hat{f}\left(  k\right)  \right\vert ^{2}dk=\int\frac{\left\vert
k\right\vert }{\tanh\left(  \left\vert k\right\vert h\right)  }\left\vert
\hat{f}\left(  k\right)  \right\vert ^{2}dk.
\]
It is easy to check that the function
\[
h\left(  x\right)  =\frac{x}{\tanh\left(  xh\right)  },\ \ x\geq0
\]
satisfies
\[
h\left(  x\right)  -\frac{1}{h}\geq0\ \text{and }\lim_{x\rightarrow\infty
}\frac{h\left(  x\right)  -\frac{1}{h}}{x}=1.
\]
So there exists $K>0$, such that $h\left(  x\right)  -\frac{1}{h}\geq\frac
{1}{2}x$, when $x>K$. Thus
\begin{align*}
\left(  \mathcal{N}f,f\right)   &  \geq\frac{1}{h}\int\left\vert \hat
{f}\left(  k\right)  \right\vert ^{2}dk+\frac{1}{2}\int_{\left\vert
k\right\vert \geq K}\left\vert k\right\vert \left\vert \hat{f}\left(
k\right)  \right\vert ^{2}dk\\
&  \geq\left(  1-\delta\right)  \frac{1}{h}\left\Vert f\right\Vert _{L^{2}%
}^{2}+\frac{\delta}{2h}\left\Vert f\right\Vert _{L^{2}}^{2}+\frac{\delta}%
{2h}\int_{\left\vert k\right\vert \leq K}\left\vert \hat{f}\left(  k\right)
\right\vert ^{2}dk+\frac{1}{2}\int_{\left\vert k\right\vert \geq K}\left\vert
k\right\vert \left\vert \hat{f}\left(  k\right)  \right\vert ^{2}dk\\
&  \geq\left(  1-\delta\right)  \frac{1}{h}\left\Vert f\right\Vert _{L^{2}%
}^{2}+\min\left\{  \frac{\delta}{2h},\frac{\delta}{2Kh},\frac{1}{2}\right\}
\int\left(  1+\left\vert k\right\vert \right)  \left\vert \hat{f}\left(
k\right)  \right\vert ^{2}dk
\end{align*}
This proves the Lemma with $C_{0}=\min\left\{  \frac{\delta}{2h},\frac{\delta
}{2Kh},\frac{1}{2}\right\}  $.
\end{proof}

We have the following properties for the operator $\mathcal{\tilde{C}%
}^{\lambda}$.

\begin{lemma}
\label{lemma-C-tilta}For $\lambda>0$, the operator $\mathcal{\tilde{C}%
}^{\lambda}:L^{2}(\mathbf{R})\rightarrow L^{2}(\mathbf{R})$ defined by
(\ref{defn-c-lb-direct}) satisfies:

(a)
\[
\left\Vert \mathcal{\tilde{C}}^{\lambda}\right\Vert _{L^{2}(\mathbf{R}%
)\rightarrow L^{2}(\mathbf{R})}\leq C,
\]
for some constant $C$ independent of $\lambda.$

(b) When $\lambda\rightarrow0+$, $\mathcal{\tilde{C}}^{\lambda}$ converges to
$\frac{1}{\psi_{ey}(\xi)}$ strongly in $L^{2}(\mathbf{R})$.

(c) When $\lambda\rightarrow+\infty,$ $\mathcal{\tilde{C}}^{\lambda}$
converges to $0$ strongly in $L^{2}(\mathbf{R})$.
\end{lemma}

\begin{proof}
By (\ref{bound-cw}), the operator $\mathcal{B}$ and $\mathcal{B}^{-1}$ are
bounded. Since $\mathcal{\tilde{C}}^{\lambda}=\mathcal{BC}^{\lambda
}\mathcal{B}^{-1}$, the above lemma follows directly from Lemma
\ref{lemma-C-lb}.
\end{proof}

To simply notations, we denote $b\left(  \xi\right)  =1+\mathcal{N}w$ and
define the operators
\[
\mathcal{\tilde{D}}=\frac{1}{b\left(  \xi\right)  }\frac{d}{d\xi}\left(
\psi_{ey}(\xi)\cdot\right)  \text{ \ and }\ \mathcal{\tilde{E}}^{\lambda,\pm
}\phi(x)=\frac{\mathcal{\lambda}}{\lambda\pm\mathcal{\tilde{D}}}.
\]
The operator $\mathcal{\tilde{D}}$ is anti-symmetric in the $b\psi_{ey}%
-$weighted space $L_{b\psi_{ey}}^{2}\left(  \mathbf{R}\right)  $. Similar to
the proof of Lemma \ref{lemma-e-lb}, we have

\begin{lemma}
\label{lemma-property-E-t-lb}(a) For any $\lambda>0,$\textrm{ }%
\begin{equation}
\mathrm{\left\Vert \mathcal{\tilde{E}}^{\lambda,\pm}\right\Vert _{L_{b\psi
_{ey}}^{2}\rightarrow L_{b\psi_{ey}}^{2}}\leq1,\ \ } \label{E-t-lb-leq-1}%
\end{equation}
and
\begin{equation}
\left\Vert 1-\mathcal{\tilde{E}}^{\lambda,\pm}\right\Vert _{L_{b\psi_{ey}}%
^{2}\rightarrow L_{b\psi_{ey}}^{2}}\leq1. \label{E-t-lb-one-leq-1}%
\end{equation}

\textrm{{(b)} When }$\lambda\rightarrow0+$, $\mathcal{\tilde{E}}^{\lambda,\pm
}$ converges to $0$ strongly in $L_{b\psi_{ey}}^{2}$.

(c) When $\lambda\rightarrow+\infty,$ $\mathcal{\tilde{E}}^{\lambda,\pm}$
converges to $1$ strongly in $L_{b\psi_{ey}}^{2}$.
\end{lemma}

The operator $\mathcal{\tilde{C}}^{\lambda}$ can be written as
\[
\mathcal{\tilde{C}}^{\lambda}=\left(  1-\frac{\mathcal{\lambda}}%
{\lambda+\mathcal{\tilde{D}}}\right)  \frac{1}{\psi_{ey}(\xi)}=\left(
1-\mathcal{\tilde{E}}^{\lambda,+}\right)  \frac{1}{\psi_{ey}(\xi)}.
\]

\begin{proposition}
\label{essential-spec}For any $\lambda>0$, we have
\begin{equation}
\sigma_{\text{ess}}\left(  \mathcal{A}^{\lambda}\right)  \subset\left\{
z\ |\ \operatorname{Re}\lambda\geq\frac{1}{2}\left(  \frac{1}{h}-\frac
{g}{c^{2}}\right)  \right\}  . \label{bound-essential}%
\end{equation}

\end{proposition}

We note that
\begin{equation}
\delta_{0}:=\frac{1}{h}-\frac{g}{c^{2}}>0 \label{defn-d-0}%
\end{equation}
by Property (P2), so the above Proposition shows that the essential spectrum
of $\mathcal{A}^{\lambda}$ lies on the right half plane and is away from the
imaginary axis. To prove Proposition \ref{essential-spec}, we need the
following lemmas.

\begin{lemma}
\label{lemma-estimates-E-t}For any $u\in H^{\frac{1}{2}}\left(  \mathbf{R}%
\right)  $, we have \

(i) For any $\lambda>0$,
\begin{equation}
\left\Vert \mathcal{\tilde{E}}^{\lambda,\pm}u\right\Vert _{H^{\frac{1}{2}}%
}\leq C\left\Vert u\right\Vert _{H^{\frac{1}{2}}},\ \label{estimate-e-t-lb}%
\end{equation}
for some constant $C$ independent of $\lambda$.

Below, let $F\left(  \xi\right)  $ be a fixed bounded function that decays at
infinity. Then

(ii) Given $\lambda>0,\ $for any $\varepsilon>0$, there exists a constant
$C_{\varepsilon}$ such that
\[
\left\Vert F\mathcal{\tilde{E}}^{\lambda,\pm}u\right\Vert _{L^{2}}%
\leq\varepsilon\left\Vert u\right\Vert _{H^{\frac{1}{2}}}+C_{\varepsilon
}\left\Vert \frac{u}{1+\xi^{2}}\right\Vert _{L^{2}}.
\]

(iii) For any $\varepsilon>0$, there exists $\lambda_{\varepsilon}>0$, such
that when $0<\lambda<\lambda_{\varepsilon}$,
\[
\left\Vert F\mathcal{\tilde{E}}^{\lambda,\pm}u\right\Vert _{L^{2}}%
\leq\varepsilon\left\Vert u\right\Vert _{H^{\frac{1}{2}}}\text{.}%
\]

(iv) For any $\varepsilon>0$, there exists $\Lambda_{\varepsilon}>0$, such
that when $\lambda>\Lambda_{\varepsilon},$
\[
\left\Vert F\left(  1-\mathcal{\tilde{E}}^{\lambda,\pm}\right)  u\right\Vert
_{L^{2}}\leq\varepsilon\left\Vert u\right\Vert _{H^{\frac{1}{2}}}\text{.}%
\]

\end{lemma}

\begin{proof}
Proof of (i): Denote $\left\{  \tilde{M}_{\alpha};\alpha\in\mathbf{R}%
^{1}\right\}  $ to be the spectral measure of the self-adjoint operator
$\tilde{R}=-i\mathcal{\tilde{D}}$ on $L_{b\psi_{ey}}^{2}$. For $s\geq0$, we
define the space
\[
\tilde{H}^{s}=\left\{  u\in L_{b\psi_{ey}}^{2}|\ \left\vert \tilde
{R}\right\vert ^{s}u\in L_{b\psi_{ey}}^{2}\right\}
\]
with the norm
\[
\left\Vert u\right\Vert _{\tilde{H}^{s}}=\left\Vert u\right\Vert
_{L_{b\psi_{ey}}^{2}}+\left\Vert \left\vert \tilde{R}\right\vert
^{s}u\right\Vert _{L_{b\psi_{ey}}^{2}}=\left\Vert u\right\Vert _{L_{b\psi
_{ey}}^{2}}+\left(  \int_{\mathbf{R}}\left\vert \alpha\right\vert ^{2s}%
d\Vert\tilde{M}_{\alpha}u\Vert_{L_{b\psi_{ey}}^{2}}^{2}\right)  ^{\frac{1}{2}%
},
\]
where $\left\vert \tilde{R}\right\vert ^{s}$ is the positive self-adjoint
operator defined by $\int\left\vert \alpha\right\vert ^{s}d\tilde{M}_{\alpha}%
$. We claim that the norm $\left\Vert \cdot\right\Vert _{\tilde{H}^{s}}$ is
equivalent to the norm $\left\Vert \cdot\right\Vert _{H^{s}}$, for $0\leq
s\leq1$. When $s=0$, $\tilde{H}^{0}=L_{b\psi_{ey}}^{2}$ and $H^{0}=L^{2}$.
Since $b$ and $\psi_{ey}$ are bounded with positive lower bounds, $\left\Vert
\cdot\right\Vert _{L_{b\psi_{ey}}^{2}}$ and $\left\Vert _{\cdot}\right\Vert
_{L^{2}}$ are equivalent. When $s=1$, we have
\[
\left\Vert u\right\Vert _{\tilde{H}^{1}}=\left\Vert u\right\Vert
_{L_{b\psi_{ey}}^{2}}+\left(  \int\left\vert \frac{1}{b}\frac{d}{d\xi}\left(
\psi_{ey}u\right)  \right\vert ^{2}b\psi_{ey}dx\right)  ^{\frac{1}{2}},
\]
which is clearly equivalent to $\left\Vert u\right\Vert _{H^{1}}^{2}$, again
due to the bounds of $b$ and $\psi_{ey}$. When $0<s<1$, the spaces $\tilde
{H}^{s}$ ($H^{s}$) are the interpolation spaces of $\tilde{H}^{0}$($H^{0}$)
and $\tilde{H}^{1}$ ($H^{1}$)$.$ So by the general interpolation theory
(\cite{bergh}), we get the equivalence of the norms $\left\Vert \cdot
\right\Vert _{\tilde{H}^{s}}$ and $\left\Vert \cdot\right\Vert _{H^{s}}$.
Thus, there exists $C_{1},C_{2}>0$, such that
\begin{equation}
C_{1}\left\Vert u\right\Vert _{\tilde{H}^{\frac{1}{2}}}\leq\left\Vert
u\right\Vert _{H_{2}^{1}}\leq C_{2}\left\Vert u\right\Vert _{\tilde{H}%
^{\frac{1}{2}}}. \label{equivalence-norm}%
\end{equation}
Since $\tilde{R}$ and $\mathcal{\tilde{E}}^{\lambda,\pm}$ are commutable, we
have
\begin{align*}
\left\Vert \mathcal{\tilde{E}}^{\lambda,\pm}u\right\Vert _{\tilde{H}^{\frac
{1}{2}}}  &  =\left\Vert \mathcal{\tilde{E}}^{\lambda,\pm}u\right\Vert
_{L_{b\psi_{ey}}^{2}}+\left\Vert \mathcal{\tilde{E}}^{\lambda,\pm}\left(
\left\vert \tilde{R}\right\vert ^{\frac{1}{2}}u\right)  \right\Vert
_{L_{b\psi_{ey}}^{2}}\\
&  \leq\left\Vert u\right\Vert _{L_{b\psi_{ey}}^{2}}+\left\Vert \left\vert
\tilde{R}\right\vert ^{\frac{1}{2}}u\right\Vert _{L_{b\psi_{ey}}^{2}%
}=\left\Vert u\right\Vert _{\tilde{H}^{\frac{1}{2}}}.
\end{align*}
The estimate (\ref{estimate-e-t-lb}) follows from above and
(\ref{equivalence-norm}).

Proof of (ii): Suppose otherwise, then there exists $\varepsilon_{0}>0$ and a
sequence $\left\{  u_{n}\right\}  \in H^{\frac{1}{2}}\left(  \mathbf{R}%
\right)  $ such that
\[
\left\Vert F\mathcal{\tilde{E}}^{\lambda,\pm}u_{n}\right\Vert _{L^{2}}%
\geq\varepsilon_{0}\left\Vert u_{n}\right\Vert _{H^{\frac{1}{2}}}+n\left\Vert
\frac{u_{n}}{1+\xi^{2}}\right\Vert _{L^{2}}.
\]
We normalize $u_{n}$ by setting $\left\Vert F\mathcal{\tilde{E}}^{\lambda,\pm
}u_{n}\right\Vert _{L^{2}}=1$. Then
\[
\left\Vert u_{n}\right\Vert _{H^{\frac{1}{2}}}\leq\frac{1}{\varepsilon_{0}%
},\ \ \ \left\Vert \frac{u_{n}}{1+\xi^{2}}\right\Vert _{L^{2}}\leq\frac{1}%
{n}.
\]
So there exists $u_{\infty}\in H^{\frac{1}{2}}$, such that $u_{n}\rightarrow
u_{\infty}$ weakly in $H^{\frac{1}{2}}$. Since $\frac{u_{n}}{1+\xi^{2}%
}\rightarrow0$ strongly in $L^{2}$, we have $u_{\infty}=0$. Thus
$v_{n}=\mathcal{\tilde{E}}^{\lambda,\pm}u_{n}$ converges to $0$ weakly in
$L^{2}$. By (i),
\[
\left\Vert v_{n}\right\Vert _{H^{\frac{1}{2}}}\leq C\left\Vert u_{n}%
\right\Vert _{H^{\frac{1}{2}}}\leq\frac{C}{\varepsilon_{0}}.
\]
Let $\chi_{R}\in C_{0}^{\infty}$ \ be a cut-off function for $\left\{
\left\vert \xi\right\vert \leq R\right\}  $. We write%
\[
F=F\chi_{R}+F\left(  1-\chi_{R}\right)  =F_{1}+F_{2}.
\]
Then
\[
\left\Vert F_{2}v_{n}\right\Vert _{L^{2}}\leq C\max_{\left\vert \xi\right\vert
\geq R}\left\vert F\left(  \xi\right)  \right\vert \left\Vert u_{n}\right\Vert
_{L^{2}}\leq C\frac{1}{\varepsilon_{0}}\max_{\left\vert \xi\right\vert \geq
R}\left\vert F\left(  \xi\right)  \right\vert \leq\frac{1}{2},
\]
when $R$ is chosen to be big enough. Since $F_{1}$ has a compact support and
$H^{\frac{1}{2}}\hookrightarrow L^{2}$ is locally compact, so $F_{1}%
v_{n}\rightarrow0$ strongly in $L^{2}$. Thus, when $n$ is large enough,
\[
\left\Vert Fv_{n}\right\Vert _{L^{2}}\leq\left\Vert F_{1}v_{n}\right\Vert
_{L^{2}}+\left\Vert F_{2}v_{n}\right\Vert _{L^{2}}\leq\frac{3}{4}\text{.}%
\]
This is a contradiction to the fact that $\left\Vert Fv_{n}\right\Vert
_{L^{2}}=\left\Vert F\mathcal{\tilde{E}}^{\lambda,\pm}u_{n}\right\Vert
_{L^{2}}=1$.

Proof of (iii): Suppose otherwise, then there exists $\varepsilon_{0}>0$ and a
sequence $\left\{  u_{n}\right\}  \in H^{\frac{1}{2}}\left(  \mathbf{R}%
\right)  ,\ \lambda_{n}\rightarrow0+,$ such that
\[
\left\Vert F\mathcal{\tilde{E}}^{\lambda_{n},\pm}u_{n}\right\Vert _{L^{2}}%
\geq\varepsilon_{0}\left\Vert u_{n}\right\Vert _{H^{\frac{1}{2}}}.
\]
Normalize $u_{n}$ by$\left\Vert F\mathcal{\tilde{E}}^{\lambda_{n},\pm}%
u_{n}\right\Vert _{L^{2}}=1$. Then $\left\Vert u_{n}\right\Vert _{H^{\frac
{1}{2}}}\leq\frac{1}{\varepsilon_{0}}$. Let $u_{n}\rightarrow u_{\infty}$
weakly in $H^{\frac{1}{2}}$. Then for any $v\in L^{2}$, we have
\[
\left(  \mathcal{\tilde{E}}^{\lambda_{n},\pm}u_{n},v\right)  =\left(
u_{n},b\psi_{ey}\mathcal{\tilde{E}}^{\lambda_{n},\mp}\left(  \frac{v}%
{b\psi_{ey}}\right)  \right)  \rightarrow0,
\]
because by Lemma \ref{lemma-property-E-t-lb}, $b\psi_{ey}\mathcal{\tilde{E}%
}^{\lambda_{n},\mp}\left(  \frac{v}{b\psi_{ey}}\right)  \rightarrow0$ strongly
in $L^{2}$ when $\lambda_{n}\rightarrow0+$. So $\mathcal{\tilde{E}}%
^{\lambda_{n},\pm}u_{n}\rightarrow0$ weakly in $L^{2}$, and this leads to a
contradiction as in the proof of (ii).

Proof of (iv) is the same as that of (iii), except that we use the strong
convergence $1-\mathcal{\tilde{E}}^{\lambda_{n},\pm}\rightarrow0\ $when
$\lambda_{n}\rightarrow\infty$.
\end{proof}

\begin{lemma}
\label{lemma-quadrtic-bound-lb}Consider any sequence%
\[
\left\{  u_{n}\right\}  \in H^{\frac{1}{2}}\left(  \mathbf{R}\right)
,\ \ \left\Vert u_{n}\right\Vert _{2}=1,\ \ supp\ u_{n}\subset\left\{
\xi|\ \left\vert \xi\right\vert \geq n\right\}  .
\]
Then for any complex number $z\ $with $\operatorname{Re}z<\frac{1}{2}%
\delta_{0}$, we have
\[
\operatorname{Re}\left(  \left(  \mathcal{A}^{\lambda}-z\right)  u_{n}%
,u_{n}\right)  \geq\frac{1}{4}\delta_{0},
\]
when $n$ is large enough. Here, $\delta_{0}\ $is defined by (\ref{defn-d-0}).
\end{lemma}

\begin{proof}
We have%
\begin{equation}
\operatorname{Re}\left(  \left(  \mathcal{A}^{\lambda}-z\right)  u_{n}%
,u_{n}\right)  =\left(  \mathcal{N}u_{n},u_{n}\right)  -\operatorname{Re}%
z+\operatorname{Re}\left(  b\mathcal{\tilde{C}}^{\lambda}P_{ey}\left(
\xi\right)  \mathcal{\tilde{C}}^{\lambda}u_{n},u_{n}\right)  .
\label{re-A-lb-qua}%
\end{equation}
For $0<\delta<1$ (to be fixed later), by Lemma \ref{lemma-operator-N}
\begin{equation}
\left(  \mathcal{N}u_{n},u_{n}\right)  \geq\left(  1-\delta\right)  \frac
{1}{h}+C_{0}\delta\left\Vert u_{n}\right\Vert _{H^{\frac{1}{2}}}^{2}.
\label{estimates-N}%
\end{equation}
Note that by (\ref{pressure})%
\begin{align*}
P_{ey}|_{\mathcal{S}_{e}}  &  =-g-\frac{1}{2}\frac{d}{dy}\left(  |\nabla
\psi_{e}|^{2}\right)  =-g-\left(  \psi_{ex}\psi_{exy}+\psi_{ey}\psi
_{eyy}\right) \\
&  =-g+\psi_{ey}\left(  \psi_{exx}+\eta_{ex}\psi_{exy}\right)  =-g+\psi
_{ey}\frac{d}{dx}\left(  \psi_{ex}\right)  .
\end{align*}
Denote
\begin{equation}
P_{ey}\left(  \xi\right)  =-g+\psi_{ey}\left(  \xi\right)  \tilde{a}\left(
\xi\right)  ,\ \ \tilde{a}\left(  \xi\right)  =\frac{d}{dx}\left(  \psi
_{ex}\right)  \left(  \xi\right)  . \label{defn-a}%
\end{equation}
Then $\tilde{a}\left(  \xi\right)  \ $decays exponentially when $\left\vert
\xi\right\vert \rightarrow\infty$. We have
\begin{align*}
&  \ \ \ \ \left(  b\mathcal{\tilde{C}}^{\lambda}P_{ey}\left(  \xi\right)
\mathcal{\tilde{C}}^{\lambda}u_{n},u_{n}\right) \\
&  =-g\left(  b\left(  1-\mathcal{\tilde{E}}^{\lambda,+}\right)  \frac{1}%
{\psi_{ey}}\left(  1-\mathcal{\tilde{E}}^{\lambda,+}\right)  \frac{1}%
{\psi_{ey}}u_{n},u_{n}\right) \\
&  \ \ \ \ \ +\left(  b\left(  1-\mathcal{\tilde{E}}^{\lambda,+}\right)
\tilde{a}\left(  1-\mathcal{\tilde{E}}^{\lambda,+}\right)  \frac{1}{\psi_{ey}%
}u_{n},u_{n}\right) \\
&  =T_{1}+T_{2}.
\end{align*}
Denote
\begin{equation}
\tilde{b}\left(  \xi\right)  =b-1,\ \ \tilde{c}\left(  \xi\right)  =\frac
{1}{\psi_{ey}}-\frac{1}{c}. \label{defn-b,c-tilta}%
\end{equation}
Then $\tilde{b},\tilde{c}$ tends to zero exponentially when $\left\vert
\xi\right\vert \rightarrow\infty$. The first term can be written as
\begin{align*}
T_{1}  &  =-g\left(  b\left(  1-\mathcal{\tilde{E}}^{\lambda,+}\right)
\frac{1}{\psi_{ey}}u_{n},\left(  1-\mathcal{\tilde{E}}^{\lambda,-}\right)
\frac{1}{\psi_{ey}}u_{n}\right) \\
&  =-g\left(  b\psi_{ey}\left(  1-\mathcal{\tilde{E}}^{\lambda,+}\right)
\frac{1}{\psi_{ey^{2}}}u_{n},\left(  1-\mathcal{\tilde{E}}^{\lambda,-}\right)
\frac{1}{\psi_{ey}}u_{n}\right) \\
\ \ \ \ \ \  &  \ \ \ \ \ \ \ \ -g\left(  b\psi_{ey}\left[  \frac{1}{\psi
_{ey}},1-\mathcal{\tilde{E}}^{\lambda,+}\right]  \frac{1}{\psi_{ey}}%
u_{n},\left(  1-\mathcal{\tilde{E}}^{\lambda,-}\right)  \frac{1}{\psi_{ey}%
}u_{n}\right) \\
&  =T_{1}^{1}+T_{1}^{2},
\end{align*}
where in the above we use the fact that the operator $\mathcal{\tilde{D}}$ is
anti-symmetric in the space $L_{b\psi_{ey}}^{2}$. In the rest of this paper,
we use $C$ to denote a generic constant in the estimates. By Lemma
\ref{E-t-lb-one-leq-1} and the assumption that $supp$ $u_{n}\subset\left\{
\xi|\ \left\vert \xi\right\vert \geq n\right\}  $, we have
\begin{align*}
\left\vert T_{1}^{1}\right\vert  &  \leq g\left\Vert \left(  1-\mathcal{\tilde
{E}}^{\lambda,+}\right)  \frac{1}{\psi_{ey^{2}}}u_{n}\right\Vert
_{L_{b\psi_{ey}}^{2}}\left\Vert \left(  1-\mathcal{\tilde{E}}^{\lambda
,-}\right)  \frac{1}{\psi_{ey}}u_{n}\right\Vert _{L_{b\psi_{ey}}^{2}}\\
&  \leq g\left\Vert \frac{1}{\psi_{ey^{2}}}u_{n}\right\Vert _{L_{b\psi_{ey}%
}^{2}}\left\Vert \frac{1}{\psi_{ey}}u_{n}\right\Vert _{L_{b\psi_{ey}}^{2}}\\
&  =g\left(  \int b\frac{1}{\psi_{ey^{3}}}\left\vert u_{n}\right\vert ^{2}%
d\xi\right)  ^{\frac{1}{2}}\left(  \int b\frac{1}{\psi_{ey}}\left\vert
u_{n}\right\vert ^{2}d\xi\right)  ^{\frac{1}{2}}\\
&  \leq g\left(  \frac{1}{c^{3}}+C\max_{\left\vert \xi\right\vert \geq
n}\left(  \left\vert \tilde{b}\left(  \xi\right)  \tilde{c}\left(  \xi\right)
\right\vert +\left\vert \tilde{b}\left(  \xi\right)  \right\vert +\left\vert
\tilde{c}\left(  \xi\right)  \right\vert \right)  \right)  ^{\frac{1}{2}}\\
&  \ \ \ \ \ \cdot\left(  \frac{1}{c}+C\max_{\left\vert \xi\right\vert \geq
n}\left(  \left\vert \tilde{b}\left(  \xi\right)  \tilde{c}\left(  \xi\right)
\right\vert +\left\vert \tilde{b}\left(  \xi\right)  \right\vert +\left\vert
\tilde{c}\left(  \xi\right)  \right\vert \right)  \right)  ^{\frac{1}{2}%
}\left\Vert u_{n}\right\Vert _{L^{2}}^{2}\\
&  =\frac{g}{c^{2}}+O\left(  \frac{1}{n}\right)  \text{.}%
\end{align*}
Since
\[
\left[  \frac{1}{\psi_{ey}},1-\mathcal{\tilde{E}}^{\lambda,+}\right]  =\left[
\tilde{c},\mathcal{\tilde{E}}^{\lambda,+}\right]  =\tilde{c}\mathcal{\tilde
{E}}^{\lambda,+}-\mathcal{\tilde{E}}^{\lambda,+}\tilde{c},
\]
we have
\begin{align*}
\left\vert T_{1}^{2}\right\vert  &  \leq g\left\Vert b\psi_{ey}\left[
\frac{1}{\psi_{ey}},1-\mathcal{\tilde{E}}^{\lambda,+}\right]  \frac{1}%
{\psi_{ey}}u_{n}\right\Vert _{L^{2}}\left\Vert \left(  1-\mathcal{\tilde{E}%
}^{\lambda,-}\right)  \frac{1}{\psi_{ey}}u_{n}\right\Vert _{L^{2}}\\
&  \leq C\left(  \left\Vert \tilde{c}\mathcal{\tilde{E}}^{\lambda,+}\left(
\frac{1}{\psi_{ey}}u_{n}\right)  \right\Vert _{L^{2}}+\left\Vert
\mathcal{\tilde{E}}^{\lambda,+}\left(  \tilde{c}\frac{1}{\psi_{ey}}%
u_{n}\right)  \right\Vert _{L^{2}}\right)  .
\end{align*}
Since $\tilde{c}\left(  \xi\right)  $ decays at infinity, by Lemma
\ref{lemma-estimates-E-t} (ii), for $\varepsilon>0\ $(to be fixed later),
there exists $C_{\varepsilon}$ such that
\[
\left\Vert \tilde{c}\mathcal{\tilde{E}}^{\lambda,+}u_{n}\right\Vert _{L^{2}%
}\leq\varepsilon\left\Vert u_{n}\right\Vert _{H^{\frac{1}{2}}}+C_{\varepsilon
}\left\Vert \frac{u_{n}}{1+\xi^{2}}\right\Vert _{L^{2}}\leq\varepsilon
\left\Vert u_{n}\right\Vert _{H^{\frac{1}{2}}}+\frac{C_{\varepsilon}}{n^{2}}.
\]
So
\begin{align*}
\left\Vert \tilde{c}\mathcal{\tilde{E}}^{\lambda,+}\left(  \frac{1}{\psi_{ey}%
}u_{n}\right)  \right\Vert _{L^{2}}  &  \leq\frac{1}{c}\left\Vert \tilde
{c}\mathcal{\tilde{E}}^{\lambda,+}u_{n}\right\Vert _{L^{2}}+\left\Vert
\tilde{c}\mathcal{\tilde{E}}^{\lambda,+}\left(  \tilde{c}u_{n}\right)
\right\Vert _{L^{2}}\\
&  \leq\frac{\varepsilon}{c}\left\Vert u_{n}\right\Vert _{H^{\frac{1}{2}}%
}+\frac{C_{\varepsilon}}{cn^{2}}+C\left\Vert \tilde{c}u_{n}\right\Vert
_{L^{2}}\\
&  \leq\frac{\varepsilon}{c}\left\Vert u_{n}\right\Vert _{H^{\frac{1}{2}}%
}+\frac{C_{\varepsilon}}{cn^{2}}+O\left(  \frac{1}{n}\right)  .
\end{align*}
Since%
\[
\left\Vert \mathcal{\tilde{E}}^{\lambda,+}\left(  \tilde{c}\frac{1}{\psi_{ey}%
}u_{n}\right)  \right\Vert _{L^{2}}\leq C\left\Vert \tilde{c}\frac{1}%
{\psi_{ey}}u_{n}\right\Vert _{L^{2}}=O\left(  \frac{1}{n}\right)  ,
\]
so
\[
\left\vert T_{1}^{2}\right\vert \leq C\left(  \varepsilon\left\Vert
u_{n}\right\Vert _{H^{\frac{1}{2}}}+\frac{C_{\varepsilon}}{n^{2}}+\frac{1}%
{n}\right)
\]
and thus
\[
\left\vert T_{1}\right\vert \leq\frac{g}{c^{2}}+C\left(  \varepsilon\left\Vert
u_{n}\right\Vert _{H^{\frac{1}{2}}}+\frac{C_{\varepsilon}}{n^{2}}+\frac{1}%
{n}\right)  \text{.}%
\]
The term $T_{2}$ can be written as
\begin{align*}
T_{2}  &  =\left(  b\left(  1-\mathcal{\tilde{E}}^{\lambda,+}\right)
\tilde{a}\left(  1-\mathcal{\tilde{E}}^{\lambda,+}\right)  \frac{1}{\psi_{ey}%
}u_{n},u_{n}\right) \\
&  =\left(  b\psi_{ey}\tilde{a}\left(  1-\mathcal{\tilde{E}}^{\lambda
,+}\right)  \frac{1}{\psi_{ey}}u_{n},\left(  1-\mathcal{\tilde{E}}^{\lambda
,-}\right)  \frac{1}{\psi_{ey}}u_{n}\right) \\
&  =\left(  b\psi_{ey}\left(  1-\mathcal{\tilde{E}}^{\lambda,+}\right)
\frac{\tilde{a}}{\psi_{ey}}u_{n},\left(  1-\mathcal{\tilde{E}}^{\lambda
,-}\right)  \frac{1}{\psi_{ey}}u_{n}\right) \\
&  \ \ \ \ \ +\left(  b\psi_{ey}\left[  \tilde{a},\mathcal{\tilde{E}}%
^{\lambda,+}\right]  \frac{1}{\psi_{ey}}u_{n},\left(  1-\mathcal{\tilde{E}%
}^{\lambda,-}\right)  \frac{1}{\psi_{ey}}u_{n}\right) \\
&  =T_{2}^{1}+T_{2}^{2}.
\end{align*}
Similar to the estimates for $T_{1}$, we have
\[
\left\vert T_{2}^{1}\right\vert \leq\left\Vert \frac{\tilde{a}}{\psi_{ey}%
}u_{n}\right\Vert _{L_{b\psi_{ey}}^{2}}\left\Vert \frac{1}{\psi_{ey}}%
u_{n}\right\Vert _{L_{b\psi_{ey}}^{2}}\leq C\max_{\left\vert \xi\right\vert
\geq n}\left\vert \tilde{a}\left(  \xi\right)  \right\vert =O\left(  \frac
{1}{n}\right)
\]
and
\begin{align*}
\left\vert T_{2}^{2}\right\vert  &  \leq C(\left\Vert \tilde{a}\mathcal{\tilde
{E}}^{\lambda,+}\left(  \frac{1}{\psi_{ey}}u_{n}\right)  \right\Vert _{L^{2}%
}+\left\Vert \mathcal{\tilde{E}}^{\lambda,+}\left(  \tilde{a}\frac{1}%
{\psi_{ey}}u_{n}\right)  \right\Vert _{L^{2}})\\
&  \leq C\left(  \varepsilon\left\Vert u_{n}\right\Vert _{H^{\frac{1}{2}}%
}+\frac{C_{\varepsilon}^{\prime}}{n^{2}}+\frac{1}{n}\right)  .
\end{align*}
So
\[
\left\vert T_{2}\right\vert \leq C\left(  \varepsilon\left\Vert u_{n}%
\right\Vert _{H^{\frac{1}{2}}}+\frac{C_{\varepsilon}^{\prime}}{n^{2}}+\frac
{1}{n}\right)
\]
Thus
\begin{align*}
\left\vert \operatorname{Re}\left(  b\mathcal{\tilde{C}}^{\lambda}%
P_{ey}\left(  \xi\right)  \mathcal{\tilde{C}}^{\lambda}u_{n},u_{n}\right)
\right\vert  &  \leq\left\vert T_{1}\right\vert +\left\vert T_{2}\right\vert
\\
&  \leq\frac{g}{c^{2}}+C\left(  \varepsilon\left\Vert u_{n}\right\Vert
_{H^{\frac{1}{2}}}+\frac{C_{\varepsilon}+C_{\varepsilon}^{\prime}}{n^{2}%
}+\frac{1}{n}\right)  .
\end{align*}
Combining with (\ref{estimates-N}), we have%
\begin{align*}
&  \ \ \ \ \operatorname{Re}\left(  \left(  \mathcal{A}^{\lambda}-z\right)
u_{n},u_{n}\right) \\
&  \geq\left(  1-\delta\right)  \frac{1}{h}+C_{0}\delta\left\Vert
u_{n}\right\Vert _{H^{\frac{1}{2}}}^{2}-\frac{1}{2}\delta_{0}-\frac{g}{c^{2}%
}-C\left(  \varepsilon\left\Vert u_{n}\right\Vert _{H^{\frac{1}{2}}}%
+\frac{C_{\varepsilon}+C_{\varepsilon}^{\prime}}{n^{2}}+\frac{1}{n}\right) \\
&  =\frac{1}{2}\delta_{0}-\frac{\delta}{h}+\left(  C_{0}\delta-C\varepsilon
\right)  \left\Vert u_{n}\right\Vert _{H^{\frac{1}{2}}}-C\left(
\frac{C_{\varepsilon}+C_{\varepsilon}}{n^{2}}+\frac{1}{n}\right) \\
&  \geq\frac{1}{4}\delta_{0},\ \text{when }n\text{ is large enough,}%
\end{align*}
by choosing $\varepsilon>0$ and $\delta\in\left(  0,1\right)  $ such that
$\varepsilon\leq\frac{C_{0}}{C}\delta$ and $\delta\leq\frac{1}{8}\delta_{0}h$.
This finishes the proof of the lemma.
\end{proof}

To study the essential spectrum of $\mathcal{A}^{\lambda}$, we first look at
the Zhislin Spectrum $Z\left(  \mathcal{A}^{\lambda}\right)  $
(\cite{hislop-sig-book}). A Zhislin sequence for $\mathcal{A}^{\lambda}$ and
$z\in\mathbb{C}$ is a sequence $\left\{  u_{n}\right\}  \in H^{1}$,
$\left\Vert u_{n}\right\Vert _{2}=1,\ supp\ u_{n}\subset\left\{
\xi|\ \left\vert \xi\right\vert \geq n\right\}  $ and $\left\Vert \left(
\mathcal{A}^{\lambda}-z\right)  u_{n}\right\Vert _{2}\rightarrow0$ as
$n\rightarrow\infty$. The set of all $z$ such that a Zhislin sequence exists
for $\mathcal{A}^{\lambda}$ and $z$ is denoted by $Z\left(  \mathcal{A}%
^{\lambda}\right)  $. From the above definition and Lemma
\ref{lemma-quadrtic-bound-lb}, we readily have
\begin{equation}
Z\left(  \mathcal{A}^{\lambda}\right)  \subset\left\{  z\in\mathbb{C}%
|\ \operatorname{Re}z\geq\frac{1}{2}\delta_{0}\right\}  \text{.}
\label{bound-zhislin}%
\end{equation}
Another related spectrum is the Weyl spectrum $W\left(  \mathcal{A}^{\lambda
}\right)  $ (\cite{hislop-sig-book}). A Weyl sequence for $\mathcal{A}%
^{\lambda}$ and $z\in\mathbb{C\ }$\ is a sequence $\left\{  u_{n}\right\}  \in
H^{1},\left\Vert u_{n}\right\Vert _{2}=1,\ u_{n}\rightarrow0$ weakly in
$L^{2}$ and $\left\Vert \left(  \mathcal{A}^{\lambda}-z\right)  u_{n}%
\right\Vert _{2}\rightarrow0$ as $n\rightarrow\infty$. The set $W\left(
\mathcal{A}^{\lambda}\right)  $ is all $z$ such that a Weyl sequence exists
for $\mathcal{A}^{\lambda}$ and $z$. By (\cite[Theorem 10.10]{hislop-sig-book}%
), $W\left(  \mathcal{A}^{\lambda}\right)  \subset\sigma_{\text{ess}}\left(
\mathcal{A}^{\lambda}\right)  $ and the boundary of $\sigma_{\text{ess}%
}\left(  \mathcal{A}^{\lambda}\right)  $ is contained in $W\left(
\mathcal{A}^{\lambda}\right)  $. So to prove Proposition \ref{essential-spec},
it suffices to show that $W\left(  \mathcal{A}^{\lambda}\right)  =Z\left(
\mathcal{A}^{\lambda}\right)  $. Since if this is true, then
(\ref{bound-essential}) follows from (\ref{bound-zhislin}). By (\cite[Theorem
10.12]{hislop-sig-book}), the proof of $W\left(  \mathcal{A}^{\lambda}\right)
=Z\left(  \mathcal{A}^{\lambda}\right)  $ can be reduced to prove the
following lemma.

\begin{lemma}
\label{lemma-commu-d}Given $\lambda>0$. Let $\chi\in C_{0}^{\infty}\left(
\mathbf{R}\right)  $ be a cut-off function such that $\chi|_{\left\{
\left\vert \xi\right\vert \leq R_{0}\right\}  }=1$ for some $R_{0}>0$. Define
$\chi_{d}=\chi\left(  \xi/d\right)  ,\ d>0.$ Then for each $d,\ \chi
_{d}\left(  \mathcal{A}^{\lambda}-z\right)  ^{-1}$ is compact for some
$z\in\rho\left(  \mathcal{A}^{\lambda}\right)  $, and that there exists
$\varepsilon\left(  d\right)  \rightarrow0$ as $d\rightarrow\infty$ such that
for any $u\in C_{0}^{\infty}\left(  \mathbf{R}\right)  $,
\begin{equation}
\left\Vert \left[  \mathcal{A}^{\lambda},\chi_{d}\right]  u\right\Vert
_{2}\leq\varepsilon\left(  d\right)  \left(  \left\Vert \mathcal{A}^{\lambda
}u\right\Vert _{2}+\left\Vert u\right\Vert _{2}\right)  .
\label{estimate-comm-d}%
\end{equation}

\end{lemma}

\begin{proof}
Since $\mathcal{A}^{\lambda}=\mathcal{N}+\mathcal{K}^{\lambda}$, where
$\mathcal{N}$ is positive and
\begin{equation}
\mathcal{K}^{\lambda}=b\mathcal{\tilde{C}}^{\lambda}P_{ey}\mathcal{\tilde{C}%
}^{\lambda}:L^{2}\rightarrow L^{2} \label{defn-cal-K}%
\end{equation}
is bounded, so if $z=-k$ with $k>0$ sufficiently large, then $z\in\rho\left(
\mathcal{A}^{\lambda}\right)  $. The compactness of $\chi_{d}\left(
\mathcal{A}^{\lambda}+k\right)  ^{-1}$ follows from the local compactness of
$H^{1}\hookrightarrow L^{2}$. To show (\ref{estimate-comm-d}), we note that
the graph norm of $\mathcal{A}^{\lambda}$ is equivalent to $\left\Vert
\cdot\right\Vert _{H^{1}}$. First, we write
\[
\left[  \mathcal{K}^{\lambda},\chi_{d}\right]  =b\left[  \mathcal{\tilde{C}%
}^{\lambda},\chi_{d}\right]  P_{ey}\mathcal{\tilde{C}}^{\lambda}%
+b\mathcal{\tilde{C}}^{\lambda}P_{ey}\left[  \mathcal{\tilde{C}}^{\lambda
},\chi_{d}\right]  .
\]
We have
\begin{align*}
\left[  \mathcal{\tilde{C}}^{\lambda},\chi_{d}\right]   &  =\left[  \left(
1-\frac{\mathcal{\lambda}}{\lambda+\mathcal{\tilde{D}}}\right)  \frac{1}%
{\psi_{ey}},\chi_{d}\right] \\
&  =-\left[  \frac{\mathcal{\lambda}}{\lambda+\mathcal{\tilde{D}}},\chi
_{d}\right]  \frac{1}{\psi_{ey}}=-\frac{\mathcal{\lambda}}{\lambda
+\mathcal{\tilde{D}}}\left[  \chi_{d},\mathcal{\tilde{D}}\right]  \frac
{1}{\lambda+\mathcal{\tilde{D}}}\frac{1}{\psi_{ey}}\\
&  =\frac{1}{\lambda d}\mathcal{\tilde{E}}^{\lambda,+}\left(  \frac{1}{b}%
\chi^{\prime}\left(  \xi/d\right)  \psi_{ey}\right)  \mathcal{\tilde{E}%
}^{\lambda,+}\frac{1}{\psi_{ey}}.
\end{align*}
Since $\left\Vert \mathcal{\tilde{E}}^{\lambda,+}\right\Vert _{L^{2}%
\rightarrow L^{2}}$ is bounded, so
\[
\left\Vert \left[  \mathcal{\tilde{C}}^{\lambda},\chi_{d}\right]  \right\Vert
_{L^{2}\rightarrow L^{2}}\leq\frac{C}{\lambda d}%
\]
and therefore
\begin{equation}
\left\Vert \left[  K^{\lambda},\chi_{d}\right]  u\right\Vert _{2}\leq\frac
{C}{\lambda d}\left\Vert u\right\Vert _{2}. \label{inter1}%
\end{equation}
Denote $\mathcal{N}_{1}=1+\frac{d}{d\xi}$ and $\mathcal{N}_{2}$ is the Fourier
multiplier operator with the symbol
\begin{equation}
n_{2}\left(  k\right)  =\frac{k}{\tanh\left(  kh\right)  \left(  1+ik\right)
}. \label{defn-n2}%
\end{equation}
Then $\mathcal{N}=\mathcal{N}_{2}\mathcal{N}_{1}$ and thus
\[
\left[  \mathcal{N},\chi_{d}\right]  =\mathcal{N}_{2}\left[  \mathcal{N}%
_{1},\chi_{d}\right]  +\left[  \mathcal{N}_{2},\chi_{d}\right]  \mathcal{N}%
_{1}.
\]
Since $\left[  \mathcal{N}_{1},\chi_{d}\right]  =\frac{1}{d}\chi^{\prime
}\left(  \xi/d\right)  $ and $\left\Vert \mathcal{N}_{2}\right\Vert
_{L^{2}\rightarrow L^{2}}$ is bounded, we have
\[
\left\Vert \mathcal{N}_{2}\left[  \mathcal{N}_{1},\chi_{d}\right]
u\right\Vert _{2}\leq\frac{C}{d}\left\Vert u\right\Vert _{2}\text{.}%
\]
To estimate $\left[  \mathcal{N}_{2},\chi_{d}\right]  $, for $v\in
C_{0}^{\infty}\left(  \mathbf{R}\right)  $, we follow \cite[p.127-128]{cordes}
to write
\begin{align*}
\left[  \mathcal{N}_{2},\chi_{d}\right]  v  &  =-\left(  2\pi\right)
^{-\frac{1}{2}}\int\check{n}_{2}\left(  \xi-y\right)  \left(  \chi_{d}\left(
\xi\right)  -\chi_{d}\left(  y\right)  \right)  v\left(  y\right)  dy\\
&  =-\int_{0}^{1}\int\left(  2\pi\right)  ^{-\frac{1}{2}}\left(  \xi-y\right)
\check{n}\left(  \xi-y\right)  \chi_{d}^{\prime}\left(  \rho\left(
\xi-y\right)  +y\right)  v\left(  y\right)  dyd\rho\\
&  =\int_{0}^{1}A_{\rho}v\ d\rho,
\end{align*}
where $A_{\rho}$ is the integral operator with the kernel function
\[
K_{\rho}\left(  \xi,y\right)  =-\left(  2\pi\right)  ^{-\frac{1}{2}}\left(
\xi-y\right)  \check{n}_{2}\left(  \xi-y\right)  \chi_{d}^{\prime}\left(
\rho\left(  \xi-y\right)  +y\right)  .
\]
Note that $\alpha\left(  \xi\right)  =\xi\check{n}_{2}\left(  \xi\right)  $ is
the inverse Fourier transformation of $in_{2}^{\prime}\left(  k\right)  $ and
obviously $n_{2}^{\prime}\left(  k\right)  \in L^{2}$, so $\alpha\left(
\xi\right)  \in L^{2}$. Thus
\begin{align*}
\int\int\left\vert K_{\rho}\left(  \xi,y\right)  \right\vert ^{2}dxdy  &
=2\pi\int\int\left\vert \alpha\right\vert ^{2}\left(  \xi-y\right)  \left\vert
\chi_{d}^{\prime}\right\vert ^{2}\left(  \rho\left(  \xi-y\right)  +y\right)
\ d\xi dy\\
&  =2\pi\int\int\left\vert \alpha\right\vert ^{2}\left(  \xi\right)
\left\vert \chi_{d}^{\prime}\right\vert ^{2}\left(  y\right)  \ d\xi dy\\
&  =2\pi\left\Vert \alpha\right\Vert _{L_{2}}^{2}\left\Vert \chi_{d}^{\prime
}\right\Vert _{L^{2}}^{2}=\frac{2\pi}{d}\left\Vert \alpha\right\Vert _{L_{2}%
}^{2}\left\Vert \chi^{\prime}\right\Vert _{L^{2}}^{2}.
\end{align*}
So
\[
\left\Vert \left[  \mathcal{N}_{2},\chi_{d}\right]  \right\Vert _{L^{2}%
\rightarrow L^{2}}\leq\frac{C}{d^{\frac{1}{2}}}%
\]
and
\[
\left\Vert \left[  \mathcal{N}_{2},\chi_{d}\right]  \mathcal{N}_{1}%
u\right\Vert _{L^{2}}\leq\frac{C}{d^{\frac{1}{2}}}\left\Vert \mathcal{N}%
_{1}u\right\Vert _{L^{2}}\leq\frac{C}{d^{\frac{1}{2}}}\left\Vert u\right\Vert
_{H^{1}}.
\]
Thus
\[
\left\Vert \left[  \mathcal{N},\chi_{d}\right]  u\right\Vert _{L^{2}}\leq
C\left(  \frac{1}{d^{\frac{1}{2}}}+\frac{1}{d}\right)  \left\Vert u\right\Vert
_{H^{1}}.
\]
Combining above with (\ref{inter1}), we get the estimate
(\ref{estimate-comm-d}). This finishes the proof of the lemma and thus
Proposition \ref{essential-spec}.
\end{proof}

Recall that to find growing modes, we need to find $\lambda>0$ such that
$\mathcal{A}^{\lambda}$ has a nontrivial kernel. We use a continuity argument,
by comparing the spectra of $\mathcal{A}^{\lambda}$ for $\lambda\ $near $0$
and infinity. First, we study the case near infinity.

\begin{lemma}
\label{lemma-no-eigen-infy}There exists $\Lambda>0$, such that when
$\lambda>\Lambda$, $\mathcal{A}^{\lambda}$ has no eigenvalues in $\left\{
z|\ \operatorname{Re}z\leq0\right\}  $.
\end{lemma}

\begin{proof}
Suppose otherwise, then there exists a sequence $\left\{  \lambda_{n}\right\}
\rightarrow\infty$, and $\left\{  k_{n}\right\}  \in\mathbb{C},\left\{
u_{n}\right\}  \in$ $H^{1}\left(  \mathbf{R}\right)  $, such that
$\operatorname{Re}k_{n}\leq0$ and $\left(  \mathcal{A}^{\lambda_{n}}%
-k_{n}\right)  u_{n}=0$. Since $\left\Vert \mathcal{A}^{\lambda}%
-\mathcal{N}\right\Vert =\mathcal{\ }\left\Vert \mathcal{K}^{\lambda
}\right\Vert \leq M$ for some constant $M$ independent of $\lambda$ and
$\mathcal{N}$ is a self-adjoint positive operator, all discrete eigenvalues of
$\mathcal{A}^{\lambda}$ lie in
\[
D_{M}=\left\{  z|\ \operatorname{Re}z\geq-M\text{ and }\left\vert
\operatorname{Im}z\right\vert \leq M\right\}  .
\]
Therefore, $k_{n}\rightarrow$ $k_{\infty}\in D_{M}$ with $\operatorname{Re}%
k_{\infty}\leq0$. Denote
\begin{equation}
e\left(  \xi\right)  =\max\left\{  \left\vert \tilde{a}\left(  \xi\right)
\right\vert ,\left\vert \tilde{b}\left(  \xi\right)  \right\vert ,\ \left\vert
\tilde{c}\left(  \xi\right)  \right\vert \right\}  , \label{defn-e}%
\end{equation}
where $\tilde{a}\left(  \xi\right)  ,\tilde{b}\left(  \xi\right)  ,\ \tilde
{c}\left(  \xi\right)  $ are defined in (\ref{defn-a}) and
(\ref{defn-b,c-tilta}). Then $e\left(  \xi\right)  \rightarrow0$ as
$\left\vert \xi\right\vert \rightarrow\infty$. Define the $e\left(
\xi\right)  $-weighted $L^{2}$ space $L_{e}^{2}$ with the norm
\begin{equation}
\left\Vert u\right\Vert _{L_{e}^{2}}=\left(  \int e\left(  \xi\right)
\left\vert u\right\vert ^{2}\ d\xi\right)  ^{\frac{1}{2}}. \label{defn-L2-e}%
\end{equation}
We normalize $u_{n}$ by setting $\left\Vert u_{n}\right\Vert _{L_{e}^{2}}=1$.
We claim that
\begin{equation}
\left\Vert u_{n}\right\Vert _{H^{\frac{1}{2}}}\leq C\text{, for a constant
}C\text{ independent of }n\text{.} \label{claim-bound-infy}%
\end{equation}
Assuming (\ref{claim-bound-infy}), we have $u_{n}\rightarrow u_{\infty}$
weakly in $H^{\frac{1}{2}}$. Moreover, $u_{\infty}\neq0$. To show this, we
choose $R>0$ large enough such that $\max_{\left\vert \xi\right\vert \geq
R}e\left(  \xi\right)  \leq\frac{1}{2C}$. Then
\[
\int_{\left\vert \xi\right\vert \geq R}e\left(  \xi\right)  \left\vert
u_{n}\right\vert ^{2}\ d\xi\leq\frac{1}{2C}\left\Vert u_{n}\right\Vert
_{L^{2}}\leq\frac{1}{2}.
\]
Since $u_{n}\rightarrow u_{\infty}$ strongly in $L^{2}(\left\{  \left\vert
\xi\right\vert \leq R\right\}  )$, we have
\[
\int_{\left\vert \xi\right\vert \leq R}e\left(  \xi\right)  \left\vert
u_{\infty}\right\vert ^{2}d\xi=\lim_{n\rightarrow\infty}\int_{\left\vert
\xi\right\vert \leq R}e\left(  \xi\right)  \left\vert u_{n}\right\vert
^{2}d\xi\geq\frac{1}{2}%
\]
and thus $u_{\infty}\neq0$. By Lemma \ref{lemma-C-tilta}, $\mathcal{A}%
^{\lambda_{n}}\rightarrow\mathcal{N}$ strongly in $L^{2}$, therefore
$\mathcal{A}^{\lambda_{n}}u_{n}\rightarrow$ $\mathcal{N}u_{\infty}$ weakly.
Thus $\mathcal{N}u_{\infty}=k_{\infty}u_{\infty}$. Since $\operatorname{Re}%
k_{\infty}\leq0$, this is a contradiction to that $\mathcal{N}>0$. It remains
to show (\ref{claim-bound-infy}). The proof is quite similar to that of Lemma
\ref{lemma-quadrtic-bound-lb}, so we only sketch it. From $\left(
\mathcal{A}^{\lambda_{n}}-k_{n}\right)  u_{n}=0$, we have
\begin{equation}
\left(  \mathcal{N}u_{n},u_{n}\right)  +\operatorname{Re}\left(
b\mathcal{\tilde{C}}^{\lambda_{n}}P_{ey}\left(  \xi\right)  \mathcal{\tilde
{C}}^{\lambda_{n}}u_{n},u_{n}\right)  =\operatorname{Re}k_{n}\left\Vert
u_{n}\right\Vert _{^{2}}^{2}\leq0. \label{inter2}%
\end{equation}
By Lemma \ref{lemma-operator-N},
\[
\left(  \mathcal{N}u_{n},u_{n}\right)  \geq\left(  1-\delta\right)  \frac
{1}{h}\left\Vert u_{n}\right\Vert _{L^{2}}^{2}+C_{0}\delta\left\Vert
u_{n}\right\Vert _{H^{\frac{1}{2}}}^{2}.
\]
Following the proof of Lemma \ref{lemma-quadrtic-bound-lb}, we write
\begin{align*}
&  \left(  b\mathcal{\tilde{C}}^{\lambda_{n}}P_{ey}\left(  \xi\right)
\mathcal{\tilde{C}}^{\lambda_{n}}u_{n},u_{n}\right) \\
&  =-g\left(  b\psi_{ey}\left(  1-\mathcal{\tilde{E}}^{\lambda_{n},+}\right)
\frac{1}{\psi_{ey^{2}}}u_{n},\left(  1-\mathcal{\tilde{E}}^{\lambda_{n}%
,-}\right)  \frac{1}{\psi_{ey}}u_{n}\right) \\
&  \ \ \ \ \ \ -g\left(  b\psi_{ey}\left[  \tilde{c},1-\mathcal{\tilde{E}%
}^{\lambda_{n},+}\right]  \frac{1}{\psi_{ey}}u_{n},\left(  1-\mathcal{\tilde
{E}}^{\lambda_{n},-}\right)  \frac{1}{\psi_{ey}}u_{n}\right) \\
\ \ \ \  &  \ \ \ \ \ \ \ +\left(  b\psi_{ey}\left(  1-\mathcal{\tilde{E}%
}^{\lambda_{n},+}\right)  \frac{\tilde{a}}{\psi_{ey}}u_{n},\left(
1-\mathcal{\tilde{E}}^{\lambda_{n},-}\right)  \frac{1}{\psi_{ey}}u_{n}\right)
\\
&  \ \ \ \ \ \ \ +\left(  b\psi_{ey}\left[  \tilde{a},1-\mathcal{\tilde{E}%
}^{\lambda_{n},+}\right]  \frac{1}{\psi_{ey}}u_{n},\left(  1-\mathcal{\tilde
{E}}^{\lambda_{n},-}\right)  \frac{1}{\psi_{ey}}u_{n}\right) \\
&  =T_{1}^{1}+T_{1}^{2}+T_{2}^{1}+T_{2}^{2}.
\end{align*}
The first term is estimated as
\begin{align*}
\left\vert T_{1}^{1}\right\vert  &  \leq g\left(  \int b\frac{1}{\psi_{ey^{3}%
}}\left\vert u_{n}\right\vert ^{2}dx\right)  ^{\frac{1}{2}}\left(  \int
b\frac{1}{\psi_{ey}}\left\vert u_{n}\right\vert ^{2}dx\right)  ^{\frac{1}{2}%
}\\
&  \leq g\left(  \frac{1}{c^{3}}\left\Vert u_{n}\right\Vert _{L^{2}}%
^{2}+C\left\Vert u_{n}\right\Vert _{L_{e}^{2}}^{2}\right)  ^{\frac{1}{2}%
}\left(  \frac{1}{c}\left\Vert u_{n}\right\Vert _{L^{2}}^{2}+C\left\Vert
u_{n}\right\Vert _{L_{e}^{2}}^{2}\right)  ^{\frac{1}{2}}\\
&  \leq g\left(  \frac{1}{c^{\frac{3}{2}}}\left\Vert u_{n}\right\Vert _{L^{2}%
}+C\left\Vert u_{n}\right\Vert _{L_{e}^{2}}\right)  \left(  \frac{1}%
{c^{\frac{1}{2}}}\left\Vert u_{n}\right\Vert _{L^{2}}+C\left\Vert
u_{n}\right\Vert _{L_{e}^{2}}\right) \\
&  \leq\frac{g}{c^{2}}\left\Vert u_{n}\right\Vert _{L^{2}}^{2}+C\left\Vert
u_{n}\right\Vert _{L^{2}}\left\Vert u_{n}\right\Vert _{L_{e}^{2}}+C\left\Vert
u_{n}\right\Vert _{L_{e}^{2}}^{2}\\
&  \leq\frac{g}{c^{2}}\left\Vert u_{n}\right\Vert _{L^{2}}^{2}+\varepsilon
\left\Vert u_{n}\right\Vert _{L^{2}}^{2}+C_{\varepsilon}\left\Vert
u_{n}\right\Vert _{L_{e}^{2}}^{2}.
\end{align*}
where in the second inequality, we use the fact that
\[
\left\vert b-1\right\vert ,\ \left\vert \frac{1}{\psi_{ey^{3}}}-\frac{1}%
{c^{3}}\right\vert ,\left\vert \frac{1}{\psi_{ey}}-\frac{1}{c}\right\vert \leq
Ce\left(  \xi\right)  .
\]
The second term is controlled by
\begin{align*}
\left\vert T_{1}^{2}\right\vert  &  \leq C\left(  \left\Vert \tilde{c}\left(
1-\mathcal{\tilde{E}}^{\lambda_{n},+}\right)  u_{n}\right\Vert _{L^{2}%
}+\left\Vert u_{n}\right\Vert _{L_{e}^{2}}\right)  \left\Vert u_{n}\right\Vert
_{L^{2}}\\
&  \leq C\left(  \varepsilon\left\Vert u_{n}\right\Vert _{H^{\frac{1}{2}}%
}+\left\Vert u_{n}\right\Vert _{L_{e}^{2}}\right)  \left\Vert u_{n}\right\Vert
_{L^{2}}\leq C\varepsilon\left\Vert u_{n}\right\Vert _{H^{\frac{1}{2}}}%
^{2}+C_{\varepsilon}\left\Vert u_{n}\right\Vert _{L_{e}^{2}}^{2},
\end{align*}
where in the second inequality we use Lemma \ref{lemma-estimates-E-t} (iv).
The third term is
\[
\left\vert T_{2}^{1}\right\vert \leq C\left\Vert u_{n}\right\Vert _{L_{e}^{2}%
}\left\Vert u_{n}\right\Vert _{L^{2}}\leq\varepsilon\left\Vert u_{n}%
\right\Vert _{L^{2}}^{2}+C_{\varepsilon}\left\Vert u_{n}\right\Vert
_{L_{e}^{2}}^{2}.
\]
By the same estimate as that of $T_{1}^{2}$, we have
\[
\left\vert T_{2}^{2}\right\vert \leq C\varepsilon\left\Vert u_{n}\right\Vert
_{H^{\frac{1}{2}}}^{2}+C_{\varepsilon}\left\Vert u_{n}\right\Vert _{L_{e}^{2}%
}^{2}.
\]
Plugging all of the above estimates into (\ref{inter2}), we have%
\begin{align*}
0  &  \geq\left[  \left(  1-\delta\right)  \frac{1}{h}-\frac{g}{c^{2}}\right]
\left\Vert u_{n}\right\Vert _{L^{2}}^{2}+\left(  C_{0}\delta-C\varepsilon
\right)  \left\Vert u_{n}\right\Vert _{H^{\frac{1}{2}}}^{2}-C_{\varepsilon
}\left\Vert u_{n}\right\Vert _{L_{e}^{2}}^{2}\\
&  \geq\frac{1}{2}\delta_{0}\left\Vert u_{n}\right\Vert _{L^{2}}^{2}+\frac
{1}{2}C_{0}\delta\left\Vert u_{n}\right\Vert _{H^{\frac{1}{2}}}^{2}%
-C_{\varepsilon}\left\Vert u_{n}\right\Vert _{L_{e}^{2}}^{2},
\end{align*}
by choosing $\delta,\ \varepsilon$ such that
\[
\delta=\frac{1}{2}h\delta_{0},\ \ \ \varepsilon=\frac{1}{2}\frac{C_{0}\delta
}{C}.
\]
Then (\ref{claim-bound-infy}) follows.
\end{proof}

\section{Asymptotic perturbations near zero}

In this Section, we study the eigenvalues of operator $\mathcal{A}^{\lambda}$
when $\lambda$ is very small. By Lemma \ref{lemma-C-tilta}, when
$\lambda\rightarrow0+,\ \mathcal{A}^{\lambda}\rightarrow\mathcal{A}^{0}$
strongly, where
\[
\mathcal{A}^{0}=\mathcal{N}+\frac{bP_{ey}}{\psi_{ey}^{2}}\left(  \xi\right)
.
\]
The related operator in the physical space is $\mathcal{A}_{e}^{0}%
:H^{1}\left(  \mathcal{S}_{e}\right)  \rightarrow L^{2}(\mathcal{S}_{e})$
defined by
\[
\mathcal{A}_{e}^{0}=\mathcal{N}_{e}\mathcal{\ }+\frac{P_{ey}}{\psi_{ey}^{2}%
}\left(  x\right)  =\mathcal{B}^{-1}\left(  \frac{1}{b}\mathcal{A}^{0}\right)
\mathcal{B},
\]
which is the strong limit of $\mathcal{A}_{e}^{\lambda}$ when $\lambda
\rightarrow0+$. We have the following properties of $\mathcal{A}^{0}$. We use
$\mathcal{A}^{0}\left(  \mu\right)  \ $to denote the dependence on the
solitary wave parameter $\mu.$

\begin{lemma}
\label{lemma-A0-property}

(i) The operator $\mathcal{A}^{0}:H^{1}\left(  \mathbf{R}\right)  \rightarrow
L^{2}(\mathbf{R})$ is self-adjoint and
\[
\sigma_{\text{ess}}\left(  \mathcal{A}^{0}\right)  =[\frac{1}{h}-\frac
{g}{c^{2}},+\infty).
\]

(ii) $\psi_{ex}\left(  \xi\right)  \in\ker\mathcal{A}^{0}$ and $\mathcal{A}%
^{0}$ has at least one negative eigenvalue that is simple.

(iii)Under the hypothesis (H1) of no secondary bifurcation, $\ker
\mathcal{A}^{0}\left(  \mu\right)  =$ $\left\{  \psi_{ex}\left(  \xi\right)
\right\}  $ when $\mu$ is not a turning point and $\ker\mathcal{A}^{0}\left(
\mu\right)  =$ $\left\{  \psi_{ex}\left(  \xi\right)  ,\frac{\partial\psi
_{e}\left(  \mu\right)  }{\partial\mu}\right\}  $ when $\mu$ is a turning
point. For any $\mu>\frac{\pi}{6}$, $\psi_{ex}\left(  \xi\right)  $ is the
only odd kernel of $\mathcal{A}^{0}\left(  \mu\right)  $.

(iv) When $\mu-\frac{6}{\pi}$ is small enough, $\mathcal{A}^{0}\left(
\mu\right)  $ has exactly one negative eigenvalue and $\ker\mathcal{A}%
^{0}\left(  \mu\right)  =$ $\left\{  \psi_{ex}\left(  \xi\right)  \right\}  $.
Under hypothesis (H1), the same is true for $\mathcal{A}^{0}\left(
\mu\right)  $ with $\mu\in\left(  \frac{6}{\pi},\mu_{1}\right)  $, where
$\mu_{1}$ is the first turning point.

(v) When $\mu\rightarrow\infty$, the number of negative eigenvalues of
$\mathcal{A}^{0}\left(  \mu\right)  $ increases without bound.
\end{lemma}

\begin{proof}
(i) The essential spectrum bound follows from the observations that
$\sigma_{\text{ess}}\left(  \mathcal{N}\right)  =[\frac{1}{h},+\infty)$ and
$bP_{ey}/\psi_{ey}^{2}\rightarrow-\frac{g}{c^{2}}$ when $\left\vert
\xi\right\vert \rightarrow\infty$.

Proof of (ii): To show $\psi_{ex}\left(  \xi\right)  \in\ker\mathcal{A}%
^{0},\ $it is equivalent to show that
\[
\psi_{ex}\left(  x\right)  =\psi_{ex}\left(  x,\eta_{e}(x)\right)  \in
\ker\mathcal{A}_{e}^{0}.
\]
On $\mathcal{S}_{e},$ we have
\begin{equation}
\psi_{ex}\left(  x\right)  +\eta_{ex}\psi_{ey}\left(  x\right)  =0,\ P_{ex}%
\left(  x\right)  +\eta_{ex}P_{ey}\left(  x\right)  =0 \label{relation-1}%
\end{equation}
and
\begin{align}
P_{ex}\left(  x\right)   &  =-\left(  \psi_{ex}\psi_{exx}+\psi_{ey}\psi
_{eyx}\right)  =-\psi_{ey}\left(  \eta_{ex}\psi_{eyy}+\psi_{eyx}\right)
\label{relation-2-P}\\
&  =-\psi_{ey}\frac{d}{dx}\left(  \psi_{ey}\right)  =-\psi_{ey}\frac{d}%
{dx}\left(  \phi_{ex}\right)  .\nonumber
\end{align}
So
\[
\frac{P_{ey}}{\psi_{ey}^{2}}\psi_{ex}\left(  x\right)  =-\frac{P_{ey}}%
{\psi_{ey}}\eta_{ex}=\frac{P_{ex}\left(  x\right)  }{\psi_{ey}}=-\frac{d}%
{dx}\left(  \phi_{ex}\right)  =-\mathcal{N}_{e}\left(  \psi_{ex}\left(
x\right)  \right)  ,
\]
and thus $\mathcal{A}_{e}^{0}\psi_{ex}\left(  x\right)  =0$. Now we show that
$\mathcal{A}^{0}$ has a negative eigenvalue. We note that the Fourier
multiplier operator $\mathcal{N-}_{h}^{1}$ has the same symbol as in the
Intermediate Long Wave equation (IIW), for which it was shown in
\cite{albert-et-87} that for $K>0$ large, the operator $\left(  \mathcal{N+}%
K\right)  ^{-1}$ is positivity preserving. Thus, by the spectrum theory for
positivity preserving operators (\cite{albert-et-87}), the lowest eigenvalue
of $\mathcal{A}^{0}$ is simple with the corresponding eigenfunction of one
sign. Since $\psi_{ex}\left(  x\right)  $ is odd, $\psi_{ex}\left(
\xi\right)  $ has a zero at $\xi=0$. So $0$ is not the lowest eigenvalue of
$\mathcal{A}^{0}$ and $\mathcal{A}^{0}$ has at least one simple negative eigenvalue.

To prove (iii)-(iv), first we show that the operator $\mathcal{A}^{0}\left(
\mu\right)  $ is exactly the operator $A\left(  \lambda\right)  $ introduced
by Plotnikov (\cite[p. 349]{plot91}) in the study of the bifurcation of
solitary waves. In \cite{plot91}$,$ $h$ is set to $1$ and the parameter
$\lambda=\frac{1}{F\left(  \mu\right)  ^{2}}$ is the inverse square of the
Froude number, then the operator $A\left(  \lambda\right)  $ is defined by
$\ $%
\[
A\left(  \lambda\right)  =\mathcal{N-}a,\ \ \ \ a\left(  \xi\right)
=\lambda\exp\left(  3\tau\right)  \cos\theta+\theta^{\prime}\left(
\xi\right)
\]
In the above, $\exp\left(  \tau+i\theta\right)  =W$ where $W$ is defined by
(\ref{defn-W}), as can be seen from (\cite[(4.2), p. 348]{plot91}) with $u=w$.
To show that $\mathcal{A}^{0}\left(  \mu\right)  =A\left(  \lambda\left(
\mu\right)  \right)  $, it suffices to prove that
\begin{equation}
A\left(  \lambda\right)  \psi_{ex}\left(  \xi\right)  =0. \label{kernel-plot}%
\end{equation}
Since this implies that
\[
0=\left(  A\left(  \lambda\right)  -\mathcal{A}^{0}\left(  \mu\right)
\right)  \psi_{ex}\left(  \xi\right)  =\left(  -a-bP_{ey}/\psi_{ey}%
^{2}\right)  \psi_{ex}\left(  \xi\right)
\]
and thus $bP_{ey}/\psi_{ey}^{2}=-a$. We prove (\ref{kernel-plot}) below. In
\cite{plot91}, solitary waves are shown to be critical points of the
functional%
\begin{equation}
\mathcal{J}\left(  \lambda,w\right)  =\frac{1}{2}\int_{\mathbf{R}}\left\{
w\mathcal{N}w-\lambda w^{2}\left(  1+\mathcal{N}w\right)  \right\}  d\xi.
\label{variational}%
\end{equation}
Let the self-adjoint operator $A_{0}\left(  \lambda\right)  $ to be the second
derivative of $\mathcal{J}\left(  \lambda,w\right)  $ at a solitary wave
solution. In \cite[p. 349]{plot91}$,$ the operator $A\left(  \lambda\right)  $
is defined via
\[
A\left(  \lambda\right)  =M^{\ast}A_{0}\left(  \lambda\right)  M.
\]
Here, the operator $M:L^{2}\rightarrow L^{2}$ is defined by
\begin{equation}
Mf=f\left(  1+\mathcal{C}w^{\prime}\right)  +w^{\prime}\mathcal{C}%
f=\operatorname{Re}\left\{  W\mathcal{R}f\right\}  , \label{defn-f}%
\end{equation}
where $\mathcal{C}$ is defined in Section 2 such that $\mathcal{R}f$
$=f-i\mathcal{C}f$ \ is the boundary value on $\left\{  \varsigma=0\right\}  $
of an analytic function on $D_{0}$. Our definition (\ref{defn-f}) above adapts
the notations in \cite[p. 228]{btd-1}, which studies the bifurcation of Stokes
waves by using a similar variational setting as \cite{plot91}. Taking $d/d\xi$
of the equation $\nabla_{w}\mathcal{J}\left(  \lambda,w\right)  =0$ for a
solitary wave solution $w$, we have $A_{0}\left(  \lambda\right)  w^{\prime
}=0$. Since
\begin{align*}
M^{-1}w^{\prime}  &  =\operatorname{Re}\left\{  \frac{\mathcal{R}w^{\prime}%
}{W}\right\}  =\operatorname{Re}\left\{  \frac{w^{\prime}-i\mathcal{C}%
w^{\prime}}{1+\mathcal{C}w^{\prime}+iw^{\prime}}\right\} \\
&  =\operatorname{Re}\left\{  \frac{\left(  w^{\prime}-i\mathcal{C}w^{\prime
}\right)  \left(  1+\mathcal{C}w^{\prime}-iw^{\prime}\right)  }{\left\vert
W\right\vert ^{2}}\right\} \\
&  =\frac{w^{\prime}}{\left\vert W\right\vert ^{2}}=\frac{1}{c}v_{e}=-\frac
{1}{c}\psi_{ex}\left(  \xi\right)  ,
\end{align*}
we have $M\psi_{ex}\left(  \xi\right)  =-cw^{\prime}$ and thus
\[
A\left(  \lambda\right)  \psi_{ex}\left(  \xi\right)  =-cM^{\ast}A_{0}\left(
\lambda\right)  w^{\prime}=0\text{. }%
\]
This finishes the proof that $\mathcal{A}^{0}\left(  \mu\right)  =A\left(
\lambda\right)  $.

Proof of (iii): By applying the analytic bifurcation theory in \cite{btd-1},
\cite{btd-2} to the variational setting (\ref{variational}) for the solitary
waves, one can relate the secondary bifurcation of solitary waves with the
null space of $\mathcal{A}^{0}$ (equivalently $\nabla_{ww}^{2}\mathcal{J}$%
)$.$Under the hypothesis (H1), there is no secondary bifurcation and therefore
the kernel of $\mathcal{A}^{0}$ is either due to the trivial translation
symmetry ($\psi_{ex}$) or due to the loss of monotonicity of $\lambda\left(
\mu\right)  $ at a turning point which generates an additional kernel
$\partial_{\mu}\psi_{e}$. In the Appendix, we prove that at a turning point
$\mu_{0}$, $\mathcal{A}^{0}\partial_{\mu}\psi_{e}=0$. By \cite{craige88} there
is no asymmetric bifurcation for solitary waves with $F>1$, so $\psi
_{ex}\left(  \xi\right)  $ is the only odd kernel of $\mathcal{A}^{0}\left(
\mu\right)  $.

Proof of (iv): Let $\lambda\left(  \rho\right)  =\exp\left(  -3\rho
^{2}\right)  $, then $\mu\approx\frac{6}{\pi}$ is equivalent to $\lambda
\approx1$ and thus $\rho$ is a small parameter. Consider an eigenvalue $\nu$
of $\mathcal{A}^{0}\left(  \mu\left(  \lambda\right)  \right)  $, let
$\nu=\rho^{2}\left(  3-\alpha\left(  \rho\right)  \right)  $. By using the KDV
scaling, it was shown in \cite[p. 353]{plot91} that when $\rho\rightarrow0$,
the limit $\alpha\left(  0\right)  $ is an eigenvalue of the operator
\[
B=\frac{1}{3}\frac{d^{2}}{dx^{2}}+9\operatorname{sech}^{2}\left(  \frac{3}%
{2}x\right)
\]
which has three eigenvalues $\frac{3}{4},3$ and $\frac{27}{4}$. Therefore,
when $\rho$ is small, $\mathcal{A}^{0}$ has three eigenvalues $-\frac{15}%
{4}\rho^{2}+o\left(  \rho^{2}\right)  ,\ o\left(  \rho^{2}\right)  $ and
$\frac{9}{4}\rho^{2}+o\left(  \rho^{2}\right)  $. Since $0$ is an eigenvalue
of $\mathcal{A}^{0}$, the middle one must be zero and the rest two eigenvalues
are one positive and one negative. Under the hypothesis (H1), when $\mu
<\mu_{1}$, that is, before the first turning point, we have $\ker
\mathcal{A}^{0}\left(  \mu\right)  =$ $\left\{  \psi_{ex}\left(  \xi\right)
\right\}  $. Then for all $\mu\in\left(  \frac{6}{\pi},\mu_{1}\right)  ,$ the
operator $\mathcal{A}^{0}\left(  \mu\right)  $ always has only one negative
eigenvalue. Suppose otherwise, then when $\mu$ increases from $\frac{6}{\pi}$
to $\mu_{1}$, the eigenvalues of $\mathcal{A}^{0}\left(  \mu\right)  $ must go
across zero at some $\mu=\mu^{\ast}\in$ $\left(  \frac{6}{\pi},\mu_{1}\right)
.$This implies that $dim\ker\mathcal{A}^{0}\left(  \mu^{\ast}\right)  \geq2$,
a contradiction to (H1).

Property (v) is Theorem 4.3 in \cite{plot91}$.$
\end{proof}

We note that by Lemma \ref{lemma-A0-property} (iv), there is no secondary
bifurcation for small solitary waves. Although this fact was not stated
explicitly in \cite{plot91}, it comes as a corollary of results there.

Next, we study the eigenvalues of $\mathcal{A}^{\lambda}$ for small $\lambda$.
Since the convergence of $\mathcal{A}^{\lambda}\rightarrow\mathcal{A}^{0}$ is
rather weak, we cannot use the regular perturbation theory. We use the
asymptotic perturbation theory developed by Vock and Hunziker
(\cite{vock-hunz82}), see also \cite{hislop-sig-book}, \cite{hunz-notes}.
First, we establish some preliminary lemmas.

\begin{lemma}
Given $F\in C_{0}^{\infty}\left(  \mathbf{R}\right)  $. Consider any sequence
$\lambda_{n}\rightarrow0+$ and $\left\{  u_{n}\right\}  \in H^{1}\left(
\mathbf{R}\right)  $ satisfying
\begin{equation}
\left\Vert \mathcal{A}^{\lambda_{n}}u_{n}\right\Vert _{2}+\left\Vert
u_{n}\right\Vert _{2}\leq M_{1}<\infty\label{bound-graph-0}%
\end{equation}
for some constant $M_{1}$. Then if $w-\lim_{n\rightarrow\infty}u_{n}=0$, we
have
\begin{equation}
\lim_{n\rightarrow\infty}\left\Vert Fu_{n}\right\Vert _{2}=0 \label{lim1}%
\end{equation}
and
\begin{equation}
\lim_{n\rightarrow\infty}\left\Vert \left[  \mathcal{A}^{\lambda_{n}%
},F\right]  u_{n}\right\Vert _{2}=0. \label{lim2}%
\end{equation}

\end{lemma}

\begin{proof}
Since (\ref{bound-graph-0}) implies that $\left\Vert u_{n}\right\Vert _{H^{1}%
}\leq C$, (\ref{lim1}) follows from the local compactness of $H^{1}%
\hookrightarrow L^{2}$. For the proof of (\ref{lim2}), we use the same
notations as in the proof of Lemma \ref{lemma-commu-d}. We write
$\mathcal{A}^{\lambda_{n}}=\mathcal{N+K}^{\lambda_{n}}$. Then
\[
\left[  \mathcal{N},F\right]  =\left[  \mathcal{N}_{2}\mathcal{N}%
_{1},F\right]  =\mathcal{N}_{2}\left[  \mathcal{N}_{1},F\right]  +\left[
\mathcal{N}_{2},F\right]  \mathcal{N}_{1},
\]
where $\mathcal{N}_{1}=1+\frac{d}{d\xi}$ and $\mathcal{N}_{2}$ has the symbol
$n_{2}\left(  k\right)  \ $defined by (\ref{defn-n2}). We have $\left\Vert
\left[  \mathcal{N}_{1},F\right]  u_{n}\right\Vert _{2\text{ }}=\left\Vert
F^{\prime}u_{n}\right\Vert _{2}\rightarrow0$, again by the local compactness.
Since $\frac{d}{dk}n_{2}\left(  k\right)  \rightarrow0$ when $\left\vert
k\right\vert \rightarrow\infty$, by \cite[Theorem C]{cordes} the commutator
$\left[  \mathcal{N}_{2},F\right]  :L^{2}\rightarrow L^{2}$ is compact. This
can also be seen from the proof of Lemma \ref{lemma-commu-d}, since $\left[
\mathcal{N}_{2},F\right]  =\int_{0}^{1}A_{\rho}\ d\rho$ where $A_{\rho}$ is an
integral operator with a $L^{2}\left(  \mathbf{R\times R}\right)  $ kernel$.$
Since $\left\Vert \mathcal{N}_{1}u_{n}\right\Vert _{2}\leq\left\Vert
u_{n}\right\Vert _{H^{1}}\leq C$ and $u_{n}\rightarrow0$ weakly in $L^{2}$, we
have $v_{n}=\mathcal{N}_{1}u_{n}\rightarrow0$ weakly in $L^{2}$. Therefore,
$\left\Vert \left[  \mathcal{N}_{2},F\right]  \mathcal{N}_{1}u_{n}\right\Vert
_{2}=\left\Vert \left[  \mathcal{N}_{2},F\right]  v_{n}\right\Vert
_{2}\rightarrow0$. Thus, $\left\Vert \left[  \mathcal{N},F\right]
u_{n}\right\Vert _{2\text{ }}\rightarrow0$. Since
\begin{align*}
\left[  \mathcal{K}^{\lambda_{n}},F\right]  u_{n}  &  =\left[  b\left(
1-\mathcal{\tilde{E}}^{\lambda_{n},+}\right)  \frac{P_{ey}}{\psi_{ey}}\left(
1-\mathcal{\tilde{E}}^{\lambda_{n},+}\right)  \frac{1}{\psi_{ey}},F\right] \\
&  =b\left[  F,\mathcal{\tilde{E}}^{\lambda_{n},+}\right]  \frac{P_{ey}}%
{\psi_{ey}}\left(  1-\mathcal{\tilde{E}}^{\lambda_{n},+}\right)  \frac{1}%
{\psi_{ey}}u_{n}+b\left(  1-\mathcal{\tilde{E}}^{\lambda_{n},+}\right)
\frac{P_{ey}}{\psi_{ey}}\left[  F,\mathcal{\tilde{E}}^{\lambda_{n},+}\right]
u_{n}\\
&  =p_{n}+q_{n}.
\end{align*}
For any $\varepsilon>0$, by Lemma \ref{lemma-estimates-E-t} (iii), when $n$ is
large we have
\[
\left\Vert F\mathcal{\tilde{E}}^{\lambda_{n},+}u_{n}\right\Vert _{L^{2}}%
\leq\varepsilon\left\Vert u_{n}\right\Vert _{H^{\frac{1}{2}}}\leq
\varepsilon\left\Vert u_{n}\right\Vert _{H^{1}}\leq\varepsilon M_{1}.
\]
So
\[
\left\Vert q_{n}\right\Vert _{2}\leq C\left(  \left\Vert F\mathcal{\tilde{E}%
}^{\lambda_{n},+}u_{n}\right\Vert _{2}+\left\Vert Fu_{n}\right\Vert
_{2}\right)  \leq\varepsilon CM_{1}+C\left\Vert Fu_{n}\right\Vert _{2}%
\leq2\varepsilon CM_{1},
\]
when $n$ is large. By the same proof as that of (\ref{estimate-e-t-lb}), for
any $\lambda>0,$ we have the estimate
\[
\left\Vert \mathcal{\tilde{E}}^{\lambda,+}\right\Vert _{H^{1}\rightarrow
H^{1}}\leq C\text{ (independent of }\lambda\text{). }%
\]
Denote
\[
r_{n}=\frac{P_{ey}}{\psi_{ey}}\left(  1-\mathcal{\tilde{E}}^{\lambda_{n}%
,+}\right)  \frac{1}{\psi_{ey}}u_{n},
\]
then $\left\Vert r_{n}\right\Vert _{H^{1}}\leq C\left\Vert u_{n}\right\Vert
_{H^{1}}\leq CM_{1}$. Since $w-\lim_{n\rightarrow\infty}u_{n}=0$, we have
$w-\lim_{n\rightarrow\infty}r_{n}=0$ as in the proof of Lemma
\ref{lemma-estimates-E-t} (iii). Then similar to the estimate of $\left\Vert
q_{n}\right\Vert _{2}$, we have
\[
\left\Vert p_{n}\right\Vert _{2}\leq C\left(  \left\Vert F\mathcal{\tilde{E}%
}^{\lambda_{n},+}r_{n}\right\Vert _{2}+\left\Vert Fr_{n}\right\Vert
_{2}\right)  \leq2\varepsilon CM_{1},
\]
when $n$ is large. Therefore, $\left\Vert \left[  \mathcal{K}^{\lambda_{n}%
},F\right]  u_{n}\right\Vert _{2}\leq4\varepsilon CM_{1}$ when $n$ is large
enough. Since $\varepsilon$ is arbitrary, we have $\left\Vert \left[
\mathcal{K}^{\lambda_{n}},F\right]  u_{n}\right\Vert _{2}\rightarrow0,\ $when
$n\rightarrow\infty$. This finishes the proof of (\ref{lim2}).
\end{proof}

\begin{lemma}
Let $z\in\mathbb{C}$ with $\operatorname{Re}z\leq\frac{1}{2}\delta_{0}$, then
for some $n>0$ and all $u\in C_{0}^{\infty}\left(  \left\vert x\right\vert
\geq n\right)  $, we have
\begin{equation}
\left\Vert \left(  \mathcal{A}^{\lambda}-z\right)  u\right\Vert _{2}\geq
\frac{1}{4}\delta_{0}\left\Vert u\right\Vert _{2}, \label{resolvent-bound}%
\end{equation}
when $\lambda$ is sufficiently small. Here $\delta_{0}>0\ $is defined by
(\ref{defn-d-0}).
\end{lemma}

\begin{proof}
The estimate (\ref{resolvent-bound}) follows from
\begin{equation}
\operatorname{Re}\left(  \left(  \mathcal{A}^{\lambda}-z\right)  u,u\right)
\geq\frac{1}{4}\delta_{0}\left\Vert u\right\Vert _{2}^{2}.
\label{quadratic-bound-zero}%
\end{equation}
The proof of (\ref{quadratic-bound-zero}) is almost the same as that of Lemma
\ref{lemma-quadrtic-bound-lb}, except that Lemma \ref{lemma-estimates-E-t}
(iii) is used in the estimates. So we skip it.
\end{proof}

With above two lemmas, we can use the asymptotic perturbation theory
(\cite{hislop-sig-book}, \cite{hunz-notes}) to get the following result on
eigenvalue perturbations of $\mathcal{A}^{\lambda}$.

\begin{proposition}
\label{Prop-asym-perturb}Each discrete eigenvalue $k_{0}$ of $\mathcal{A}^{0}$
with $k_{0}\leq\frac{1}{2}\delta_{0}$ is stable with respect to the family
$\mathcal{A}^{\lambda}$ in the following sense: there exists $\lambda
_{1},\delta>0$, such that for $0<\lambda<\lambda_{1}$, we have

(i)
\[
B\left(  k_{0};\delta\right)  =\left\{  z|\ 0<\left\vert z-k_{0}\right\vert
<\delta\right\}  \subset P\left(  \mathcal{A}^{\lambda}\right)  ,
\]
where
\[
P\left(  \mathcal{A}^{\lambda}\right)  =\left\{  z|\ R^{\lambda}\left(
z\right)  =\left(  \mathcal{A}^{\lambda}-z\right)  ^{-1}\text{ exists and is
uniformly bounded for }\lambda\in\left(  0,\lambda_{1}\right)  \right\}  .
\]

(ii) Denote
\[
P_{\lambda}=\oint_{\left\{  \left\vert z-k_{0}\right\vert =\delta\right\}
}R^{\lambda}\left(  z\right)  \ dz\text{ and \ }P_{0}=\oint_{\left\{
\left\vert z-k_{0}\right\vert =\delta\right\}  }R^{0}\left(  z\right)  \ dz
\]
to be the perturbed and unperturbed spectral projection. Then $\dim
P_{\lambda}=\dim P_{0}$ and $\lim_{\lambda\rightarrow0}\left\Vert P_{\lambda
}-P_{0}\right\Vert =0.$
\end{proposition}

It follows from above that for $\lambda$ small, the operators $\mathcal{A}%
^{\lambda}$ have discrete eigenvalues inside $B\left(  k_{0};\delta\right)  $
with the total algebraic multiplicity equal to that of $k_{0}$.

\section{Moving kernel and proof of main results\newline}

To study growing modes, we need to understand how the zero eigenvalue of
$\mathcal{A}^{0}$ is perturbed, in particular its moving direction. In this
Section, we derive a moving kernel formula and use it to prove the main
results. We assume hypothesis (H1) and that $\mu$ is not a turning point. Then
by Lemma \ref{lemma-A0-property} (iii), $\ker\mathcal{A}^{0}\left(
\mu\right)  =$ $\left\{  \psi_{ex}\left(  \xi\right)  \right\}  $. Let
$\lambda_{1},\delta>0$ be as given in Proposition \ref{Prop-asym-perturb} for
$k_{0}=0$. Since $\dim P_{\lambda}=\dim P_{0}=1$, when $\lambda<\lambda_{1}$
there is only one real eigenvalue of $\mathcal{A}^{\lambda}$ inside $B\left(
0;\delta\right)  $, which we denote by $k_{\lambda}\in\mathbf{R}$. The
following lemma determines the sign of $k_{\lambda}$ when $\lambda$ is
sufficiently small.

\begin{lemma}
\label{lemma-moving}Assume hypothesis (H1) and that $\mu$ is not a turning
point. For $\lambda>0$ small enough, let $k_{\lambda}\in\mathbf{R}$ to be the
eigenvalue of $\mathcal{A}^{\lambda}$ near zero. Then$\ $
\begin{equation}
\lim_{\lambda\rightarrow0+}\frac{k_{\lambda}}{\lambda^{2}}=-\frac{1}{c}%
\frac{dE}{dc}/\left\Vert \psi_{ex}\right\Vert _{L^{2}}^{2},
\label{formula-moving}%
\end{equation}
where $E\left(  \mu\right)  $ is the total energy defined in (\ref{energy})$.$
\end{lemma}

The following a priori estimate is used in the proof.

\begin{lemma}
\label{lemma-apriori}For $\lambda>0$ small enough, consider $u\in H^{1}\left(
\mathbf{R}\right)  $ satisfying the equation $\left(  \mathcal{A}^{\lambda
}-z\right)  u=v$, where $z\in\mathbb{C}$ with $\operatorname{Re}z\leq\frac
{1}{2}\delta_{0}$ and $v\in L^{2}$. Then we have
\begin{equation}
\left\Vert u\right\Vert _{H^{\frac{1}{2}}}\leq C\left(  \left\Vert
u\right\Vert _{L_{e}^{2}}+\left\Vert v\right\Vert _{L^{2}}\right)  ,
\label{apriori-small-lb}%
\end{equation}
for some constant $C$ independent of $\lambda$. Here, the norm $\left\Vert
\cdot\right\Vert _{L_{e}^{2}}$ is defined in (\ref{defn-L2-e}) with the weight
$e\left(  \xi\right)  $ defined by (\ref{defn-e}).
\end{lemma}

\begin{proof}
The proof is almost the same as that of the estimate (\ref{claim-bound-infy})
in the proof of Lemma \ref{lemma-no-eigen-infy}. So we only sketch it. We
have
\[
\left(  \mathcal{N}u,u\right)  +\operatorname{Re}\left(  b\mathcal{\tilde{C}%
}^{\lambda}P_{ey}\left(  \xi\right)  \mathcal{\tilde{C}}^{\lambda}u,u\right)
-\operatorname{Re}z\left\Vert u\right\Vert _{^{2}}^{2}=\operatorname{Re}%
\left(  u,v\right)  .
\]
By the same estimates as in proving (\ref{claim-bound-infy}), except that
Lemma \ref{lemma-estimates-E-t} (iii) is used, we have
\begin{align*}
\  &  \left[  \left(  1-\delta\right)  \frac{1}{h}-\frac{g}{c^{2}}-\frac{1}%
{2}\delta_{0}\right]  \left\Vert u\right\Vert _{L^{2}}^{2}\ +\left(
C_{0}\delta-C\varepsilon\right)  \left\Vert u\right\Vert _{H^{\frac{1}{2}}%
}^{2}-C_{\varepsilon}\left\Vert u\right\Vert _{L_{e}^{2}}^{2}\\
&  \leq\varepsilon\left\Vert u\right\Vert _{L^{2}}^{2}+\frac{1}{\varepsilon
}\left\Vert v\right\Vert _{L^{2}}^{2}.
\end{align*}
Then the estimate (\ref{apriori-small-lb}) follows by choosing $\varepsilon>0$
and $\delta\in\left(  0,1\right)  \ $properly.
\end{proof}

Assuming Lemma \ref{lemma-moving}, we prove Theorem \ref{thm-main}.

\begin{proof}
[Proof of Theorem \ref{thm-main}]We fixed $\mu\in\left(  \tilde{\mu}_{1}%
,\mu_{1}\right)  $. Under the assumption (H1), it follows from Lemma
\ref{lemma-A0-property} that $\mathcal{A}^{0}\left(  \mu\right)  $ has only
one negative eigenvalue $k_{0}^{-}<0\ $and $\ker\mathcal{A}^{0}\left(
\mu\right)  =$ $\left\{  \psi_{ex}\left(  \xi\right)  \right\}  $. By
Proposition \ref{Prop-asym-perturb} and Lemma \ref{lemma-moving}, there
exists$\ \lambda_{1},\delta>0\ $small enough, such that for $0<\lambda
<\lambda_{1}$, $\mathcal{A}^{\lambda}$ has one negative eigenvalue
$k_{\lambda}^{-}\ $in $B\left(  k_{0}^{-};\delta\right)  \ $with multiplicity
$1$ and one positive eigenvalue $k_{\lambda}\ $in $B\left(  0;\delta\right)  $
because $\frac{dE}{dc}=E^{\prime}\left(  \mu\right)  /c^{\prime}\left(
\mu\right)  <0$ for $\mu\in\left(  \tilde{\mu}_{1},\mu_{1}\right)  $. Consider
the region
\[
\Omega=\left\{  z|\ 0>\operatorname{Re}z>-2M\text{ and }\left\vert
\operatorname{Im}z\right\vert <2M\right\}  ,
\]
where $M$ is the uniform bound of $\left\Vert \mathcal{A}^{\lambda
}-\mathcal{N}\right\Vert $. We claim that: for $\lambda$ small enough,
$\mathcal{A}^{\lambda}$ has exactly $2$ eigenvalues (counting multiplicity)
in
\[
\Omega_{\delta}=\left\{  z|\ 2\delta>\operatorname{Re}z>-2M\text{ and
}\left\vert \operatorname{Im}z\right\vert <2M\right\}  .
\]
That is, all eigenvalues of $\mathcal{A}^{\lambda}$ with real parts no greater
than $2\delta$ lie$\ $in $B\left(  k_{0}^{-};\delta\right)  \cup B\left(
0;\delta\right)  $. Suppose otherwise, there exists a sequence $\lambda
_{n}\rightarrow0+$ and
\[
\left\{  u_{n}\right\}  \in H^{1}\left(  \mathbf{R}\right)  ,z_{n}\in
\Omega/\left(  B\left(  k_{0}^{-};\delta\right)  \cup B\left(  0;\delta
\right)  \right)
\]
such that $\left(  \mathcal{A}^{\lambda_{n}}-z_{n}\right)  u_{n}=0$. We
normalize $u_{n}$ by setting $\left\Vert u_{n}\right\Vert _{L_{e}^{2}}=1$.
Then by Lemma \ref{lemma-apriori}, we have $\left\Vert u_{n}\right\Vert
_{H^{\frac{1}{2}}}\leq C$. By the same argument as in the proof of Lemma
\ref{lemma-no-eigen-infy}, $u_{n}\rightarrow u_{\infty}\neq0$ weakly in
$H^{\frac{1}{2}}$. Let
\[
z_{n}\rightarrow z_{\infty}\in\bar{\Omega}/\left(  B\left(  k_{0}^{-}%
;\delta\right)  \cup B\left(  0;\delta\right)  \right)  ,
\]
then $\mathcal{A}^{0}u_{\infty}=z_{\infty}u_{\infty}$ which is a
contradiction. This proves the claim. Thus for $\lambda$ small enough,
$\mathcal{A}^{\lambda}$ has exactly one eigenvalue in $\Omega$.

Suppose the conclusion of Theorem \ref{thm-main} does not hold, then
$\mathcal{A}^{\lambda}\left(  \mu\right)  $ has no kernel for any $\lambda>0$.
Define $n_{\Omega}\left(  \lambda\right)  $ to be the number of eigenvalues
(counting multiplicity) of $\mathcal{A}^{\lambda}\ $in $\Omega$. By
(\ref{bound-essential}), the region $\Omega\ $is away from the essential
spectrum of $\mathcal{A}^{\lambda}$, so $n_{\Omega}\left(  \lambda\right)  $
is a finite integer. For $\lambda$ small enough, we have proved that
$n_{\Omega}\left(  \lambda\right)  =1$. By Lemma \ref{lemma-no-eigen-infy},
$n_{\Omega}\left(  \lambda\right)  =0$ for $\lambda>\Lambda$. Define the two
sets
\[
S_{\text{odd}}=\left\{  \lambda>0|\text{ }n_{\Omega}\left(  \lambda\right)
\text{ is odd}\right\}  \text{ and }\ S_{\text{even}}=\left\{  \lambda
>0|\text{ }n_{\Omega}\left(  \lambda\right)  \text{ is even}\right\}  .
\]
Then both sets are non-empty. Below, we show that both $S_{\text{odd}}$ and
$S_{\text{even}}$ are open. Let $\lambda_{0}\in$ $S_{\text{odd}}$ and denote
$k_{1},\cdots,k_{l}$ $\left(  l\leq n_{\Omega}\left(  \lambda_{0}\right)
\right)  \ $to be all distinct eigenvalues of $\mathcal{A}^{\lambda_{0}}$ in
$\Omega$. Denote $ih_{1},\cdots,ih_{m}$ to be all eigenvalues of
$\mathcal{A}^{\lambda_{0}}$ on the imaginary axis. Then $\left\vert
h_{j}\right\vert \leq M$, $1\leq j\leq m$. Choose $\delta>0$ sufficiently
small such that the disks $B\left(  k_{i};\delta\right)  $ $\left(  1\leq
i\leq l\right)  $ and $B\left(  ih_{j};\delta\right)  \ \left(  1\leq j\leq
m\right)  $ are disjoint, $B\left(  k_{i};\delta\right)  \subset\Omega$ and
$B\left(  ih_{j};\delta\right)  $ does not contain $0$. Note that
$\mathcal{A}^{\lambda}$ depends on $\lambda$ analytically in $\left(
0,+\infty\right)  $. By the analytic perturbation theory
(\cite{hislop-sig-book}), if $\left\vert \lambda-\lambda_{0}\right\vert $ is
sufficiently small, any eigenvalue of $\mathcal{A}^{\lambda}$ in
$\Omega_{\delta}$ lies in one of the disks $B\left(  k_{i};\delta\right)  $ or
$B\left(  ih_{j};\delta\right)  $. So $n_{\Omega}\left(  \lambda\right)  $ is
$n_{\Omega}\left(  \lambda_{0}\right)  $ plus the number of eigenvalues in
$\cup_{i=1}^{m}B\left(  ih_{j};\delta\right)  $ with the negative real part.
The later number must be even, since the complex eigenvalues of $\mathcal{A}%
^{\lambda}$ appears in conjugate pairs. Thus, $n_{\Omega}\left(
\lambda\right)  $ is odd for $\left\vert \lambda-\lambda_{0}\right\vert $
small enough. This shows that $S_{\text{odd}}$ is open. For the same reason,
$S_{\text{even}}$ is open$.$Thus, $\left(  0,+\infty\right)  $ is the union of
two non-empty, disjoint open sets $S_{\text{odd}}$ and $S_{\text{even}}$. A contradiction.

So there exists $\lambda>0$ and $0\neq u\in$ $H^{1}\left(  \mathbf{R}\right)
$ such that $\mathcal{A}^{\lambda}u=0$. Let $f=\mathcal{B}^{-1}u\in
H^{1}\left(  \mathcal{S}_{e}\right)  $, then $\mathcal{A}_{e}^{\lambda}f=0$.
Define $\eta\left(  x\right)  =\mathcal{C}^{\lambda}f,\ P(x)=-P_{ey}%
\eta\left(  x\right)  \ $and $\psi\left(  x,y\right)  $ to be the solution of
the Dirichlet problem
\[
\Delta\psi=0\ \ \ \text{in}\ \ \mathcal{D}_{e},\ \psi|_{\mathcal{S}_{e}%
}=f\ ,\ \psi(x,-h)=0.
\]
Then $\left(  \eta\left(  x\right)  ,\psi\left(  x,y\right)  \right)  $
satisfies the system (\ref{eqn-G-phi})-(\ref{eqn-G-bottom}), thus $e^{\lambda
t}\left[  \eta\left(  x\right)  ,\psi\left(  x,y\right)  \right]  $ is a
growing mode solution to the linearized problem (\ref{eqn-L}). Below we prove
the regularity of $\left[  \eta\left(  x\right)  ,\psi\left(  x,y\right)
\right]  $. Since the operator $\mathcal{C}^{\lambda}$ is regularity
preserving, from $\mathcal{A}_{e}^{\lambda}f=0$ we have
\[
\psi_{n}(x)=-\mathcal{C}^{\lambda}P_{ey}\mathcal{C}^{\lambda}f\in H^{1}\left(
\mathcal{S}_{e}\right)  .
\]
By the elliptic regularity of Neumann problems (\cite{ag-book}), we have
$\psi\left(  x,y\right)  \in H^{5/2}\left(  \mathcal{D}_{e}\right)  .$ So by
the trace theorem, $f=\psi|_{\mathcal{S}_{e}}\in H^{2}\left(  \mathcal{S}%
_{e}\right)  $. Repeating this process, we get $\psi_{n}(x)\in H^{2}\left(
\mathcal{S}_{e}\right)  $ and $\psi\left(  x,y\right)  \in H^{7/2}\left(
\mathcal{D}_{e}\right)  .$ Since the irrotational solitary wave profile and
the boundary $\mathcal{S}_{e}$ are analytic (\cite{lewy}), we can repeat the
above process to show $\psi\left(  x,y\right)  \in H^{k}\left(  \mathcal{D}%
_{e}\right)  $ for any $k>0$. Therefore $\eta\left(  x\right)  =\mathcal{C}%
^{\lambda}\left(  \psi|_{\mathcal{S}_{e}}\right)  \in H^{k}\left(
\mathcal{S}_{e}\right)  $, for any$\ k>0$. By Sobolev embedding, $\left[
\eta\left(  x\right)  ,\psi\left(  x,y\right)  \right]  \in C^{\infty}$. This
finishes the proof of Theorem \ref{thm-main}.
\end{proof}

\begin{proof}
[Proof of Theorem \ref{thm-minor}]Let $\mu_{1}<\mu_{2},\cdots<\mu_{n}<\cdots$
be all the turning points. Then $\mu_{n}\rightarrow+\infty.$ Under the
assumption (H1), $\ker\mathcal{A}^{0}\left(  \mu\right)  =$ $\left\{
\psi_{ex}\left(  \xi\right)  \right\}  $, for $\mu\in\left(  \mu_{i},\mu
_{i+1}\right)  ,$ $i\geq1$. Denote by $n^{-}\left(  \mu\right)  \ $the number
of negative eigenvalues of $\mathcal{A}^{0}\left(  \mu\right)  $. Then
$n^{-}\left(  \mu\right)  $ is a constant in $\left(  \mu_{i},\mu
_{i+1}\right)  $, by the same argument in the proof of Lemma
\ref{lemma-A0-property} (iv). Denote $\tilde{\mu}_{1}<\tilde{\mu}_{2}%
,\cdots<\tilde{\mu}_{n}<\cdots$ to be all the critical points of $E\left(
\mu\right)  $. Each $\tilde{\mu}_{k}$ lies in some interval $\left(  \mu
_{i},\mu_{i+1}\right)  $. Then the sign of $dE/dc=E^{\prime}\left(
\mu\right)  /c^{\prime}\left(  \mu\right)  $ changes at $\tilde{\mu}_{k}$ in
$\left(  \mu_{i},\mu_{i+1}\right)  $. So we can find an interval $I_{k}$
$\subset\left(  \mu_{i},\mu_{i+1}\right)  $ such that the number
\[
\tilde{n}^{-}\left(  \mu\right)  =n^{-}\left(  \mu\right)  +\left(
1+sign\left(  E^{\prime}\left(  \mu\right)  /c^{\prime}\left(  \mu\right)
\right)  \right)  /2
\]
is odd for $\mu\in I_{k}$. Note that $\tilde{n}^{-}\left(  \mu\right)  $ is
the number of eigenvalues of $\mathcal{A}^{\lambda}\left(  \mu\right)  $ in
the left half plane, for $\lambda$ sufficiently small. So by the same proof as
that of Theorem \ref{thm-main}, we get a purely growing mode for solitary
waves with $\mu\in I_{k}$. Since $\tilde{\mu}_{n}\rightarrow\infty,$ the
intervals $I_{n}$ goes to infinity.
\end{proof}

We make two remarks about the unstable solitary waves proved above.

\begin{remark}
\label{remark-appx}In terms of the parameter $\omega=1-F^{2}q_{c}^{2}$, it was
found (\cite{tanaka86}) from numerical computations that the energy maximum is
achieved at $\omega\approx$ $0.88$, which corresponds to the
amplitude-to-depth ratio $\alpha=\eta_{e}\left(  0\right)  /h=$ $0.7824$
(\cite{long-tanaka}). The highest wave has the parameters $\omega=1$,
$\alpha=$ $0.8332$ and the maximal travel speed is achieved at $\omega
=0.917,\alpha=0.790$ (\cite{tanaka86}, \cite{long-fenton}). So the unstable
waves proved in Theorem \ref{thm-main} and \ref{thm-minor} are of large
amplitude, and their height is comparable to the water depth. Therefore, this
type of instability cannot be captured in the approximate models based on
small amplitude assumptions. Indeed, although some approximate models share
certain features of the full water wave model, no unstable solitary waves have
been found. For example, the \textit{Green-Naghdi model is proposed to model
water waves of larger amplitude and it also has an in}definite energy
functional, but the numerical computation (\textit{\cite[P. 529]{li-y})
}indicates that all the G-N solitary waves are spectrally stable. For
Camassa-Holm and \textit{Degasperis-Procesi equations, there exist cornered
solitary waves (peakons). However, these peakons are shown to be nonlinearly
stable (\cite{cs-peakon}, \cite{lin-dp}). So the instability of large solitary
waves seems to be a particular feature of the full water wave model. }
\end{remark}

\textit{ }

\begin{remark}
\label{remark-break}The linear instability suggests that the solitary wave
cannot preserve its shape for all the time. The long time evolution around an
unstable wave was studied numerically in \cite{tanaka-dold-pere87}. It was
found that small perturbations with the same amplitude but opposite signs can
lead to totally different long time behaviors. For one sign, the perturbed
wave breaks quickly and for the other sign, the perturbed wave never breaks
and it finally approaches a slightly lower stable solitary wave with almost
the same energy. Note that in the breaking case, the initial perturbed profile
has a rather negative slope (\cite{tanaka-dold-pere87}). The wave breaking for
shallow water waves models such as Camassa-Holm (\cite{con-esher}) and Whitham
equations (\cite{siegler}, \cite{naumkim}) is due to the initial large
negative slope. It would be interesting to clarify whether or not the wave
breaking found in \cite{tanaka-dold-pere87} has the same mechanism.

The wave breaking due to the instability of large solitary waves had been used
to explain the breaking waves approaching beaches (\cite{duncan},
\cite{peregrine}, \cite{peregrine2}). When a wave approaches the beach, the
amplitude-to-depth ratio can increase to be near the critical ratio
($\approx0.7824$) for instability and consequently the wave breaking can occur.
\end{remark}

It remains to prove the moving kernel formula (\ref{formula-moving}).

\begin{proof}
[Proof of Lemma \ref{lemma-moving}]As described at the beginning of this
Section, for $\lambda>0$ small enough, there exists $u_{\lambda}\in
H^{1}\left(  \mathbf{R}\right)  $, such that $\left(  \mathcal{A}^{\lambda
}-k_{\lambda}\right)  u_{\lambda}=0$ with $k_{\lambda}\in\mathbf{R}$ and
$\lim_{\lambda\rightarrow0+}k_{\lambda}=0$. We normalize $u_{\lambda}$ by
$\left\Vert u_{\lambda}\right\Vert _{L_{e}^{2}}=1$. Then by Lemma
\ref{lemma-apriori}, we have $\left\Vert u_{\lambda}\right\Vert _{H^{\frac
{1}{2}}}\leq C$ and as in the proof of Lemma \ref{lemma-no-eigen-infy},
$u_{\lambda}\rightarrow u_{0}\neq0$ weakly in $H^{\frac{1}{2}}$. Since
$\mathcal{A}^{0}u_{0}=0$ and $\ker\mathcal{A}^{0}\left(  \mu\right)  =$
$\left\{  \psi_{ex}\left(  \xi\right)  \right\}  $, we have $u_{0}=c_{0}%
\psi_{ex}\left(  \xi\right)  $ for some $c_{0}\neq0$. Moreover, we have
$\left\Vert u_{\lambda}-u_{0}\right\Vert _{H^{\frac{1}{2}}}=0$. To show this,
first we note that $\left\Vert u_{\lambda}-u_{0}\right\Vert _{L_{e}^{2}%
}\rightarrow0$, since
\[
\left\Vert u_{\lambda}-u_{0}\right\Vert _{L_{e}^{2}}^{2}\leq\int_{\left\vert
\xi\right\vert \leq R}e\left(  \xi\right)  \left\vert u_{\lambda}%
-u_{0}\right\vert ^{2}\ d\xi+\max_{\left\vert \xi\right\vert \geq R}e\left(
\xi\right)  \left\Vert u_{\lambda}-u_{0}\right\Vert _{L^{2}}^{2},
\]
and the second term is arbitrarily small for large $R$ while the first term
tends to zero by the local compactness. Since
\[
\left(  \mathcal{A}^{\lambda}-k_{\lambda}\right)  \left(  u_{\lambda}%
-u_{0}\right)  =k_{\lambda}u_{0}+\left(  \mathcal{A}^{0}-\mathcal{A}^{\lambda
}\right)  u_{0},
\]
by Lemma \ref{lemma-apriori} we have
\[
\left\Vert u_{\lambda}-u_{0}\right\Vert _{H^{\frac{1}{2}}}\leq C\left(
\left\Vert u_{\lambda}-u_{0}\right\Vert _{L_{e}^{2}}^{2}+\left\vert
k_{\lambda}\right\vert \left\Vert u_{\lambda}\right\Vert _{L^{2}}%
^{2}+\left\Vert \left(  \mathcal{A}^{0}-\mathcal{A}^{\lambda}\right)
u_{0}\right\Vert _{L^{2}}^{2}\right)  \rightarrow0,
\]
when $\lambda\rightarrow0+$. We can set $c_{0}=1$ by renormalizing the sequence.

Next, we show that $\lim_{\lambda\rightarrow0+}\frac{k_{\lambda}}{\lambda}=0$.
From $\left(  \mathcal{A}^{\lambda}-k_{\lambda}\right)  u_{\lambda}=0,$ we
have
\begin{equation}
\mathcal{A}^{0}\frac{u_{\lambda}}{\lambda}+\frac{\mathcal{A}^{\lambda
}-\mathcal{A}^{0}}{\lambda}u_{\lambda}=\frac{k_{\lambda}}{\lambda}u_{\lambda}.
\label{eqn-1st}%
\end{equation}
Taking the inner product of above with $\psi_{ex}\left(  \xi\right)  $, we
have
\[
\frac{k_{\lambda}}{\lambda}\left(  u_{\lambda},\psi_{ex}\left(  \xi\right)
\right)  =\left(  \frac{\mathcal{A}^{\lambda}-\mathcal{A}^{0}}{\lambda
}u_{\lambda},\psi_{ex}\left(  \xi\right)  \right)  :=m\left(  \lambda\right)
.
\]
We compute the integral $m\left(  \lambda\right)  $ in the physical space, by
the change of variable $\xi\rightarrow x$. Denote by $\left\langle
\ ,\right\rangle $ the inner product in $L^{2}\left(  \mathcal{S}_{e}\right)
$. Noting that $dx=b\left(  \xi\right)  \ d\xi$, we have
\begin{align*}
m\left(  \lambda\right)   &  =\left\langle \frac{\mathcal{A}_{e}^{\lambda
}-\mathcal{A}_{e}^{0}}{\lambda}u_{\lambda}\left(  x\right)  ,\psi_{ex}\left(
x\right)  \right\rangle \\
&  =-\left\langle \frac{1}{\lambda+\mathcal{D}}\frac{P_{ey}}{\psi_{ey}^{2}%
}u_{\lambda},\psi_{ex}\right\rangle -\left\langle \frac{P_{ey}}{\psi_{ey}%
}\frac{1}{\lambda+\mathcal{D}}\frac{1}{\psi_{ey}}u_{\lambda},\psi
_{ex}\right\rangle \\
&  \,\ \ \ \ +\left\langle \frac{1}{\lambda+\mathcal{D}}\frac{P_{ey}}%
{\psi_{ey}}\mathcal{E}^{\lambda,+}\frac{1}{\psi_{ey}}u_{\lambda},\psi
_{ex}\right\rangle \\
&  =m_{1}+m_{1}+m_{3},
\end{align*}
where we use
\begin{align}
\frac{\mathcal{A}_{e}^{\lambda}-\mathcal{A}_{e}^{0}}{\lambda}  &  =\frac
{1}{\lambda}\left(  \mathcal{C}^{\lambda}P_{ey}\left(  x\right)
\mathcal{C}^{\lambda}-\frac{P_{ey}}{\psi_{ey}^{2}}\right)
\label{exp-differnece-A}\\
&  =\frac{1}{\lambda}\left(  \left(  1-\frac{\mathcal{\lambda}}{\lambda
+\mathcal{D}}\right)  \frac{1}{\psi_{ey}}P_{ey}\left(  x\right)  \left(
1-\frac{\mathcal{\lambda}}{\lambda+\mathcal{D}}\right)  \frac{1}{\psi_{ey}%
}-\frac{P_{ey}}{\psi_{ey}^{2}}\right) \nonumber\\
&  =-\frac{1}{\lambda+\mathcal{D}}\frac{P_{ey}}{\psi_{ey}^{2}}-\frac{P_{ey}%
}{\psi_{ey}}\frac{1}{\lambda+\mathcal{D}}\frac{1}{\psi_{ey}}+\frac{1}%
{\lambda+\mathcal{D}}\frac{P_{ey}}{\psi_{ey}}\mathcal{E}^{\lambda,+}\frac
{1}{\psi_{ey}}.\nonumber
\end{align}
We compute each term separately. For the first term, when $\lambda
\rightarrow0+,$
\begin{align*}
m_{1}\left(  \lambda\right)   &  =\left\langle \frac{1}{\lambda+\mathcal{D}%
}\frac{P_{ey}}{\psi_{ey}^{2}}u_{\lambda},\psi_{ey}\eta_{ex}\right\rangle
=\left\langle \frac{P_{ey}}{\psi_{ey}}u_{\lambda},\frac{1}{\lambda
-\mathcal{D}}\eta_{ex}\right\rangle \\
&  =\left\langle \frac{P_{ey}}{\psi_{ey}}u_{\lambda},\left(  \mathcal{E}%
^{\lambda,-}-1\right)  \frac{1}{\psi_{ey}}\eta_{e}\right\rangle \rightarrow
\left\langle \frac{P_{ey}}{\psi_{ey}}\psi_{ex},\frac{1}{\psi_{ey}}\eta
_{e}\right\rangle =0,\text{ }%
\end{align*}
where we use Lemma \ref{lemma-e-lb} (b) in the above and the resultant
integral is zero because $P_{ey},\psi_{ey},\eta_{e}$ are even and $\psi_{ex}$
is odd in $x$. The second term is
\begin{align*}
m_{2}\left(  \lambda\right)   &  =\left\langle \frac{P_{ey}}{\psi_{ey}}%
\frac{1}{\lambda+\mathcal{D}}\frac{1}{\psi_{ey}}u_{\lambda},\psi_{ey}\eta
_{ex}\right\rangle =\left\langle \frac{1}{\lambda+\mathcal{D}}\frac{1}%
{\psi_{ey}}u_{\lambda},\psi_{ey}\frac{d}{dx}\left(  \phi_{ex}\right)
\right\rangle \\
&  =\left\langle u_{\lambda},\frac{1}{\lambda-\mathcal{D}}\frac{d}{dx}\left(
u_{e}\right)  \right\rangle =\left\langle u_{\lambda},\left(  \mathcal{E}%
^{\lambda,-}-1\right)  \frac{1}{\psi_{ey}}u_{e}\right\rangle \\
&  \rightarrow-\left\langle \psi_{ex},\frac{1}{\psi_{ey}}u_{e}\right\rangle
=0,
\end{align*}
where the relations (\ref{relation-1}) and (\ref{relation-2-P}) are used in
the above computation. For the last term,
\begin{align*}
m_{3}\left(  \lambda\right)   &  =-\left\langle \frac{1}{\lambda+\mathcal{D}%
}\frac{P_{ey}}{\psi_{ey}}\mathcal{E}^{\lambda,+}\frac{1}{\psi_{ey}}u_{\lambda
},\psi_{ey}\eta_{ex}\right\rangle =-\left\langle P_{ey}\mathcal{E}^{\lambda
,+}\frac{1}{\psi_{ey}}u_{\lambda},\frac{1}{\lambda-\mathcal{D}}\eta
_{ex}\right\rangle \\
&  =\left\langle P_{ey}\mathcal{E}^{\lambda,+}\frac{1}{\psi_{ey}}u_{\lambda
},\left(  \mathcal{E}^{\lambda,-}-1\right)  \frac{1}{\psi_{ey}}\eta
_{e}\right\rangle \rightarrow0,
\end{align*}
because $\mathcal{E}^{\lambda,+}\frac{1}{\psi_{ey}}u_{\lambda}\rightarrow0$
weakly in $L^{2}$, when $\lambda\rightarrow0+$. So $m\left(  \lambda\right)
\rightarrow0$, and thus
\[
\lim_{\lambda\rightarrow0+}\frac{k_{\lambda}}{\lambda}=\lim_{\lambda
\rightarrow0+}\frac{m\left(  \lambda\right)  }{\left(  u_{\lambda},\psi
_{ex}\left(  \xi\right)  \right)  }=0.
\]

Now we write $u_{\lambda}=c_{\lambda}\psi_{ex}+\lambda v_{\lambda}$, with
$c_{\lambda}=\left(  u_{\lambda},\psi_{ex}\right)  /\left(  \psi_{ex}%
,\psi_{ex}\right)  $. Then $\left(  v_{\lambda},\psi_{ex}\right)  =0$ and
$c_{\lambda}\rightarrow1$ as $\lambda\rightarrow0+$. We claim that:
$\left\Vert v_{\lambda}\right\Vert _{L_{e}^{2}}\leq C$ (independent of
$\lambda$). Suppose otherwise, there exists a sequence $\lambda_{n}%
\rightarrow0+$ such that $\left\Vert v_{\lambda_{n}}\right\Vert _{L_{e}^{2}%
}\geq n$. Denote $\tilde{v}_{\lambda_{n}}=v_{\lambda_{n}}/\left\Vert
v_{\lambda_{n}}\right\Vert _{L_{e}^{2}}$. Then $\left\Vert \tilde{v}%
_{\lambda_{n}}\right\Vert _{L_{e}^{2}}=1$ and $\tilde{v}_{\lambda_{n}}$
satisfies
\begin{equation}
\mathcal{A}^{\lambda_{n}}\tilde{v}_{\lambda_{n}}=\frac{1}{\left\Vert \tilde
{v}_{\lambda_{n}}\right\Vert _{L_{e}^{2}}}\left(  \frac{k_{\lambda_{n}}%
}{\lambda_{n}}u_{\lambda_{n}}-c_{\lambda_{n}}\frac{\mathcal{A}^{\lambda_{n}%
}-\mathcal{A}^{0}}{\lambda_{n}}\psi_{ex}\left(  \xi\right)  \right)  .
\label{eqn-vn-tilt}%
\end{equation}
Denote%
\[
g_{n}\left(  \xi\right)  =\frac{\mathcal{A}^{\lambda_{n}}-\mathcal{A}^{0}%
}{\lambda_{n}}\psi_{ex}\left(  \xi\right)  =b\left(  \xi\right)
w_{\lambda_{n}}\left(  x\left(  \xi\right)  \right)  ,
\]
where
\begin{align}
w_{\lambda}\left(  x\right)   &  =\frac{\mathcal{A}_{e}^{\lambda}%
-\mathcal{A}_{e}^{0}}{\lambda}\psi_{ex}\left(  x\right) \label{exp-w-lb}\\
&  =-\left(  -\frac{1}{\lambda+\mathcal{D}}\frac{P_{ey}}{\psi_{ey}^{2}}%
-\frac{P_{ey}}{\psi_{ey}}\frac{1}{\lambda+\mathcal{D}}\frac{1}{\psi_{ey}%
}+\mathcal{E}^{\lambda,+}\frac{P_{ey}}{\psi_{ey}}\frac{1}{\lambda+\mathcal{D}%
}\frac{1}{\psi_{ey}}\right)  \psi_{ey}\eta_{ex}\nonumber\\
&  =\frac{1}{\lambda+\mathcal{D}}\frac{d}{dx}\left(  u_{e}\right)
+\frac{P_{ey}}{\psi_{ey}}\frac{1}{\lambda+\mathcal{D}}\eta_{ex}-\mathcal{E}%
^{\lambda,+}\frac{P_{ey}}{\psi_{ey}}\frac{1}{\lambda+\mathcal{D}}\eta
_{ex}\nonumber\\
&  =\left(  1-\mathcal{E}^{\lambda,+}\right)  \frac{u_{e}}{\psi_{ey}}%
+\frac{P_{ey}}{\psi_{ey}}\left(  1-\mathcal{E}^{\lambda,+}\right)  \frac
{\eta_{e}}{\psi_{ey}}-\mathcal{E}^{\lambda,+}\frac{P_{ey}}{\psi_{ey}}\left(
1-\mathcal{E}^{\lambda,+}\right)  \frac{\eta_{e}}{\psi_{ey}}.\nonumber
\end{align}
By Lemma \ref{lemma-e-lb}, $\left\Vert w_{\lambda}\right\Vert _{L^{2}%
(\mathcal{S}_{e})}\leq C$ (independent of $\lambda$), and moreover
\begin{equation}
w_{\lambda}\left(  x\right)  \rightarrow\frac{u_{e}}{\psi_{ey}}+\frac
{P_{ey}\eta_{e}}{\psi_{ey}^{2}}\text{ strongly in }L^{2}(\mathcal{S}%
_{e}),\ \text{when }\lambda\rightarrow0+\text{.} \label{limit-w-lb}%
\end{equation}
So $\left\Vert g_{n}\right\Vert _{L^{2}}\leq C\ $and thus by applying the
estimate (\ref{apriori-small-lb}) to (\ref{eqn-vn-tilt}), we have $\left\Vert
\tilde{v}_{\lambda_{n}}\right\Vert _{H^{\frac{1}{2}}}\leq C$. Therefore, as
before, $\tilde{v}_{\lambda_{n}}\rightarrow$ $\tilde{v}_{0}\neq0$ weakly in
$H^{\frac{1}{2}}$. Since $\frac{k_{\lambda_{n}}}{\lambda_{n}},\frac
{1}{\left\Vert \tilde{v}_{\lambda_{n}}\right\Vert _{L_{e}^{2}}}\rightarrow0$,
we have $\mathcal{A}^{0}\tilde{v}_{0}=0$. So $\tilde{v}_{0}=c_{1}\psi
_{ex}\left(  \xi\right)  $ for some $c_{1}\neq0$. But $\left(  \tilde
{v}_{\lambda_{n}},\psi_{ex}\left(  \xi\right)  \right)  =0$ implies that
$\left(  \tilde{v}_{0},\psi_{ex}\left(  \xi\right)  \right)  =0,\ $a
contradiction. This proves that $\left\Vert v_{\lambda}\right\Vert _{L_{e}%
^{2}}\leq C$. The equation satisfied by $v_{\lambda}$ is
\[
\mathcal{A}^{\lambda}v_{\lambda}=\frac{k_{\lambda}}{\lambda_{n}}u_{\lambda
}-c_{\lambda}\frac{\mathcal{A}^{\lambda}-\mathcal{A}^{0}}{\lambda}\psi
_{ex}\left(  \xi\right)  =\frac{k_{\lambda}}{\lambda_{n}}u_{\lambda
}-c_{\lambda}b\left(  \xi\right)  w_{\lambda}\left(  x\left(  \xi\right)
\right)  .
\]
Applying Lemma \ref{lemma-apriori} to the above equation, we have $\left\Vert
v_{\lambda}\right\Vert _{H^{\frac{1}{2}}}\leq C$ and thus $v_{\lambda
}\rightarrow$ $v_{0}$ weakly in $H^{\frac{1}{2}}$. By (\ref{limit-w-lb}),
$v_{0}$ satisfies
\[
\mathcal{A}^{0}v_{0}=-b\left(  \xi\right)  \left(  \frac{u_{e}}{\psi_{ey}%
}+\frac{P_{ey}\eta_{e}}{\psi_{ey}^{2}}\right)  \left(  \xi\right)  .
\]
It is shown in the Appendix that
\begin{equation}
\mathcal{A}_{e}^{0}\partial_{c}\bar{\psi}_{e}\left(  x\right)  =-\left(
\frac{u_{e}}{\psi_{ey}}+\frac{P_{ey}\eta_{e}}{\psi_{ey}^{2}}\right)  \left(
x\right)  , \label{eqn-d-phi-dc}%
\end{equation}
and equivalently,
\[
\mathcal{A}^{0}\partial_{c}\bar{\psi}_{e}\left(  \xi\right)  =-b\left(
\xi\right)  \left(  \frac{u_{e}}{\psi_{ey}}+\frac{P_{ey}\eta_{e}}{\psi
_{ey}^{2}}\right)  \left(  \xi\right)  ,
\]
where
\[
\bar{\psi}_{e}\left(  x,y\right)  =\psi_{e}\left(  x,y\right)  -cy.
\]
So $\mathcal{A}^{0}\left(  v_{0}-\partial_{c}\bar{\psi}_{e}\left(  \xi\right)
\right)  =0$. Since $\left(  v_{0},\psi_{ex}\left(  \xi\right)  \right)
=\lim_{\lambda\rightarrow0+}\left(  v_{\lambda},\psi_{ex}\left(  \xi\right)
\right)  =0,$ we have
\[
v_{0}=\partial_{c}\bar{\psi}_{e}\left(  \xi\right)  +d_{0}\psi_{ex}\left(
\xi\right)  ,\ d_{0}=-\left(  \partial_{c}\bar{\psi}_{e},\psi_{ex}\right)
\ /\left\Vert \psi_{ex}\right\Vert _{L^{2}}^{2}.
\]
By the same argument as in the proof of $\left\Vert u_{\lambda}-u_{0}%
\right\Vert _{H^{\frac{1}{2}}}\rightarrow0$, we have $\left\Vert v_{\lambda
}-v_{0}\right\Vert _{H^{\frac{1}{2}}}\rightarrow0.$ We rewrite
\[
u_{\lambda}=c_{\lambda}\psi_{ex}+\lambda v_{\lambda}=\bar{c}_{\lambda}%
\psi_{ex}+\lambda\bar{v}_{\lambda},
\]
where $\bar{c}_{\lambda}=c_{\lambda}+\lambda d_{0}$, $\bar{v}_{\lambda
}=v_{\lambda}-d_{0}\psi_{ex}$. Then $\bar{c}_{\lambda}\rightarrow1$, $\bar
{v}_{\lambda}\rightarrow\partial_{c}\bar{\psi}_{e}\left(  \xi\right)  \ $when
$\lambda\rightarrow0+$.

Now we compute $\lim_{\lambda\rightarrow0+}\frac{k_{\lambda}}{\lambda^{2}}$.
From (\ref{eqn-1st}), we have
\[
\mathcal{A}^{0}\frac{u_{\lambda}}{\lambda^{2}}+\frac{\mathcal{A}^{\lambda
}-\mathcal{A}^{0}}{\lambda}\left(  \frac{\bar{c}_{\lambda}}{\lambda}\psi
_{ex}+\bar{v}_{\lambda}\right)  =\frac{k_{\lambda}}{\lambda^{2}}u_{\lambda}.
\]
Taking the inner product of above with $\psi_{ex}\left(  \xi\right)  $, we
have%
\begin{align*}
\frac{k_{\lambda}}{\lambda^{2}}\left(  u_{\lambda},\psi_{ex}\left(
\xi\right)  \right)   &  =\bar{c}_{\lambda}\left(  \frac{\mathcal{A}^{\lambda
}-\mathcal{A}^{0}}{\lambda^{2}}\psi_{ex},\psi_{ex}\right)  +\left(
\frac{\mathcal{A}^{\lambda}-\mathcal{A}^{0}}{\lambda}\bar{v}_{\lambda}%
,\psi_{ex}\right) \\
&  =\bar{c}_{\lambda}I_{1}+I_{2}.
\end{align*}
Again, we do the computations in the physical space. For the first term, we
use (\ref{exp-w-lb}) to get
\begin{align*}
I_{1}  &  =\left\langle \frac{\mathcal{A}_{e}^{\lambda}-\mathcal{A}_{e}^{0}%
}{\lambda^{2}}\psi_{ex}\left(  x\right)  ,\psi_{ex}\left(  x\right)
\right\rangle =\left\langle \frac{w_{\lambda}\left(  x\right)  }{\lambda}%
,\psi_{ex}\left(  x\right)  \right\rangle \\
&  =\left\langle \frac{\mathcal{D}}{\left(  \lambda+\mathcal{D}\right)
\lambda}\frac{u_{e}}{\psi_{ey}},\psi_{ex}\left(  x\right)  \right\rangle
+\left\langle \frac{P_{ey}}{\psi_{ey}}\frac{\mathcal{D}}{\left(
\lambda+\mathcal{D}\right)  \lambda}\frac{\eta_{e}}{\psi_{ey}},\psi
_{ex}\left(  x\right)  \right\rangle \\
&  \ \ \ \ \ -\left\langle \frac{1}{\lambda+\mathcal{D}}\frac{P_{ey}}%
{\psi_{ey}}\left(  1-\mathcal{E}^{\lambda,+}\right)  \frac{\eta_{e}}{\psi
_{ey}},\psi_{ex}\left(  x\right)  \right\rangle \\
&  =I_{1}^{1}+I_{1}^{2}+I_{1}^{3}\text{.}%
\end{align*}
We have
\begin{align*}
I_{1}^{1}  &  =-\left\langle \left(  \frac{1}{\lambda}-\frac{1}{\lambda
+\mathcal{D}}\right)  \frac{u_{e}}{\psi_{ey}},\psi_{ey}\eta_{ex}\right\rangle
\\
&  =-\frac{1}{\lambda}\left\langle u_{e},\eta_{ex}\right\rangle +\left\langle
\frac{1}{\lambda+\mathcal{D}}\frac{u_{e}}{\psi_{ey}},\psi_{ey}\eta
_{ex}\right\rangle =\left\langle \frac{1}{\lambda+\mathcal{D}}\frac{u_{e}%
}{\psi_{ey}},\psi_{ey}\eta_{ex}\right\rangle \text{ }\\
&  =\left\langle u_{e},\frac{1}{\lambda-\mathcal{D}}\eta_{ex}\right\rangle
=\left\langle u_{e},\left(  \mathcal{E}^{\lambda,-}-1\right)  \frac{\eta_{e}%
}{\psi_{ey}}\right\rangle \rightarrow-\left\langle u_{e},\frac{\eta_{e}}%
{\psi_{ey}}\right\rangle ,
\end{align*}%
\begin{align*}
I_{1}^{2}  &  =-\left\langle \frac{P_{ey}}{\psi_{ey}}\left(  \frac{1}{\lambda
}-\frac{1}{\lambda+\mathcal{D}}\right)  \frac{\eta_{e}}{\psi_{ey}},\psi
_{ey}\eta_{ex}\right\rangle =\left\langle P_{ey}\frac{1}{\lambda+\mathcal{D}%
}\frac{\eta_{e}}{\psi_{ey}},\eta_{ex}\right\rangle \\
&  =\left\langle \frac{1}{\lambda+\mathcal{D}}\frac{\eta_{e}}{\psi_{ey}}%
,\psi_{ey}\frac{d}{dx}\left(  u_{e}\right)  \right\rangle =\left\langle
\eta_{e},\frac{1}{\lambda-\mathcal{D}}\frac{d}{dx}\left(  u_{e}\right)
\right\rangle =\left\langle \eta_{e},\left(  \mathcal{E}^{\lambda,-}-1\right)
\frac{u_{e}}{\psi_{ey}}\right\rangle \\
&  \rightarrow-\left\langle \eta_{e},\frac{u_{e}}{\psi_{ey}}\right\rangle ,
\end{align*}
and
\begin{align*}
I_{1}^{3}  &  =\left\langle \frac{1}{\lambda+\mathcal{D}}\frac{P_{ey}}%
{\psi_{ey}}\left(  1-\mathcal{E}^{\lambda,+}\right)  \frac{\eta_{e}}{\psi
_{ey}},\psi_{ey}\eta_{ex}\right\rangle =\left\langle P_{ey}\left(
1-\mathcal{E}^{\lambda,+}\right)  \frac{\eta_{e}}{\psi_{ey}},\frac{1}%
{\lambda-\mathcal{D}}\eta_{ex}\right\rangle \\
&  =\left\langle P_{ey}\left(  1-\mathcal{E}^{\lambda,+}\right)  \frac
{\eta_{e}}{\psi_{ey}},\left(  \mathcal{E}^{\lambda,-}-1\right)  \frac{\eta
_{e}}{\psi_{ey}}\right\rangle \rightarrow-\left\langle P_{ey}\frac{\eta_{e}%
}{\psi_{ey}},\frac{\eta_{e}}{\psi_{ey}}\right\rangle .
\end{align*}
So
\[
\lim_{\lambda\rightarrow0+}I_{1}\left(  \lambda\right)  =-\left\langle
\eta_{e},\frac{2u_{e}}{\psi_{ey}}+\frac{P_{ey}}{\psi_{ey}^{2}}\eta
_{e}\right\rangle .
\]
To compute $I_{2}$, we write
\begin{align*}
I_{2}  &  =\left\langle \frac{1}{\lambda+\mathcal{D}}\frac{P_{ey}}{\psi
_{ey}^{2}}\bar{v}_{\lambda},\psi_{ey}\eta_{ex}\right\rangle +\left\langle
\frac{P_{ey}}{\psi_{ey}}\frac{1}{\lambda+\mathcal{D}}\frac{1}{\psi_{ey}}%
\bar{v}_{\lambda},\psi_{ey}\eta_{ex}\right\rangle \\
&  \ \ \ \ \ -\left\langle \frac{1}{\lambda+\mathcal{D}}\frac{P_{ey}}%
{\psi_{ey}}\mathcal{E}^{\lambda,+}\frac{1}{\psi_{ey}}\bar{v}_{\lambda}%
,\psi_{ey}\eta_{ex}\right\rangle \\
&  =I_{2}^{1}+I_{2}^{2}+I_{2}^{3}\text{. }%
\end{align*}
We have%
\[
I_{2}^{1}=\left\langle \frac{P_{ey}}{\psi_{ey}}\bar{v}_{\lambda},\frac
{1}{\lambda-\mathcal{D}}\eta_{ex}\right\rangle =\left\langle \frac{P_{ey}%
}{\psi_{ey}}\bar{v}_{\lambda},\left(  \mathcal{E}^{\lambda,-}-1\right)
\frac{\eta_{e}}{\psi_{ey}}\right\rangle \rightarrow-\left\langle \frac{P_{ey}%
}{\psi_{ey}}\partial_{c}\bar{\psi}_{e},\frac{\eta_{e}}{\psi_{ey}}\right\rangle
,
\]%
\begin{align*}
I_{2}^{2}  &  =\left\langle \frac{1}{\lambda+\mathcal{D}}\frac{1}{\psi_{ey}%
}\bar{v}_{\lambda},\psi_{ey}\frac{d}{dx}\left(  u_{e}\right)  \right\rangle
=\left\langle \bar{v}_{\lambda},\frac{1}{\lambda-\mathcal{D}}\frac{d}%
{dx}\left(  u_{e}\right)  \right\rangle \\
&  =\left\langle \bar{v}_{\lambda},\left(  -1+\mathcal{E}^{\lambda,-}\right)
\frac{u_{e}}{\psi_{ey}}\right\rangle \rightarrow-\left\langle \partial_{c}%
\bar{\psi}_{e},\frac{u_{e}}{\psi_{ey}}\right\rangle
\end{align*}
and
\begin{align*}
I_{2}^{3}  &  =-\left\langle \frac{1}{\lambda+\mathcal{D}}\frac{P_{ey}}%
{\psi_{ey}}\mathcal{E}^{\lambda,+}\frac{1}{\psi_{ey}}\bar{v}_{\lambda}%
,\psi_{ey}\eta_{ex}\right\rangle =-\left\langle P_{ey}\mathcal{E}^{\lambda
,+}\frac{1}{\psi_{ey}}\bar{v}_{\lambda},\frac{1}{\lambda-\mathcal{D}}\eta
_{ex}\right\rangle \\
&  =-\left\langle P_{ey}\mathcal{E}^{\lambda,+}\frac{1}{\psi_{ey}}\bar
{v}_{\lambda},\left(  -1+\mathcal{E}^{\lambda,-}\right)  \frac{\eta_{e}}%
{\psi_{ey}}\right\rangle \rightarrow0.
\end{align*}
So
\[
\lim_{\lambda\rightarrow0+}I_{2}\left(  \lambda\right)  =-\left\langle
\partial_{c}\bar{\psi}_{e},\frac{u_{e}}{\psi_{ey}}+\frac{P_{ey}}{\psi_{ey}%
^{2}}\eta_{e}\right\rangle .
\]
Thus
\begin{align*}
\lim_{\lambda\rightarrow0+}\frac{k_{\lambda}}{\lambda^{2}}  &  =\lim
_{\lambda\rightarrow0+}\frac{\bar{c}_{\lambda}I_{1}+I_{2}}{\left(  u_{\lambda
},\psi_{ex}\right)  }\\
&  =\frac{-\left\langle \eta_{e},\frac{2u_{e}}{\psi_{ey}}+\frac{P_{ey}}%
{\psi_{ey}^{2}}\eta_{e}\right\rangle -\left\langle \partial_{c}\bar{\psi}%
_{e},\frac{u_{e}}{\psi_{ey}}+\frac{P_{ey}}{\psi_{ey}^{2}}\eta_{e}\right\rangle
}{\left\Vert \psi_{ex}\right\Vert _{L^{2}}^{2}}.
\end{align*}
It is shown in the Appendix that
\begin{equation}
\frac{dP}{dc}=-\left\langle \eta_{e},2\frac{u_{e}}{\psi_{ey}}+\frac{P_{ey}%
}{\psi_{ey}^{2}}\eta_{e}\right\rangle -\left\langle \partial_{c}\bar{\psi}%
_{e},\frac{u_{e}}{\psi_{ey}}+\frac{P_{ey}}{\psi_{ey}^{2}}\eta_{e}\right\rangle
, \label{formula-dp-dc}%
\end{equation}
where the momentum $P$ is defined by
\[
P=\int_{\mathcal{S}_{e}}\eta_{e}\frac{d}{dx}\left(  \bar{\phi}_{e}\left(
x,\eta_{e}\left(  x\right)  \right)  \right)  dx,
\]
with $\bar{\phi}_{e}=\phi_{e}-cx$. It is shown in \cite{benjamin73} (see also
\cite{long-soli-74}) that for a solitary wave solution
\[
\frac{dE}{dc}=-c\frac{dP}{dc},
\]
where we note that the travel direction considered in this paper is opposite
to the one in the above references. A combination of above results yields the
formula (\ref{formula-moving}).
\end{proof}

As a corollary of the above proof, we show Theorem \ref{thm-transition}.

\begin{proof}
[Proof of Theorem \ref{thm-transition}]The proof is very similar to that of
Lemma \ref{lemma-moving}, so we only sketch it. The main difference is that
the computations depend on the parameter $\mu_{n}$. We use $\eta_{e,n}%
,\psi_{ey,n}$ etc. to denote the dependence on $\mu_{n}$, and $\eta_{e}%
,\psi_{ey}$ etc. for quantities depending on $\mu_{0}$.$\ $Denote
$u_{n}\left(  \xi\right)  =\mathcal{B}_{n}\left(  \psi_{n}|_{\mathcal{S}%
_{e,n}}\right)  $. Then $\mathcal{A}_{n}^{\lambda_{n}}u_{n}=0$ and we
normalize $u_{n}\ $by$\ \left\Vert u_{n}\right\Vert _{L_{e_{n}}^{2}}=1$.
First, we show that $\mathcal{\tilde{E}}_{n}^{\lambda_{n},\pm}\rightarrow0$
strongly in $L^{2}\left(  \mathbf{R}\right)  $. Indeed, for any $v\in
L^{2}\left(  \mathbf{R}\right)  $, since the norms $\left\Vert \cdot
\right\Vert _{L^{2}}$ and $\left\Vert \cdot\right\Vert _{L_{b_{n}\psi_{ey,n}%
}^{2}}$ are equivalent, we have
\begin{align*}
\left\Vert \mathcal{\tilde{E}}_{n}^{\lambda_{n},\pm}v\right\Vert _{L^{2}}^{2}
&  \leq C\left\Vert \mathcal{\tilde{E}}_{n}^{\lambda_{n},\pm}v\right\Vert
_{L_{b_{n}\psi_{ey,n}}^{2}}^{2}=C\int\frac{\lambda_{n}^{2}}{\lambda_{n}%
^{2}+\alpha^{2}}d\Vert\tilde{M}_{\alpha,n}v\Vert_{L_{b_{n}\psi_{ey,n}}^{2}%
}^{2}\\
&  \leq C\frac{\lambda_{n}^{2}}{\delta^{2}}\left\Vert v\right\Vert
_{L_{b_{n}\psi_{ey,n}}^{2}}^{2}+C\int_{\left\vert \alpha\right\vert \leq
\delta}d\Vert\tilde{M}_{\alpha,n}v\Vert_{L_{b_{n}\psi_{ey,n}}^{2}}^{2}%
=A_{1}+A_{2},
\end{align*}
where $\delta>0$ is arbitrary and $\left\{  \tilde{M}_{\alpha,n};\alpha
\in\mathbf{R}^{1}\right\}  $ is the spectral measure of the self-adjoint
operator $\tilde{R}_{n}=-i\mathcal{\tilde{D}}_{n}$ on $L_{b_{n}\psi_{ey,n}%
}^{2}$. For any $\varepsilon>0$, since $M_{\{0\}}=0$ we can choose $\delta$
small enough such that
\[
\int_{\left\vert \alpha\right\vert \leq\delta}d\Vert\tilde{M}_{\alpha}%
v\Vert_{L_{b\psi_{ey}}^{2}}^{2}\leq\frac{\varepsilon}{2}.
\]
Since the measure $d\Vert\tilde{M}_{\alpha,n}v\Vert_{L_{b_{n}\psi_{ey,n}}^{2}%
}^{2}$ converges to $d\Vert\tilde{M}_{\alpha}v\Vert_{L_{b\psi_{ey}}^{2}}^{2}$,
we have $A_{2}\leq\varepsilon\ $when $n$ is big enough. The term $A_{1}$ tends
to zero when $n\rightarrow\infty.$ Because $\varepsilon$ can be arbitrarily
small, we have
\[
\lim_{n\rightarrow\infty}\left\Vert \mathcal{\tilde{E}}_{n}^{\lambda_{n},\pm
}v\right\Vert _{L^{2}}^{2}=0.
\]
So $\mathcal{A}_{n}^{\lambda_{n}}\rightarrow\mathcal{A}^{0}$ strongly and
similar to the proof of Lemma \ref{lemma-moving}, we have $u_{n}%
\rightarrow\psi_{ex}\left(  \xi\right)  $ in $H^{\frac{1}{2}}\left(
\mathbf{R}\right)  $ by a renormalization. We write $u_{n}=c_{n}\psi
_{ex,n}\left(  \xi\right)  +\lambda_{n}v_{n},$with $c_{n}=\left(  u_{n}%
,\psi_{ex,n}\right)  /\left(  \psi_{ex,n},\psi_{ex,n}\right)  $. Then
$c_{n}\rightarrow1$ and $\left(  v_{n},\psi_{ex,n}\right)  =0$. As before, we
have $v_{n}\rightarrow v_{0}\ $in $H^{\frac{1}{2}}\left(  \mathbf{R}\right)
$, where
\[
v_{0}=\partial_{c}\bar{\psi}_{e}\left(  \xi\right)  +d_{0}\psi_{ex}\left(
\xi\right)  ,\ d_{0}=-\left(  \partial_{c}\bar{\psi}_{e},\psi_{ex}\right)
\ /\left(  \psi_{ex},\psi_{ex}\right)  .
\]
We rewrite
\[
u_{n}=c_{n}\psi_{ex,n}\left(  \xi\right)  +\lambda_{n}v_{n}=\bar{c}_{n}%
\psi_{ex,n}+\lambda_{n}\bar{v}_{n},
\]
where $\bar{c}_{n}=c_{n}+\lambda_{n}d_{0}$ and $\bar{v}_{n}=v_{n}-d_{0}%
\psi_{ex,n}$. Then $\bar{c}_{n}\rightarrow1$, $\bar{v}_{n}\rightarrow
\partial_{c}\bar{\psi}_{e}\left(  \xi\right)  \ $when $n\rightarrow\infty$.
Similarly, we have
\[
0=\bar{c}_{n}\left(  \frac{\mathcal{A}_{n}^{\lambda_{n}}-\mathcal{A}_{n}^{0}%
}{\lambda_{n}^{2}}\psi_{ex,n},\psi_{ex,n}\right)  +\left(  \frac
{\mathcal{A}_{n}^{\lambda_{n}}-\mathcal{A}_{n}^{0}}{\lambda_{n}}\bar{v}%
_{n},\psi_{ex,n}\right)  =\bar{c}_{n}I_{1}+I_{2}.
\]
As before, the calculations of $I_{1},I_{2}$ are first done in the physical
space $\mathcal{S}_{e,n}$ with the inner product $\left\langle
.,.\right\rangle _{n}$. We use the same notations as in the proof of Lemma
\ref{lemma-moving}. By the same computations, we have
\begin{align*}
I_{1}^{1}  &  =\left\langle u_{e,n},\left(  -1+\mathcal{E}_{n}^{\lambda_{n}%
,-}\right)  \frac{\eta_{e,n}}{\psi_{ey,n}}\right\rangle _{n}=\left(
b_{n}u_{e,n}\left(  \xi\right)  ,\left(  -1+\mathcal{\tilde{E}}_{n}%
^{\lambda_{n},-}\right)  \frac{\eta_{e,n}}{\psi_{ey,n}}\left(  \xi\right)
\right) \\
&  \rightarrow-\left(  bu_{e}\left(  \xi\right)  ,\frac{\eta_{e}}{\psi_{ey}%
}\left(  \xi\right)  \right)  =-\left\langle u_{e},\frac{\eta_{e}}{\psi_{ey}%
}\right\rangle .
\end{align*}
The other terms are handled similarly. So finally we have
\begin{align*}
0  &  =\lim_{n\rightarrow\infty}\bar{c}_{n}I_{1}+I_{2}=-\left\langle \eta
_{e},2\frac{u_{e}}{\psi_{ey}}+\frac{P_{ey}}{\psi_{ey}^{2}}\eta_{e}%
\right\rangle -\left\langle \partial_{c}\bar{\psi}_{e},\frac{u_{e}}{\psi_{ey}%
}+\frac{P_{ey}}{\psi_{ey}^{2}}\eta_{e}\right\rangle \\
&  =-\frac{1}{c}\frac{dE}{dc}|_{\mu_{0}}.
\end{align*}
So $E^{\prime}\left(  \mu_{0}\right)  =0$. This finishes the proof of Theorem
\ref{thm-transition}.
\end{proof}

\section{Appendix}

In this Appendix, we prove (\ref{eqn-d-phi-dc}) and (\ref{formula-dp-dc}).

\begin{proof}
[Proof of (\ref{eqn-d-phi-dc})]We derive (\ref{eqn-d-phi-dc}) from the
linearized system (\ref{eqn-L}) to avoid working on the parameter dependent
fluid domains and wave profiles. Note that (\ref{eqn-L}) describes the
evolution of the first order variations of the wave profile and the stream
function in the travelling frame of the basic wave. The basic wave is $\left(
\eta_{e}(x;c),\bar{\psi}_{e}(x,y;c)\right)  $ in its travelling frame $\left(
x+ct,y,t\right)  $. Here, the stream function $\bar{\psi}_{e}\ $and the
relative stream function$\ \psi_{e}$ are related by $\bar{\psi}_{e}=\psi
_{e}-cy,\ $and thus $\bar{\psi}_{e}\rightarrow0$ when $\left\vert x\right\vert
\rightarrow\infty$. As an example to illustrate the ideas, we first consider a
perturbed solution with a trivial translation $\left(  \eta_{e}(x+\delta
x;c),\bar{\psi}_{e}(x+\delta x,y;c)\right)  $. The first order variations
$\delta x\left(  \eta_{ex},\bar{\psi}_{ex}\right)  $ and $\delta xP_{ex}$
satisfy the linearized system (\ref{eqn-L}), so we have
\begin{subequations}
\begin{gather*}
\Delta\bar{\psi}_{ex}=0\qquad in\quad\mathcal{D}_{e},\\
\frac{d}{dx}(\psi_{ey}\eta_{ex})+\frac{d}{dx}\bar{\psi}_{ex}=0\qquad
\text{on}\quad\mathcal{S}_{e};\\
P_{ex}+P_{ey}\eta_{ex}=0\qquad\text{on}\quad\mathcal{S}_{e};\\
\frac{d}{dx}(\psi_{ey}\partial_{n}\bar{\psi}_{ex})+\frac{d}{dx}P_{ex}%
=0\qquad\text{on}\quad\mathcal{S}_{e};\\
\bar{\psi}_{ex}=0\qquad\text{on}\quad\{y=-h\},
\end{gather*}
from which we get $\mathcal{A}_{e}^{0}\bar{\psi}_{ex}\left(  x\right)  =0$.
Since $\bar{\psi}_{ex}=\psi_{ex}$, this recovers $\mathcal{A}_{e}^{0}\psi
_{ex}=0$ which is proved in Lemma \ref{lemma-A0-property}. The solitary wave
with the speed $c+\delta c$ is given by
\end{subequations}
\[
\left(  \eta_{e}(x+\delta ct;c+\delta c),\bar{\psi}_{e}(x+\delta ct,y;c+\delta
c)\right)
\]
in the $\left(  x+ct,y,t\right)  $ frame. The first order variation are
\begin{align*}
&  \left(  \eta_{e}(x+\delta ct;c+\delta c),\bar{\psi}_{e}(x+\delta
ct,y;c+\delta c)\right)  -\left(  \eta_{e}(x;c),\bar{\psi}_{e}(x,y;c)\right)
\\
&  =\delta c\left(  \eta_{ex}(x;c)t+\partial_{c}\eta_{e}(x;c),\bar{\psi}%
_{ex}(x,y;c)t+\partial_{c}\bar{\psi}_{e}(x,y;c)\right)  .
\end{align*}
So
\[
\left(  \eta_{ex}(x;c)t+\partial_{c}\eta_{e}(x;c),\bar{\psi}_{ex}%
(x,y;c)t+\partial_{c}\bar{\psi}_{e}(x,y;c)\right)
\]
and $P_{ex}(x,y;c)t+\partial_{c}P_{e}(x,y;c)\ $satisfy the linearized system
(\ref{eqn-L}). By using the above linear system for $\left(  \eta_{ex}%
,\bar{\psi}_{ex}\right)  $, we get%
\[
\Delta\partial_{c}\bar{\psi}_{e}=0\qquad in\quad\mathcal{D}_{e},
\]%
\begin{equation}
\eta_{ex}+\frac{d}{dx}(\psi_{ey}\partial_{c}\eta_{e})+\frac{d}{dx}\left(
\partial_{c}\bar{\psi}_{e}\right)  =0\qquad\text{on}\quad\mathcal{S}_{e};
\label{eqn-eta-c}%
\end{equation}%
\begin{equation}
\partial_{c}P_{e}+P_{ey}\partial_{c}\eta_{e}=0\qquad\text{on}\quad
\mathcal{S}_{e}; \label{eqn-P-c}%
\end{equation}%
\begin{equation}
\partial_{n}\bar{\psi}_{ex}+\frac{d}{dx}(\psi_{ey}\partial_{n}\left(
\partial_{c}\bar{\psi}_{e}\right)  )+\frac{d}{dx}\left(  \partial_{c}%
P_{e}\right)  =0\qquad\text{on}\quad\mathcal{S}_{e}; \label{eqn-phi-n-c}%
\end{equation}%
\[
\partial_{c}\bar{\psi}_{e}=0\qquad\text{on}\quad\{y=-h\}.
\]
From (\ref{eqn-eta-c}),
\begin{equation}
\partial_{c}\eta_{e}=-\frac{1}{\psi_{ey}}\left(  \partial_{c}\bar{\psi}%
_{e}+\eta_{e}\right)  . \label{intern-dc-eta}%
\end{equation}
Since $\partial_{n}\bar{\psi}_{ex}=\frac{d}{dx}\left(  \bar{\phi}_{ex}\right)
=\frac{d}{dx}\left(  u_{e}\right)  $ and $u_{e}\rightarrow0$ when $\left\vert
x\right\vert \rightarrow\infty,\ $combining (\ref{intern-dc-eta}) with
(\ref{eqn-phi-n-c}) and (\ref{eqn-P-c}), we have
\begin{equation}
\partial_{n}\left(  \partial_{c}\bar{\psi}_{e}\right)  =-\frac{1}{\psi_{ey}%
}\left(  u_{e}-P_{ey}\partial_{c}\eta_{e}\right)  =-\frac{u_{e}}{\psi_{ey}%
}-\frac{P_{ey}\eta_{e}}{\psi_{ey}^{2}}-\frac{P_{ey}}{\psi_{ey}^{2}}%
\partial_{c}\bar{\psi}_{e}. \label{interm-dndc-phi}%
\end{equation}
So
\[
\partial_{n}\left(  \partial_{c}\bar{\psi}_{e}\right)  +\frac{P_{ey}}%
{\psi_{ey}^{2}}\partial_{c}\bar{\psi}_{e}=-\frac{u_{e}}{\psi_{ey}}%
-\frac{P_{ey}\eta_{e}}{\psi_{ey}^{2}},
\]
that is,
\[
\mathcal{A}_{e}^{0}\partial_{c}\bar{\psi}_{e}=-\frac{u_{e}}{\psi_{ey}}%
-\frac{P_{ey}\eta_{e}}{\psi_{ey}^{2}}.
\]
Lastly, we show that at a turning point $\mu_{0},$ $\mathcal{A}_{e}%
^{0}\partial_{\mu}\bar{\psi}_{e}=0$. Indeed, the first order variation of
\[
\left(  \eta_{e}(x+\delta ct;\mu_{0}+\delta\mu),\bar{\psi}_{e}(x+\delta
ct,y;\mu_{0}+\delta\mu)\right)  -\left(  \eta_{e}(x;\mu_{0}),\bar{\psi}%
_{e}(x,y;\mu_{0})\right)
\]
is $\delta\mu\left(  \partial_{\mu}\eta_{e},\partial_{\mu}\bar{\psi}%
_{e}\right)  $, since $\delta c=O\left(  \left\vert \delta\mu\right\vert
^{2}\right)  $ is of higher order. So $\left(  \partial_{\mu}\eta_{e}%
,\partial_{\mu}\bar{\psi}_{e}\right)  $ satisfies the same system for $\left(
\eta_{ex},\bar{\psi}_{ex}\right)  $, which yields $\mathcal{A}_{e}^{0}%
\partial_{\mu}\bar{\psi}_{e}=0$.
\end{proof}

\begin{proof}
[Proof of (\ref{formula-dp-dc})]We have
\[
P\left(  c\right)  =\int_{\mathcal{S}_{e}}\eta_{e}\frac{d}{dx}\left(
\bar{\phi}_{e}\right)  dx=\int\eta_{e}\left(  x;c\right)  \partial_{n}%
\bar{\psi}_{e}\left(  \left(  x,\eta_{e}\left(  x;c\right)  ;c\right)
\right)  dx
\]
Since
\begin{align*}
\ \ \  &  \ \ \ \ \ \frac{d}{dc}\left(  \partial_{n}\bar{\psi}_{e}\left(
x,\eta_{e}\left(  x;c\right)  ;c\right)  \right) \\
&  =\frac{d}{dc}\left(  \bar{\psi}_{ey}\left(  x,\eta_{e}\left(  x;c\right)
;c\right)  -\eta_{ex}\left(  x;c\right)  \bar{\psi}_{ex}\left(  x,\eta
_{e}\left(  x;c\right)  ;c\right)  \right) \\
&  =\partial_{y}\left(  \partial_{c}\bar{\psi}_{e}\right)  +\bar{\psi}%
_{eyy}\partial_{c}\eta_{e}-\partial_{x}\left(  \partial_{c}\eta_{e}\right)
\bar{\psi}_{ex}-\eta_{ex}\left(  \partial_{x}\left(  \partial_{c}\bar{\psi
}_{e}\right)  +\bar{\psi}_{exy}\partial_{c}\eta_{e}\right) \\
&  =\left(  \partial_{y}\left(  \partial_{c}\bar{\psi}_{e}\right)  -\eta
_{ex}\partial_{x}\left(  \partial_{c}\bar{\psi}_{e}\right)  \right)
-\partial_{c}\eta_{e}\left(  \bar{\psi}_{exx}+\eta_{ex}\bar{\psi}%
_{exy}\right)  -\partial_{x}\left(  \partial_{c}\eta_{e}\right)  \bar{\psi
}_{ex}\\
&  =\partial_{n}\left(  \partial_{c}\bar{\psi}_{e}\right)  -\partial_{c}%
\eta_{e}\frac{d}{dx}\left(  \bar{\psi}_{ex}\right)  -\frac{d}{dx}\left(
\partial_{c}\eta_{e}\right)  \bar{\psi}_{ex}=\partial_{n}\left(  \partial
_{c}\bar{\psi}_{e}\right)  -\frac{d}{dx}\left(  \partial_{c}\eta_{e}\bar{\psi
}_{ex}\right)  ,
\end{align*}
we have
\begin{align*}
\frac{dP}{dc}  &  =\int\left\{  \partial_{c}\eta_{e}\partial_{n}\bar{\psi}%
_{e}+\eta_{e}\frac{d}{dc}\left(  \partial_{n}\bar{\psi}_{e}\right)  \right\}
dx\\
&  =\int\left\{  \partial_{c}\eta_{e}\partial_{n}\bar{\psi}_{e}+\eta
_{e}\left(  \partial_{n}\left(  \partial_{c}\bar{\psi}_{e}\right)  -\frac
{d}{dx}\left(  \partial_{c}\eta_{e}\bar{\psi}_{ex}\right)  \right)  \right\}
dx\\
&  =\int\left\{  (\partial_{c}\eta_{e}\left(  u_{e}-\eta_{ex}\bar{\psi}%
_{ex}\right)  +\eta_{e}\partial_{n}\left(  \partial_{c}\bar{\psi}_{e}\right)
+\eta_{ex}\partial_{c}\eta_{e}\bar{\psi}_{ex}\right\}  dx\text{ }\\
&  =\int\left\{  u_{e}\partial_{c}\eta_{e}+\eta_{e}\partial_{n}\left(
\partial_{c}\bar{\psi}_{e}\right)  \right\}  dx\\
&  =\left\langle u_{e},-\frac{1}{\psi_{ey}}\left(  \partial_{c}\bar{\psi}%
_{e}+\eta_{e}\right)  \right\rangle +\left\langle \eta_{e},-\frac{u_{e}}%
{\psi_{ey}}-\frac{P_{ey}\eta_{e}}{\psi_{ey}^{2}}-\frac{P_{ey}}{\psi_{ey}^{2}%
}\partial_{c}\bar{\psi}_{e}\right\rangle \\
&  \text{(by equations (\ref{intern-dc-eta}) and (\ref{interm-dndc-phi}))}\\
&  =-\left\langle \eta_{e},\frac{2 u_{e}}{\psi_{ey}}+\frac{P_{ey}}{\psi
_{ey}^{2}}\eta_{e}\right\rangle -\left\langle \partial_{c}\bar{\psi}_{e}%
,\frac{u_{e}}{\psi_{ey}}+\frac{P_{ey}}{\psi_{ey}^{2}}\eta_{e}\right\rangle .
\end{align*}

\end{proof}

\begin{center}
{\Large Acknowledgement}
\end{center}

This work is supported partly by the NSF grants DMS-0505460 and DMS-0707397.
The author thanks Sijue Wu for many helpful discussions.

\end{document}